\documentclass[10pt,british,english]{article}
\usepackage[T1]{fontenc}
\usepackage[latin9]{inputenc}
\usepackage{geometry}
\usepackage{color}
\usepackage{babel}
\usepackage{mathrsfs}
\usepackage{amsmath}
\usepackage{amsthm}
\usepackage{amssymb}
\usepackage{stmaryrd}
\usepackage{graphicx}
\usepackage[unicode=true,pdfusetitle,
 bookmarks=true,bookmarksnumbered=false,bookmarksopen=false,
 breaklinks=false,pdfborder={0 0 1},backref=false,colorlinks=false]
 {hyperref}

\makeatletter
\numberwithin{equation}{section}
\theoremstyle{plain}
\newtheorem{thm}{\protect\theoremname}[section]
\theoremstyle{plain}
\newtheorem{cor}[thm]{\protect\corollaryname}
\theoremstyle{remark}
\newtheorem*{acknowledgement*}{\protect\acknowledgementname}
\theoremstyle{definition}
\newtheorem{condition}[thm]{\protect\conditionname}
\theoremstyle{plain}
\newtheorem{prop}[thm]{\protect\propositionname}
\theoremstyle{plain}
\newtheorem{lem}[thm]{\protect\lemmaname}
\theoremstyle{definition}
\newtheorem{defn}[thm]{\protect\definitionname}
\theoremstyle{plain}
\newtheorem{fact}[thm]{\protect\factname}
\theoremstyle{remark}
\newtheorem{notation}[thm]{\protect\notationname}
\theoremstyle{remark}
\newtheorem{rem}[thm]{\protect\remarkname}
\theoremstyle{plain}
\newtheorem{assumption}[thm]{\protect\assumptionname}
\theoremstyle{definition}
\newtheorem{example}[thm]{\protect\examplename}
\theoremstyle{remark}
\newtheorem*{notation*}{\protect\notationname}

\@ifundefined{date}{}{\date{}}

\theoremstyle{plain}
\newtheorem{mythm}{\protect\theoremname}

\usepackage{amssymb}
\usepackage{lmodern}

\@ifundefined{showcaptionsetup}{}{%
 \PassOptionsToPackage{caption=false}{subfig}}
\usepackage{subfig}
\makeatother

\addto\captionsbritish{\renewcommand{\acknowledgementname}{Acknowledgement}}
\addto\captionsbritish{\renewcommand{\assumptionname}{Assumption}}
\addto\captionsbritish{\renewcommand{\conditionname}{Condition}}
\addto\captionsbritish{\renewcommand{\corollaryname}{Corollary}}
\addto\captionsbritish{\renewcommand{\definitionname}{Definition}}
\addto\captionsbritish{\renewcommand{\examplename}{Example}}
\addto\captionsbritish{\renewcommand{\factname}{Fact}}
\addto\captionsbritish{\renewcommand{\lemmaname}{Lemma}}
\addto\captionsbritish{\renewcommand{\notationname}{Notation}}
\addto\captionsbritish{\renewcommand{\propositionname}{Proposition}}
\addto\captionsbritish{\renewcommand{\remarkname}{Remark}}
\addto\captionsbritish{\renewcommand{\theoremname}{Theorem}}
\addto\captionsenglish{\renewcommand{\acknowledgementname}{Acknowledgement}}
\addto\captionsenglish{\renewcommand{\assumptionname}{Assumption}}
\addto\captionsenglish{\renewcommand{\conditionname}{Condition}}
\addto\captionsenglish{\renewcommand{\corollaryname}{Corollary}}
\addto\captionsenglish{\renewcommand{\definitionname}{Definition}}
\addto\captionsenglish{\renewcommand{\examplename}{Example}}
\addto\captionsenglish{\renewcommand{\factname}{Fact}}
\addto\captionsenglish{\renewcommand{\lemmaname}{Lemma}}
\addto\captionsenglish{\renewcommand{\notationname}{Notation}}
\addto\captionsenglish{\renewcommand{\propositionname}{Proposition}}
\addto\captionsenglish{\renewcommand{\remarkname}{Remark}}
\addto\captionsenglish{\renewcommand{\theoremname}{Theorem}}
\providecommand{\acknowledgementname}{Acknowledgement}
\providecommand{\assumptionname}{Assumption}
\providecommand{\conditionname}{Condition}
\providecommand{\corollaryname}{Corollary}
\providecommand{\definitionname}{Definition}
\providecommand{\examplename}{Example}
\providecommand{\factname}{Fact}
\providecommand{\lemmaname}{Lemma}
\providecommand{\notationname}{Notation}
\providecommand{\propositionname}{Proposition}
\providecommand{\remarkname}{Remark}
\providecommand{\theoremname}{Theorem}

\begin{document}
\global\long\def\RR{\mathbb{R}}%
\global\long\def\CC{\mathbb{C}}%
\global\long\def\HH{\mathbb{H}}%
\global\long\def\NN{\mathbb{N}}%
\global\long\def\ZZ{\mathbb{Z}}%
\global\long\def\QQ{\mathbb{Q}}%

\global\long\def\cA{\mathcal{A}}%

\global\long\def\l{\ell}%

\global\long\def\gam{\Gamma}%

\global\long\def\z{\zeta}%
\global\long\def\e{\epsilon}%
\global\long\def\a{\alpha}%
\global\long\def\b{\beta}%
\global\long\def\ga{\gamma}%
\global\long\def\ph{\varphi}%
\global\long\def\om{\omega}%
\global\long\def\lm{\lambda}%
\global\long\def\dl{\delta}%
\global\long\def\t{\theta}%
\global\long\def\s{\sigma}%
\global\long\def\little{\varepsilon}%

\global\long\def\gl#1{\operatorname{GL}_{#1}}%

\global\long\def\sl#1{\operatorname{SL}_{#1}}%

\global\long\def\so#1{\operatorname{SO}_{#1}}%

\global\long\def\ort#1{\operatorname{O}_{#1}}%

\global\long\def\pgl#1{\operatorname{PGL}_{#1}}%

\global\long\def\po#1{\operatorname{PO}_{#1}}%

\global\long\def\Lat{\Gamma}%
\global\long\def\disgrp{\gam}%
\global\long\def\wc{Q}%

\global\long\def\lra{\longrightarrow}%
\global\long\def\del{\partial}%
\global\long\def\half{\frac{1}{2}}%

\global\long\def\transpose{\operatorname{t}}%

\global\long\def\exd#1#2{\underset{#2}{\underbrace{#1}}}%
\global\long\def\exup#1#2{\overset{#2}{\overbrace{#1}}}%

\global\long\def\Onevec{\underline{1}}%
\global\long\def\One{\mathbf{1}}%

\global\long\def\porsmall{\ll}%
\global\long\def\porbig{\gg}%
\global\long\def\porequal{\asymp}%

\global\long\def\diffeo{\simeq}%

\global\long\def\brac#1{(#1)}%
\global\long\def\vbrac#1{|#1|}%
\global\long\def\sbrac#1{[#1]}%
\global\long\def\dbrac#1{\langle#1\rangle}%
\global\long\def\cbrac#1{\{#1\}}%

\global\long\def\norm#1{\left\Vert #1\right\Vert }%
\global\long\def\lilnorm#1{\Vert#1\Vert}%

\global\long\def\manifold{\mathcal{M}}%
\global\long\def\base{\mathbf{B}}%

\global\long\def\sbgrp{Y}%
\global\long\def\linmap{L}%

\global\long\def\id{\operatorname{id}}%

\global\long\def\idmat#1{\operatorname{I}_{#1}}%

\global\long\def\stab{\operatorname{stab}}%

\global\long\def\prim{\operatorname{prim}}%

\global\long\def\interior#1{\operatorname{int}\left(#1\right)}%

\global\long\def\dist#1{\operatorname{dist}\brac{#1}}%

\global\long\def\diag#1{\operatorname{diag}\left(#1\right)}%
 
\global\long\def\lildiag#1{\operatorname{diag}(#1)}%

\global\long\def\rank#1{\operatorname{rank}\left(#1\right)}%

\global\long\def\covol#1{\operatorname{covol}\brac{#1}}%

\global\long\def\Leb#1{\operatorname{Leb}(#1)}%

\global\long\def\vol{\operatorname{vol}}%

\global\long\def\sym{\mbox{Sym}}%

\global\long\def\sp{\mbox{span}}%

\global\long\def\trans#1{#1^{\transpose}}%
\global\long\def\mtrans#1{#1^{-\transpose}}%

\global\long\def\projset#1#2{#1^{\left(#2\right)}}%

\global\long\def\parby#1#2{#1_{#2}}%

\global\long\def\trunc#1#2{#1^{#2}}%

\global\long\def\perpen#1{#1^{\perp}}%

\global\long\def\lieG{\mathfrak{g}}%
\global\long\def\lieA{\mathfrak{a}}%
\global\long\def\lieN{\mathfrak{n}}%

\global\long\def\lieNelement{Z}%
\global\long\def\lieNvec{\underline{\lieNelement}}%
\global\long\def\dimN{p}%

\global\long\def\lieAelement{H}%
\global\long\def\lieAvec{\underline{\lieAelement}}%

\global\long\def\dimA{L}%
\global\long\def\dimAp{l}%

\global\long\def\topindex{q}%

\global\long\def\Ad#1{\operatorname{Ad}_{#1}}%

\global\long\def\ad#1{\operatorname{ad}_{#1}}%

\global\long\def\symfund#1{F_{#1}}%
\global\long\def\groupfund#1{\widetilde{F_{#1}}}%

\global\long\def\funddom{\Omega}%

\global\long\def\short{\operatorname{short}}%
\global\long\def\shortdom{\funddom_{\short}}%
\global\long\def\almostshortdom{\Delta_{\short}}%

\global\long\def\dirdom#1{\mbox{Dir}(#1)}%

\global\long\def\Gset{\mathcal{B}}%
\global\long\def\latset{\Psi}%
\global\long\def\symset{\mathcal{E}}%
\global\long\def\uniset{\widetilde{\symset}}%
\global\long\def\sphereset{\Phi}%

\global\long\def\Kdprimeset{\Phi^{\dprime}}%

\global\long\def\fdomN{\mathscr{D}}%
\global\long\def\domN{{\cal D}}%

\global\long\def\nbhd#1#2{\mathcal{O}_{#1}^{#2}}%

\global\long\def\hex#1{Y\left(#1\right)}%
\global\long\def\fdomhex{\mathscr{Y}}%
\global\long\def\ellipse#1#2{Y^{#2}\brac{#1}}%
\global\long\def\fdomellipse#1#2{\mathscr{Y}_{#2}^{\,#1}}%

\global\long\def\halfplane#1{\mathcal{H}_{#1}}%
 
\global\long\def\line#1{\l_{#1}}%
 
\global\long\def\striplane#1{\mathcal{H}_{\left|#1\right|}}%

\global\long\def\ball#1#2{B_{#1}^{#2}}%

\global\long\def\shapespace#1{\mathcal{X}_{#1}}%

\global\long\def\unilatspace#1{\mathcal{U}_{#1}}%

\global\long\def\latspace#1{\mathcal{L}_{#1}}%

\global\long\def\sphere#1{\mathbb{S}^{#1}}%

\global\long\def\Mat#1{\text{Mat}_{#1}}%

\global\long\def\dprime{\prime\prime}%

\global\long\def\ff#1{\operatorname{err}\,\brac{#1}}%

\global\long\def\nor{\operatorname{norm}}%
\global\long\def\cnorm{C_{\nor}}%

\global\long\def\svec{\underline{s}}%
\global\long\def\Svec{\underline{S}}%

\global\long\def\sumS{\mathbf{S}}%
\global\long\def\sums{\mathbf{s}}%

\global\long\def\rad{\rho}%
\global\long\def\sn{\iota}%

\global\long\def\roundo{r}%
\global\long\def\roundogen{r}%
\global\long\def\maxexp{\textrm{m}}%
\global\long\def\errexp{\tau}%

\global\long\def\volmin{V_{\text{min}}}%
\global\long\def\volmax{V_{\text{max}}}%
\global\long\def\radmax{R}%

\global\long\def\fam#1#2{#1_{#2}}%

\global\long\def\GIcomp{S}%
\global\long\def\zgt#1{z^{#1}}%

\global\long\def\dirfunc{\text{L}_{\a}}%

\global\long\def\sigmn{\underline{\sigma}_{n}}%

\global\long\def\lat{\Lambda}%
\global\long\def\cov{X}%

\global\long\def\shape#1{\operatorname{shape}\left(#1\right)}%

\global\long\def\unilat#1{\left[#1\right]}%

\global\long\def\unisimlat#1{\left\llbracket #1\right\rrbracket }%

\global\long\def\dir#1{\hat{#1}}%

\global\long\def\Asub#1#2{A_{#2}\brac{#1}}%

\title{Equidistribution of primitive vectors, and the shortest solutions
to their GCD equations}
\author{Tal Horesh\thanks{IST Austria, \texttt{tal.horesh@ist.ac.at}.} \and
Yakov Karasik\thanks{Department of Mathematics and Computer Science, Justus-Liebig-Universität Gießen, Germany, \texttt{theyakov@gmail.com}.}}
\maketitle
\begin{abstract}
We prove effective joint equidistribution of several natural parameters
associated to primitive vectors in $\mathbb{Z}^{n}$, as the norm
of these vectors tends to infinity. These parameters include the direction,
the orthogonal lattice, and the length of the shortest solution to
the associated $\gcd$ equation. We show that the first two parameters
equidistribute w.r.t.\ the Haar measure on the corresponding spaces,
which are the unit sphere and the space of unimodular  rank $n-1$
lattices in $\RR^{n}$ respectively. The main novelty is the equidistribution
of the shortest solutions to the $\gcd$ equations: we show that,
when normalized by the covering radius of the orthogonal lattice,
the lengths of these solutions equidistribute in the interval $\left[0,1\right]$
w.r.t. a measure that is Lebesgue only when $n=2$, and non-Lebesgue
otherwise. These equidistribution results are deduced from effectively
counting lattice points in domains which are defined w.r.t. a generalization
of the Iwasawa decomposition in simple algebraic Lie groups, where
we apply a method due to A.\ Gorodnik and A.\ Nevo. 

\tableofcontents{}
\end{abstract}

\section{Introduction\label{sec: Introduction}}

An integral vector $v=\left(a_{1},\ldots,a_{n}\right)$ is called
\emph{primitive} if $\gcd\left(a_{1},\ldots,a_{n}\right)=1$. Equidistribution
problems concerning primitive vectors first arose under the umbrella
of \emph{Linnik type problems} \cite{Linnik_68,Erdos_Hall_99,Duke_03,Duke_07,EMV_13},
a unifying name for questions that concern the distribution of the
projections of integral vectors to the unit sphere. These projections
can also be thought of as \emph{directions} of primitive vectors,
which we denote by $\dir v:=v/\norm v$. Another equidistribution
problem of primitive vectors concerns their orthogonal lattices $\lat_{v}:=\ZZ^{n}\cap\perpen v$,
where $v$ is a primitive vector, and $\perpen v$ is its orthogonal
hyperplane. Note that one can achieve a one-to-one correspondence
between primitive vectors and their orthogonal lattices by either
identifying $v$ with $-v$, or by choosing an orientation on the
lattices $\lat_{v}$; we opt for the latter. With this one-to-one
correspondence in mind, we associate to each primitive vector the
\emph{shape} of the lattice $\lat_{v}$, which is the equivalence
class of rank $n-1$ lattices in $\RR^{n}$ that can be obtained from
$\lat_{v}$ by an orientation preserving linear transformation, i.e.\
by a rotation and multiplication by a positive scalar. The equidistribution
of shapes of $\lat_{v}$, denoted $\shape{\lat_{v}}$, in the finite
volume space 
\[
\shapespace{n-1}:=\so{n-1}\left(\RR\right)\backslash\sl{n-1}\left(\RR\right)/\sl{n-1}\left(\ZZ\right)
\]
has been considered in \cite{Marklof_10,Schmidt_98}; the joint equidistribution
of $\shape{\lat_{v}}$, along with the directions of $v$, denoted
$\dir v$, has been studied in \cite{AES_16A,AES_16B,EMSS_16,ERW17}. 

Another equidistribution question for primitive vectors has been suggested
by Risager and Rudnick in \cite{R&R}, and it concerns the normalized
\emph{shortest solutions} to $\gcd$ equations: given a primitive
$v=\left(a_{1},\ldots,a_{n}\right)$, the \emph{gcd equation} of $v$
is the Diophantine equation 
\begin{equation}
a_{1}x_{1}+\cdots+a_{n}x_{n}=1,\label{eq: GCD equation}
\end{equation}
whose set of solutions is the grid $w+\lat_{v}$, with $w$ being
any solution to (\ref{eq: GCD equation}). Let $w_{v}$ denote the
shortest solution to the equation (\ref{eq: GCD equation}) w.r.t.\
the $L^{2}$ norm. The length $\norm{w_{v}}$ is unbounded as $\norm v\to\infty$,
so in order to formulate an equidistribution question for $\norm{w_{v}}$,
it should be normalized to a bounded quantity. Risager and Rudnick
(see also \cite{Truelsen,HN16_Counting}) have considered the case
of $n=2$, and showed that the quotients $\norm{w_{v}}/\norm v$ uniformly
distribute in the interval $\sbrac{0,\frac{1}{2}}$ as $\norm v\to\infty$.
This raises the question of what would be the analogous phenomenon
in higher dimensions. It turns out that one can not expect equidistribution
of $\norm{w_{v}}/\norm v$ when $n\geq3$, since these quotients tend
to zero on a full-density subset of the set of all $n$\textendash primitive
vectors, denoted $\ZZ_{\prim}^{n}$.

\begin{mythm}

\label{thm: |w_v|/|v| to zero}There exists a subset $\cA$ of $\ZZ_{\prim}^{n}$
with 
\[
\underset{R\to\infty}{\lim}\frac{\#(\cA\cap\ball R{})}{\#(\ZZ_{\prim}^{n}\cap\ball R{})}=1,
\]
where $\ball R{}=\cbrac{v\in\RR^{n}:\norm v\leq R}$, such that for
every sequence $\{v_{m}\}\subset\cA$, the quotients $\norm{w_{v_{m}}}/\norm{v_{m}}$
tend to zero as $m\to\infty$. 

\end{mythm}

Indeed, the above theorem (as well as Corollary \ref{cor: |w_v|/|v| no ed}
below) suggests that in dimension greater than $2$, the ``correct''
normalization of the shortest solution is not by the norm of $v$.
Hence, approaching Risager and Rudnick's problem in higher dimensions
consists in fact of three questions:
\begin{enumerate}
\item[(i)] What is the correct normalization of the shortest solutions in dimension
$n\geq3$?
\item[(ii)] In which interval do the normalized shortest solutions fall?
\item[(iii)] With respect to which measure on this interval, if any, do the normalized
shortest solutions equidistribute?
\end{enumerate}
\begin{figure}
\begin{centering}
\includegraphics[scale=1.35]{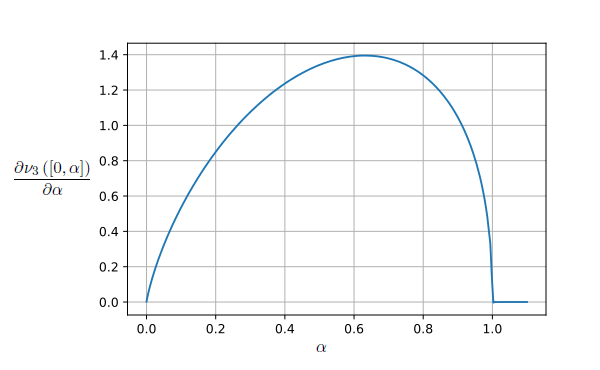}
\par\end{centering}
\caption{\label{fig: measure for gcd 3}The density of $\nu_{3}$}
\end{figure}

In this paper we offer a complete solution to the problem of equidistribution
of the normalized shortest solutions of $\gcd$ equations, answering
all three questions above. We show that the correct normalization
of $\norm{w_{v}}$ is by the \emph{covering radius} $\rad_{v}$ of
the lattice $\lat_{v}$ (the covering radius of a lattice is the radius
of a bounding sphere for its Dirichlet domain), and construct a measure
$\nu_{n}$ with respect to which the quotients $\norm{w_{v}}/\rad_{v}$
equidistribute in the interval $\left[0,1\right]$. It turns out that
in general the measure $\nu_{n}$ on $\left[0,1\right]$ is  non-uniform
(see Figure \ref{fig: measure for gcd 3} for the density function
of $\nu_{3}$), except for the case of $n=2$: there, the measure
$\nu_{2}$ is Lebesgue and the covering radius is $\rad_{v}=\norm v/2$,
hence we recover the result of Risager and Rudnick. 

In fact we do more, and show that the equidistribution of $\norm{w_{v}}/\rad_{v}$
occurs \emph{jointly} with the uniform distribution of $\dir v$ in
$\sphere{n-1}$.  We also obtain the previously known joint equidistribution
of shapes $\shape{\lat_{v}}$ and directions $\dir v$ from the equidistribution
of another parameter of $\lat_{v}$, that encodes information of both
$\shape{\lat_{v}}$ and $\dir v$. Consider the space 
\[
\latspace{n-1,n}:=\sl n\left(\RR\right)/\left(\left[\begin{array}{cc}
\sl{n-1}\left(\ZZ\right) & \RR^{n-1}\\
0_{1\times n} & 1
\end{array}\right]\times\left\{ \left[\begin{array}{cc}
\a^{-\frac{1}{n-1}}\idmat{n-1} & 0_{n\times1}\\
0_{1\times n} & \a
\end{array}\right]:\a>0\right\} \right),
\]
which is the space of homothety classes of $\left(n-1\right)$-lattices
inside $\RR^{n}$. We identify this space with the space of unimodular
(i.e. covolume one) $\left(n-1\right)$-lattices inside $\RR^{n}$,
\begin{equation}
\latspace{n-1,n}\diffeo\so n(\RR)\left[\begin{array}{cc}
P_{n-1} & 0\\
0 & 1
\end{array}\right]/\left[\begin{array}{cc}
\sl{n-1}\left(\ZZ\right) & 0\\
0 & 1
\end{array}\right],\label{eq: right quotient for L_n-1,n}
\end{equation}
where $P_{n-1}<\sl{n-1}(\RR)$ is the group of upper triangular matrices
with positive diagonal entries. The identification is by associating
to each equivalence class $\unilat{\lat}$ the unique representative
of covolume one, which we also denote by $\unilat{\lat}$.  The space
$\latspace{n-1,n}$ is canonically projected to $\shapespace{n-1}$
and to $\sphere{n-1}$, by modding out from the left by $\so n(\RR)$
or by $\so{n-1}(\RR)\left[\begin{smallmatrix}P_{n-1} & 0\\
0 & 1
\end{smallmatrix}\right]$ respectively, and the projections of $\unilat{\lat_{v}}$ to $\shapespace{n-1}$
and $\sphere{n-1}$ are exactly $\shape{\lat_{v}}$ and $\dir v$. 

From the equidistribution of $\unilat{\lat_{v}}$ in $\latspace{n-1,n}$,
we will also conclude the joint equidistribution of the directions
$\dir v$ together with the projections of $\lat_{v}$ to the following
space: 

\[
\unilatspace{n-1}:=\sl{n-1}\left(\RR\right)/\sl{n-1}\left(\ZZ\right),
\]
which is the space of \emph{unimodular} lattices of rank $n-1$. We
denote these projections by $\unisimlat{\lat_{v}}$ (this projection
is in fact not canonical, and depends on a choice of coordinates that
will be made in Section \ref{subsec: RI coordinates of an element }). 

The equidistribution in the spaces $\shapespace{n-1}$, $\unilatspace{n-1}$,
$\latspace{n-1,n}$ and $\sphere{n-1}$ is a uniform distribution,
namely w.r.t.\ a finite uniform invariant measure, which is unique
up to a choice of normalization. We denote these measures by $\mu_{\shapespace{n-1}}$,
$\mu_{\unilatspace{n-1}}$, $\mu_{\latspace{n-1,n}}$ and $\mu_{\sphere{n-1}}$,
and expand about them below, after the statement of our main result.
The measure $\mu_{\sphere{n-1}}$, for example, is the Lebesgue measure
on the sphere. 

The equidistribution of the quotients $\norm{w_{v}}/\rad_{v}$ inside
$[0,1]$ is, as we have already mentioned, not uniform. The proportion
of primitive vectors $v$ for which the quotients $\norm{w_{v}}/\rad_{v}$
fall within the interval $[0,\a]$ with $0\leq\a\leq1$ is given by
the map $\dirfunc:\shapespace{n-1}\to\RR^{+}$ which is defined by
associating to every $z\in\shapespace{n-1}$ the following quantity.
Recall that $z\in\shapespace{n-1}$ is a unimodular lattice in $\RR^{n-1}$
up to rotation. Recall also that the Dirichlet domain of a lattice
is symmetric around the origin, and so the Lebesgue volume of $\dirdom z\cap\ball{}{}$,
where $\dirdom z$ is the Dirichlet domain of any lattice in the class
$z$ and $\ball{}{}$ is a ball centered at the origin, is independent
of the choice of a representative from $z$. Let
\[
\dirfunc(z)=\Leb{\dirdom z\cap\ball{\a\rad\left(z\right)}{}},
\]
where $\text{Leb}$ is the Lebesgue measure, $\rad(z)$ is the covering
radius of (any representative from) $z$, and $\ball{\a\rad\left(z\right)}{}$
is an origin centered ball in $\RR^{n-1}$ with radius $\a\rad\left(z\right)$. 

Finally, we derive our equidistribution results by counting primitive
vectors $v$ (resp.\ primitive $(n-1)$-lattices $\lat_{v}$) whose
projections to the aforementioned spaces lie in subsets that have
\emph{controlled boundary}: this is a rather soft condition on the
boundary of subsets of orbifolds that is defined explicitly in Section
\ref{sec: Fundamental domains}, and is met, e.g., when the boundary
of the set is contained in a finite union of $C^{1}$ submanifolds
of strictly lower dimension than the one of the orbifold. We refer
to a set with controlled boundary as a \emph{boundary controllable
set}, or a BCS. Our main result is the following. 

\begin{mythm}

\label{thm: MainThm}Assume that $\sphereset\subseteq\sphere{n-1}$,
$\symset\subseteq\shapespace{n-1}$ $\ensuremath{\uniset}\subseteq\unilatspace{n-1}$
and $\latset\subset\latspace{n-1,n}$ are BCS's.
\begin{enumerate}
\item \label{enu: MainThm_restricted shape}The number of $v\in\ZZ_{\prim}^{n}$
 with $\norm v\leq e^{T}$, $\dir v\in\sphereset$, $\shape{\lat_{v}}\in\symset$
and $\left\Vert w_{v}\right\Vert /\rad_{v}\in\left[0,\a\right]$ is
\[
\frac{\mu_{\sphere{n-1}}(\sphereset)\cdot\int_{\symset}\dirfunc(z)d\mu_{\shapespace{n-1}}\left(z\right)}{n\prod_{i=2}^{n}\zeta\left(i\right)}\cdot\frac{\prod_{i=1}^{n-2}\Leb{\sphere i}}{\sn\left(n-1\right)}\cdot e^{nT}+\text{error term}
\]
where 
\begin{equation}
\sn\left(m\right)=\left[\so m\left(\RR\right):Z\left(\so m\left(\RR\right)\right)\right]=\begin{cases}
2 & \mbox{if }m\mbox{ is even}\\
1 & \mbox{if }m\mbox{ is odd}
\end{cases}.\label{notation: sign of n}
\end{equation}
\item \label{enu: MainThm_restricted G''}The number of $v\in\ZZ_{\prim}^{n}$
with $\norm v\leq e^{T}$, $\dir v\in\sphereset$, $\unisimlat{\lat_{v}}\in\ensuremath{\uniset}$
and $\left\Vert w_{v}\right\Vert /\rad_{v}\in\left[0,\a\right]$ is
\[
\frac{\mu_{\sphere{n-1}}(\sphereset)\cdot\int_{\ensuremath{\uniset}}\dirfunc(\pi_{_{\unilatspace{}\to\shapespace{}}}\left(\tilde{z}\right))d\mu_{\unilatspace{n-1}}\left(\tilde{z}\right)}{n\prod_{i=2}^{n}\zeta\left(i\right)}\cdot e^{nT}+\text{error term},
\]
where $\pi_{_{\unilatspace{}\to\shapespace{}}}$ is the projection
from $\unilatspace{n-1}$ to $\shapespace{n-1}$.
\item \label{enu: MainThm_restricted Q}The number of $v\in\ZZ_{\prim}^{n}$
with $\norm v\leq e^{T}$, $\unilat{\lat_{v}}\in\latset$ and $\left\Vert w_{v}\right\Vert /\rad_{v}\in\left[0,\a\right]$
is 
\[
\frac{\int_{\latset}\dirfunc(\pi_{_{\latspace{}\to\shapespace{}}}\left(y\right))d\mu_{\latspace{n-1,n}}\left(y\right)}{n\prod_{i=2}^{n}\zeta\left(i\right)}\cdot e^{nT}+\text{error term},
\]
where $\pi_{_{\latspace{}\to\shapespace{}}}$ is the projection from
$\latspace{n-1,n}$ to $\shapespace{n-1}$.
\end{enumerate}
The error term is $O_{\e}(e^{nT\left(1-\errexp_{n}+\e\right)})$
with $\errexp_{n}=\left\lceil \left(n-1\right)/2\right\rceil /4n^{2}$
for every $\e>0$ when $\symset$ (resp.\  $\latset$, $\ensuremath{\uniset}$)
is bounded, and $O_{\e}(e^{nT\left(1-\eta_{n}\errexp_{n}+\e\right)})$
with $\eta_{n}=n^{2}/(2n^{3}-3n^{2}-2n+4)$ when it is not. 

\end{mythm}

The lattice $\lat_{v}$ has covolume $\norm v$ and it is primitive,
where a lattice $\lat$ in $\ZZ^{n}$ is said to be primitive if it
is of the form $V\cap\ZZ^{n}$, with $V$ being a linear subspace
of $\RR^{n}$ of dimension $\rank{\lat}$. Then, Theorem \ref{thm: MainThm}
can also be read as a counting result for primitive $(n-1)$\textendash lattices,
as their covolume tends to infinity. 

The above theorem solves the question of equidistribution of the normalized
shortest solutions; indeed, for $0\leq\a\leq1$, let 
\[
\nu_{n}\left(\left[0,\a\right]\right)=\int_{z\in\shapespace{n-1}}\dirfunc(z)d\mu_{\shapespace{n-1}}\left(z\right).
\]
The following is now straightforward from part (\ref{enu: MainThm_restricted shape})
of Theorem \ref{thm: MainThm}:
\begin{cor}
For primitive vectors $v\in\ZZ^{n}$ with $n\geq2$, the normalized
shortest solutions $\norm{w_{v}}/\rad_{v}$ and the directions $\dir v$
jointly equidistribute as $\norm v\to\infty$: the quotients $\norm{w_{v}}/\rad_{v}$
inside $[0,1]$ w.r.t.\ $\nu_{n}$, and the directions $\dir v$
inside the unit sphere w.r.t.\ the Lebesgue measure. 
\end{cor}

As we have already mentioned, for the case of $n=2$ the above corollary
recovers the result of Risager and Rudnick for uniform distribution
of $\norm{w_{v}}/\norm v$ in the interval $[0,1/2]$. In particular,
the ${n \choose 2}$ embeddings of $\RR^{2}$ into $\RR^{n}$ that
are of the form 
\[
(x,y)\mapsto(0,\ldots,0,x,0,\ldots,0,y,0,\ldots0)
\]
give birth to ${n \choose 2}$ sequences of primitive vectors $v\in\ZZ^{n}$
for which the quotients $\norm{w_{v}}/\norm v$ uniformly distribute
in the interval $[0,1/2]$ as $\norm v\to\infty$. Combining this
with Theorem \ref{thm: |w_v|/|v| to zero}, we conclude:
\begin{cor}
\label{cor: |w_v|/|v| no ed}For primitive vectors $v\in\ZZ^{n}$
with $n\geq3$, there is no Borel measure on $\RR$ w.r.t.\ which
the quotients $\norm{w_{v}}/\norm v$ equidistribute as  $\norm v\to\infty$.
\end{cor}

\paragraph{The measures $\mu_{\protect\shapespace{n-1}}$, $\mu_{\protect\unilatspace{n-1}}$,
$\mu_{\protect\latspace{n-1,n}}$ and $\mu_{\protect\sphere{n-1}}$.}

The measure $\mu_{\sphere{n-1}}$ is the Lebesgue measure on the sphere.
The measures $\mu_{\shapespace{n-1}}$ and $\mu_{\unilatspace{n-1}}$
are the unique Radon invariant measures arriving from a Haar measure
on $\sl{n-1}(\RR)$ that are normalized as follows: the $\mu_{\unilatspace{n-1}}$
volume of $\unilatspace{n-1}$ is 
\[
\prod_{i=2}^{n-1}\zeta\left(i\right),
\]
and the $\mu_{\shapespace{n-1}}$-volume of $\shapespace{n-1}$ is
\[
\sn\left(n-1\right)\prod_{i=2}^{n-1}\zeta\left(i\right)/\brac{\prod_{i=1}^{n-2}\Leb{\sphere i}},
\]
where $\sn:\NN\to\left\{ 1,2\right\} $ was defined in (\ref{notation: sign of n}).
The justification for the volume of $\unilatspace{n-1}$ is the computation
in \cite{Volumes} along our choice of Haar measure on $\sl n\left(\RR\right)$
that is explained in Subsection \ref{subsec: Defining RI}. This
choice determines the volumes of $\shapespace{n-1}$, as shown in
Lemma \ref{lem: lift to G''}. On $\latspace{n-1,n}$, however there
is no invariant measure induced from $\sl n(\RR)$, and instead we
view this space as the quotient in (\ref{eq: right quotient for L_n-1,n}),
where a submanifold of $\sl n(\RR)$ quotiented by a discrete group.
This submanifold supports a transitive action of the product group
$\so n(\RR)\times\left[\begin{smallmatrix}P_{n-1} & 0\\
0 & 1
\end{smallmatrix}\right]$, and $\mu_{\latspace{n-1,n}}$ is the unique Radon measure that is
invariant under this action and satisfies that the $\mu_{\latspace{n-1,n}}$-volume
of $\latspace{n-1,n}$ is the product of volumes of $\sphere{n-1}$
and $\unilatspace{n-1}$. 

\paragraph{Comparison with previous work. }

Let us comment on related work that preceded the theorem above. As
already mentioned, equidistribution of the $\norm{w_{v}}/\frac{1}{2}\norm v$
was known for $n=2$; it was first proved in \cite{R&R}, and effective
versions were later established in \cite{Truelsen} and \cite{HN16_Counting},
where the error term coincides with the one of Theorem \ref{thm: MainThm}
for $n=2$. The equidistribution (in a non-effective manner) of shapes
of primitive lattices of any rank was established in \cite{Schmidt_98};
the case of rank $n-1$ was also obtained in \cite{Marklof_10}, using
a dynamical approach. Theorem \ref{thm: MainThm} adds an error term
(i.e. rate of convergence) to two of the aforementioned results, as
well as the consideration of the projections to $\unilatspace{n-1}$
and $\latspace{n-1,n}$ (as apposed to just $\shapespace{n-1}$),
and most importantly, the equidistribution related to the $\gcd$
problem. Another significant addition is the fact that we allow the
projections to the relevant spaces ($\symset$, $\ensuremath{\uniset}$,
$\latset$) to be unbounded; to this end, it is critical that the
counting includes an error term, since it could be compromised to
allow unboundedness. Our method can be used to consider the case of
general co-dimension as well, which we will do in a forthcoming paper.
Effective counting of primitive lattices was done in \cite{Schmidt_68},\cite{Schmidt_15},
but the subsets $\symset$ in the shape space were not general enough
to deduce equidistribution.  Joint equidistribution of shapes and
directions has been studied, e.g. in \cite{AES_16A,AES_16B,EMSS_16},
in the  case where the primitive vectors $v$ are restricted to a
large \emph{sphere} $\norm v=e^{T}$, as apposed to a large \emph{ball}
$\norm v\leq e^{T}$, the latter being the case considered in Theorem
\ref{thm: MainThm}. The sphere case is of course much more delicate,
and this is the reason why almost\footnote{In \cite{ERW17} an error term is established for dimensions $n=4,5$.}
all existing results do not include an error term. The key to proving
Theorem \ref{thm: MainThm} is counting lattice points in the group
$\sl n\left(\RR\right)$ w.r.t.\ the Iwasawa coordinates; in the
context of counting points of discrete subgroups inside simple Lie
groups w.r.t.\ a decomposition of the group, we mention \cite{Good,GN1,GOS_wavefront,Margulis_Mohammadi_Oh}.

\paragraph{Outline of the paper.}

The proof of Theorems \ref{thm: |w_v|/|v| to zero} and \ref{thm: MainThm}
consist of two main ideas, and the paper is divided accordingly:
\begin{enumerate}
\item \textbf{A reduction to a problem of counting lattice points in the
group $\sl n\left(\RR\right)$} (Part I), which is done by finding
``isomorphic'' copies of the spaces $\shapespace{n-1}$, $\unilatspace{n-1}$,
$\latspace{n-1,n}$, $\sphere{n-1}$ inside $\sl n\left(\RR\right)$
(Section \ref{sec: Fundamental domains}) and establishing a correspondence
between primitive vectors $v$ (resp.\  primitive lattices $\lat_{v}$)
and integral matrices in  $\sl n\left(\RR\right)$ (Section \ref{sec: Integral matrices representing primitive vectors}),
such that the projections of the primitive lattices to the spaces
$\shapespace{n-1}$ etc.\ will correspond to the projections of the
integral matrices in their isomorphic copies. This converts Theorem
\ref{thm: MainThm} into a counting lattice points problem in $\sl n\left(\RR\right)$
(Section \ref{sec: Defining a counting problem}). A key role in this
translation is played by a refinement of the Iwasawa coordinates of
$\sl n\left(\RR\right)$, introduced in Section \ref{sec: RI components}.
In section \ref{sec: Number of integral points up the cusp} we simplify
the counting problem by reducing to counting in a family of compact
subsets of $\sl n\left(\RR\right)$, by providing a rather direct
estimate for the number of lattice up to a given covolume that lie
far up the cusp  in the space of $(n-1)$-lattices. In the concluding
section \ref{sec: Almost-a-proof} of Part I we state Proposition
\ref{prop: Counting with Hexagons (A counting)}, which formulates
the final counting question in $\sl n\left(\RR\right)$ that is required
in order to complete the proof Theorem \ref{thm: MainThm}, and then
use it to prove Theorems \ref{thm: |w_v|/|v| to zero} and \ref{thm: MainThm}.
\item \textbf{Solving the counting problems} (Part II ). This part is devoted
to proving the aforementioned Proposition \ref{prop: Counting with Hexagons (A counting)}.
The main ingredient is a method due to A. Gorodnik and A. Nevo \cite{GN1},
which concerns counting lattice points in increasing families $\left\{ \Gset_{T}\right\} _{T>0}$
inside non-compact algebraic simple Lie groups. In Section \ref{sec: Well roundededness}
we describe this method, and sketch a plan for completing the proof
of Proposition \ref{prop: Counting with Hexagons (A counting)} according
to it. In Sections \ref{sec: Regularity-results-forA}, \ref{sec: effective Iwasawa and GI decompositions},
\ref{sec: The base sets are LWR}, \ref{sec: Family for gcd solution is BLC}
we follow that plan, and the proofs are concluded in Section \ref{sec: Concluding the proofs}.
\end{enumerate}

\paragraph*{Notations for inequalities.\label{subsec: Notations and conventions}}

We will use the following conventions for inequalities. If $\underline{a}=\left(a_{1},\ldots,a_{n}\right)$
and $\underline{b}=\left(b_{1},\ldots,b_{n}\right)$ are two $n$-tuples
of real numbers, we denote $\underline{a}\leq\underline{b}$ if $a_{i}\leq b_{i}$
for every $i=1,\ldots,n$. If $f$ and $g$ are two non-negative functions
then we denote $f\porsmall g$ if there exists a positive constant
$C$ and some $t_{0}$ such that for $t_{0}<t$ one has $f(t)\leq Cg(t)$.
We denote $f\porequal g$ if $g\porsmall f\porsmall g$. 
\begin{acknowledgement*}
This work was done when both authors were at IHES (Institut des Hautes
Études Scientifiques, France), and we are grateful for the opportunity
to work there, and for the outstanding hospitality. The authors are
also grateful to Nadav Horesh for his help with numerical estimations
for the measure $\nu_{3}$, and to Ami Paz for his help with preparing
the figures. We would also like to thank Barak Weiss and Amos Nevo
for helpful discussions in early stages of the project, and to Micheal
Bersudsky for referring us to Schmidt's work on effective counting
of primitive lattices. 
\end{acknowledgement*}

\part{From $\protect\ZZ^{n}$ to $\protect\sl n\left(\protect\ZZ\right)$\label{part: Non Technical part}}

\section{The Refined Iwasawa decomposition of $\protect\sl n\left(\protect\RR\right)$\label{sec: RI components} }

\subsection{\label{subsec: Defining RI}Refining the Iwasawa decomposition}

Set $G:=\sl n\left(\RR\right)$ and let $K$ be $\so n\left(\RR\right)$,
$A$ the diagonal subgroup in $G$, and $N$ the subgroup of upper
unipotent matrices. Then, $G=KAN$ is the Iwasawa decomposition of
$G$. Consider yet another subgroup of $G$,
\[
G^{\dprime}:=\left[\begin{array}{c|c}
\sl{n-1}\left(\RR\right) & \begin{array}{c}
0\\
\underset{}{\vdots}
\end{array}\\
\hline \begin{array}{ccc}
0 & \cdots & 0\end{array} & 1
\end{array}\right],
\]
which is clearly an isomorphic copy of $\sl{n-1}\left(\RR\right)$
inside $G$. Write $G^{\dprime}=K^{\dprime}A^{\dprime}N^{\dprime}$
for the Iwasawa decomposition of $G^{\dprime}$, i.e. 
\begin{eqnarray*}
K^{\dprime} & := & K\cap G^{\dprime}=\left[\begin{array}{c|c}
\so{n-1}\left(\RR\right) & \begin{array}{c}
0\\
0
\end{array}\\
\hline \begin{array}{cc}
0 & 0\end{array} & 1
\end{array}\right],\\
A^{\dprime} & := & A\cap G^{\dprime}=\diag{\a_{1},\ldots,\a_{n-1},1}\text{ with }\a_{1}\cdots\a_{n-1}=1,\\
N^{\dprime} & := & N\cap G^{\dprime}=\left[\begin{array}{c|c}
\begin{array}{c}
\text{upper unipotent}\\
\text{of order \ensuremath{n-1}}
\end{array} & \begin{array}{c}
0\\
0
\end{array}\\
\hline \begin{array}{cc}
0 & 0\end{array} & 1
\end{array}\right].
\end{eqnarray*}
The crux of the RI decomposition is that it completes the Iwasawa
decomposition of $G^{\dprime}$ to the Iwasawa decomposition of $G$.
For this we define $K^{\prime},A^{\prime},N^{\prime}$ that complete
$K^{\dprime},A^{\dprime},N^{\dprime}$ to $K$, $A$ and $N$ respectively.
Define 
\[
N^{\prime}:=\left[\begin{array}{c|c}
\idmat{n-1} & \RR^{n-1}\\
\hline \begin{array}{cc}
0 & 0\end{array} & 1
\end{array}\right],\quad A^{\prime}:=\left[\begin{array}{c|c}
a^{-\frac{1}{n-1}}\idmat{n-1} & \begin{array}{c}
0\\
0
\end{array}\\
\hline \begin{array}{cc}
0 & 0\end{array} & a
\end{array}\right]
\]
and note that $N=N^{\dprime}N^{\prime}$, $A=A^{\dprime}A^{\prime}$,
and that $A^{\prime}$ is a one-parameter subgroup of $A$ which commutes
with $G^{\dprime}$. Fix a transversal $K^{\prime}$ of the diffeomorphism
$K/K^{\dprime}\to\sphere{n-1}$ with the following property:
\begin{condition}
\label{cond: K' property}If $\sphereset\subseteq\sphere{n-1}$ and
$\Kdprimeset\subseteq K^{\dprime}$ are BCS, then so does $\Kdprimeset K_{\sphereset}^{\prime}\subseteq K$,
where $K_{\sphereset}$ is the inverse image of $\sphereset$ in $K^{\prime}$.
\end{condition}

The existence of such a transversal $K^{\prime}$ is proved in Lemma
\ref{lem: K'}. Let
\[
P^{\dprime}:=A^{\dprime}N^{\dprime}\text{ and }\wc:=KP^{\dprime};
\]
note that $\wc$ is not a group, but that it is a smooth manifold
that is diffeomorphic to the group $K\times P^{\dprime}$. The RI
decomposition is given by 

\[
G=K^{\prime}G^{\dprime}A^{\prime}N^{\prime}=K^{\prime}K^{\dprime}A^{\dprime}A^{\prime}N^{\dprime}N^{\prime},
\]
and we also have $G=\wc A^{\prime}N^{\prime}$. 

\paragraph*{Parameterizations of the RI components. }

Clearly the groups $A,A^{\prime},A^{\dprime}$ and $N,N^{\prime},N^{\dprime}$
are parameterized by the Euclidean spaces of the corresponding dimensions.
For $t\in\RR$, $\underline{s}=\left(s_{1},\ldots,s_{n-1}\right)\in\RR^{n-2}$
and $\underline{x}\in\RR^{n-1}$, we let $a_{t}^{\prime}:=\lildiag{e^{\frac{t}{n-1}}\idmat{n-1},e^{-t}}$,
$a_{\underline{s}}^{\dprime}:=\lildiag{e^{-\frac{s_{1}}{2}},e^{\frac{s_{1-}s_{2}}{2}},\ldots,e^{\frac{s_{n-2}}{2}},1}$
and $n_{\underline{x}}^{\prime}=\left[\begin{smallmatrix}I_{n-1} & \underline{x}\\
0 & 1
\end{smallmatrix}\right]$. Similarly, since $K^{\prime}$ parameterizes the unit sphere $\sphere{n-1}$,
we let $k_{u}^{\prime}$ denote the element in $K^{\prime}$ corresponding
to a unit vector $u\in\sphere{n-1}$. In addition to the above, we
will show in Section \ref{subsec: Spread models} that certain subsets
of $\wc$, $G^{\dprime}$ and $P^{\dprime}$ parameterize the spaces
$\latspace{n-1,n}$, $\unilatspace{n-1}$ and $\shapespace{n-1}$.
When an RI component $\GIcomp$ (or a subset of it) parameterizes
a space $X$, and $\Gset\subset X$ is a subset, we let $\GIcomp_{\Gset}$
denote the image of $\Gset$ under the parameterization. For example,
if $\domN\subset\RR^{n-1}$, then $N_{\domN}^{\prime}$ denotes its
image in $N^{\prime}$, namely the set of $n_{\underline{x}}^{\prime}$
where $\underline{x}\in\domN$. 

\paragraph*{Measures on the RI components. }

For every $\GIcomp\subset G$ appearing as a component in the Iwasawa
or Refined Iwasawa decompositions of $G$, we let $\mu_{\GIcomp}$
denote a measure on $\GIcomp$ as follows: $\mu_{K},\mu_{N}$ are
Haar measures, and so do $\mu_{K^{\dprime}}$, $\mu_{N^{\dprime}}$,
$\mu_{P^{\dprime}}$, $\mu_{G^{\dprime}},\mu_{G}$ and $\mu_{N^{\prime}}$.
The measures $\mu_{N}$, $\mu_{N^{\prime}}$ and $\mu_{N^{\dprime}}$
are Lebesgue; as $N=N^{\dprime}\ltimes N^{\prime}$ and all three
groups are unimodular, $\mu_{N}=\mu_{N^{\dprime}}\times\mu_{N^{\prime}}$.
Since $K^{\prime}$ parameterizes $\sphere{n-1}$,  we can endow
it with a measure $\mu_{K^{\prime}}$ that is the pullback of the
Lebesgue measure on the sphere. We assume that the Haar measures
$\mu_{K}$ and $\mu_{K^{\dprime}}$ are normalized such that $\mu_{K}=\mu_{K^{\dprime}}\times\mu_{K^{\prime}}$.
Then, by choosing the measure of $K^{\dprime}$ to be $\prod_{i=1}^{n-2}\Leb{\sphere i}$,
we have that the measure of $K$ is $\prod_{i=1}^{n-1}\Leb{\sphere i}$.
The measures $\mu_{A},\mu_{A^{\prime}},\mu_{A^{\dprime}}$ are Radon
measures such that 
\[
\mu_{A^{\prime}}=e^{nt}dt,\quad\mu_{A^{\dprime}}=\prod_{i=1}^{n-2}e^{-s_{i}}ds_{i}
\]
as we compute in Example \ref{exa: SL(n,R)  1}, and $\mu_{A}=\mu_{A^{\prime}}\times\mu_{A^{\dprime}}$
by Remark \ref{rem: H sbgrps of A}. Note that these measures are
non-Haar. Since $\wc$ is diffeomorphic to the group $K\times P^{\dprime}$,
we endow it with the Haar measure on this group: $\mu_{\wc}=\mu_{K}\times\mu_{P^{\dprime}}$.
Since $\mu_{K}=\mu_{K^{\prime}}\times\mu_{K^{\dprime}}$, we also
have that also $\mu_{\wc}=\mu_{K^{\prime}}\times\mu_{G^{\dprime}}$.
All in all, the Haar measure on $G$, which can be written in Iwasawa
coordinates as $\mu_{G}=\mu_{K}\times\mu_{A}\times\mu_{N}$, (e.g.
\cite[Prop. 8.43]{Knapp}), can be also decomposed according to the
Refined Iwasawa coordinates: 
\begin{equation}
\begin{array}{c}
\mu_{G}=\mu_{K^{\prime}}\times\mu_{G^{\dprime}}\times\mu_{A^{\prime}}\times\mu_{N^{\prime}}=\mu_{\wc}\times\mu_{A^{\prime}}\times\mu_{N^{\prime}}\\
=\mu_{K^{\prime}}\times\mu_{K^{\dprime}}\times\mu_{A^{\dprime}}\times\mu_{N^{\dprime}}\times\mu_{A^{\prime}}\times\mu_{N^{\prime}}
\end{array}.\label{eq: Haar measure on G}
\end{equation}
Where it should be clear from the context, we will occasionally denote
$\mu$ instead of $\mu_{G}$

\subsection{\label{subsec: RI coordinates of an element }Explicit RI components
of $g\in\protect\sl n\left(\protect\RR\right)$, and their interpretation}

The following proposition reveals the role of the RI decomposition
of $\sl n\left(\RR\right)$ in studying the parameters $\norm v$,
$\dir v$, $\unilat{\lat_{v}}$, $\unisimlat{\lat_{v}}$ and $\shape{\lat_{v}}$
of a vector $v$. Let us observe that the projection $\unilat{\lat_{v}}$
to $\unilatspace{n-1}$ is now well defined, following the choice
of a transversal $K^{\prime}$, which determines a unique way to rotate
any hyperplane in $\RR^{n}$ to $\sp\cbrac{e_{1},\ldots,e_{n-1}}\cong\RR^{n-1}$.

It will be convenient to set the following notations: for any invertible
matrix $g$, let $\lat_{g}$ denote the lattice spanned by the columns
of $g$, and $\lat_{g}^{j}$ denote the lattice spanned by the first
$j$ columns of $g$. Also, for $0\neq v\in\RR^{n}$ define:

\begin{equation}
G_{v}=\left\{ g=\sbrac{v_{1}|\cdots|v_{n}}\in\sl n\left(\RR\right):v_{1}\wedge\cdots\wedge v_{n-1}=v\right\} .\label{eq: def of G_v}
\end{equation}

\begin{prop}
\label{prop: explicit RI coordinates of g}Let $g=\left[v_{1}|\cdots|v_{n-1}|v_{n}\right]\in\sl n\left(\RR\right)$
and write $g=kan=qa_{t}^{\prime}n_{\underline{x}}^{\prime}$ with
$q=k_{u}^{\prime}g^{\dprime}$ and $g^{\dprime}=k^{\dprime}a_{\underline{s}}^{\dprime}n^{\dprime}$.
Let $w=v_{n}$, and $p^{\dprime}=a_{\underline{s}}^{\dprime}n^{\dprime}=\left[\begin{smallmatrix}z & 0\\
0 & 1
\end{smallmatrix}\right]$. If $g\in G_{v}$, then the RI components of $g$ are as follows:
\[
\begin{array}{ccccc}
(i) & u & = & \dir v\\
(ii) & e^{t} & = & \norm v\\
(iii) & e^{-\frac{s_{i}}{2}+\frac{it}{n-1}} & = & \covol{\lat_{g}^{i}}\\
(iv) & \lat_{q} & \in & \unilat{\lat_{g}^{n-1}}\\
(v) & \lat_{g^{\dprime}} & \in & \unisimlat{\lat_{g}^{n-1}}\\
(vi) & \ensuremath{\lat_{z}} & \in & \shape{\lat_{g}^{n-1}}
\end{array}
\]
and (vii) $\underline{x}=\left(x_{1},\ldots,x_{n-1}\right)$ is such
that $w^{\perpen v}=\sum_{i=1}^{n-1}x_{i}v_{i}$. 
\end{prop}

The proof of this proposition requires two short lemmas regarding
the elements of $G_{v}$. 
\begin{lem}
\label{lem: characterize G_v}For $g=\left[v_{1}|\cdots|v_{n-1}|v_{n}\right]\in\sl n\left(\RR\right)$,
the following are equivalent:

\begin{enumerate}
\item \label{enu: g in G_v}$g\in G_{v}$.
\item \label{enu: first columns span v perp}The columns $\left\{ v_{1},\ldots v_{n-1}\right\} $
form a basis of co-volume $\left\Vert v\right\Vert $ to $\perpen v$
such that $\left\{ v_{1},\ldots,v_{n-1},v\right\} $ is a positively
oriented basis w.r.t. the standard basis of $\RR^{n}$. 
\item \label{enu: upper row's projection is 1}$\left\langle v_{n},v\right\rangle =1$,
and $\left\langle v_{i},v\right\rangle =0$ for $i=1,\ldots,n-1$. 
\end{enumerate}
\end{lem}

\begin{proof}
(\ref{enu: g in G_v}) $\iff$ (\ref{enu: first columns span v perp})
by definition. The direction (\ref{enu: g in G_v}) $\Longrightarrow$
(\ref{enu: upper row's projection is 1}) follows from 
\[
1=\det\brac g=\dbrac{\brac{v_{1}\wedge\cdots\wedge v_{n-1}},v_{n}}=\dbrac{v,v_{n}}.
\]
Conversely, (\ref{enu: upper row's projection is 1}) implies $v=\a\cdot v_{1}\wedge\cdots\wedge v_{n-1}$
for some $\a\neq0$, and that $\dbrac{\a^{-1}v,v_{n}}=\a^{-1}$. But
since (as above) $1=\dbrac{\a^{-1}v,v_{n}}$, this forces $\a=1$.
\end{proof}
\begin{lem}
\label{lem: v-comp of w}If $g\in G_{v}$, the last column of $g$
is $w$ and $w^{\perpen v}$ is the orthogonal projection of $w$
on the hyperplane $\perpen v$, then 
\[
w=w^{\perpen v}+\norm v^{-2}\,v.
\]
\end{lem}

\begin{proof}
Write $w=w^{\perpen v}+\a v$. By part (\ref{enu: upper row's projection is 1})
of Lemma \ref{lem: characterize G_v}, $1=\left\langle w,v\right\rangle =\langle w^{\perpen v}+\a v,v\rangle=\left\langle \a v,v\right\rangle $,
hence $\a=\frac{1}{\norm v^{2}}$.
\end{proof}
\begin{proof}[proof of Proposition \ref{prop: explicit RI coordinates of g}]
Write $k=\left[\phi_{1}|\cdots|\phi_{n-1}|\phi_{n}\right]$. Since
the columns of $k$ are the orthonormal basis obtained by the Gram-Schmidt
algorithm on the columns of $g$, we have that $\sp\left\{ \phi_{1},\ldots,\phi_{n-1}\right\} =\sp\left\{ v_{1},\ldots,v_{n-1}\right\} =\perpen v$.
By orthonormality and part (\ref{enu: first columns span v perp})
of Lemma \ref{lem: characterize G_v}, $\phi_{n}=\dir v=v/\norm v$.
Since $k$ and $k^{\prime}$ have the same last column, then $\hat{v}$
is also the last column of $k^{\prime}$, i.e. $k^{\prime}=k_{\hat{v}}^{\prime}$,
which proves (i). 

It is clear that if $a=\lildiag{a_{1},\ldots,a_{n}}$, then ${\scriptstyle \prod_{1}^{i}}a_{j}=\left\Vert v_{1}\wedge\cdots\wedge v_{i-1}\wedge v_{i}\right\Vert =\covol{\lat_{g}^{i}}$.
 Since $g\in G_{v}$, and $a$ has determinant $1$, we get that
the last diagonal entry of $a$ (hence also of $a^{\prime}$) is $1/\norm v$.
This proves (ii) and (iii). In particular $a^{\prime}=\lildiag{\norm v^{1/\left(n-1\right)},\dots,\norm v^{1/\left(n-1\right)},\norm v^{-1}}$.

Write  $g\brac{n^{\prime}}^{-1}\brac{a^{\prime}}^{-1}=k^{\prime}g^{\dprime}=q$;
right multiplication by an element of $N^{\prime}$ does not change
the first $n-1$ columns of $g$, and right multiplication by $\brac{a^{\prime}}^{-1}$
multiplies these columns by $\norm v^{-1/\left(n-1\right)}$. This
proves (iv), and (v), (vi) immediately follow. 

Write $w$ as the sum of its projections to the orthogonal spaces
$\RR v$ and $\perpen v$: $w=w^{v}+w^{\perpen v}$. Observe that
$g\brac{n^{\prime}}^{-1}=kp^{\dprime}a^{\prime}$. The last column
of $kp^{\dprime}a^{\prime}$ is $\phi_{n}/\norm v$, where from the
calculation on $k^{\prime}$ we know that $\phi_{n}=\dir v$; by Lemma
\ref{lem: v-comp of w}, we get that the last column of $kp^{\dprime}a^{\prime}$
is $w^{v}$. The last column of $g\left(n^{\prime}\right)^{-1}$ is
$w-\sum_{i=1}^{n-1}x_{i}v_{i}$, so we conclude that $w-\sum_{i=1}^{n-1}x_{i}v_{i}=w^{v}=w-w^{\perpen v}$,
which implies  (vii). 
\end{proof}

\section{Fundamental domains representing spaces of lattices, shapes and directions\label{sec: Fundamental domains}}

In this section we find ``isomorphic'' copies of the spaces $\shapespace{n-1}$,
$\unilatspace{n-1}$, $\latspace{n-1,n}$, $\sphere{n-1}$ inside
$\sl n\left(\RR\right)$. The property we are after in these isomorphic
copies, is that the images of sets satisfying a boundary condition,
will also satisfy it. This boundary condition is the following:
\begin{defn}
\label{def: BCS}A subset $B$ of an orbifold $\manifold$ will be
called \emph{boundary controllable set}, or a BCS, if for every $x\in\mathcal{M}$
there is an open neighborhood $U_{x}$ of $x$ such that $U_{x}\cap\del B$
is contained in a finite union of embedded $C^{1}$ submanifolds of
$\mathcal{M}$, whose dimension is strictly smaller than $\dim\manifold$.
In particular, $B$ is a BCS if its (topological) boundary consists
of finitely many  subsets of embedded $C^{1}$ submanifolds. 
\end{defn}

The goal of this section is to prove the following:
\begin{prop}
\label{prop: spread models that we need}There exist full sets of
representatives in $\sl n\left(\RR\right)$: 
\begin{itemize}
\item $K^{\prime}\subset K$ parameterizing $\sphere{n-1}\cong K^{\dprime}\backslash K$
\item $\groupfund{n-1}\subset G^{\dprime}$ parameterizing $\unilatspace{n-1}=G^{\dprime}/G^{\dprime}\left(\ZZ\right)$
\item $\symfund{n-1}\subset P^{\dprime}$ parameterizing $\shapespace{n-1}\cong K^{\dprime}\backslash G^{\dprime}/G^{\dprime}\left(\ZZ\right)$
\item $K^{\prime}\groupfund{n-1}\subset\wc$ parameterizing $\latspace{n-1,n}\cong\wc/G^{\dprime}\left(\ZZ\right)$
\end{itemize}
that are BCS's and with the properties that (i) a BCS is parameterized
by a BCS and vice versa; for $K^{\prime}$, a product of BCS's in
$K^{\prime}$ and $K^{\dprime}$ is a BCS in $K$. (ii) The pullbacks
of the invariant measures on the parameterized spaces to their set
of representatives coincide with the measures that the sets of representatives
inherent from their ambient manifolds: for all cases but $K^{\prime}$
it is the restriction of the  measure $\mu_{\GIcomp}$ on the ambient
manifold, and for $K^{\prime}$ it is the measure $\mu_{K^{\prime}}$
defined in Subsection \ref{subsec: Defining RI}.  
\end{prop}

A full proof of Proposition \ref{prop: spread models that we need}
can be found in \cite[Prop. 8.1]{HK_WellRoundedness} Here, we will
only prove it fully for $K^{\prime}$ and $\sphere{n-1}$ (this case
is easier since $K^{\dprime}/K$ is compact), and for the remaining
spaces we will settle for constructing fundamental domains that are
BCS, with the property that the Haar measure restricted to them coincides
with the unique (up to a scalar) invariant measure on the quotient
(which is the space parameterized by the fundamental domain in question).
We start by constructing sets of representatives for the sphere (Subsection
\ref{subsec: K'}), then for the spaces of lattices (Subsection \ref{subsec: Siegel reduced bases}),
and we conclude with a partial proof of Proposition \ref{prop: spread models that we need}
in Subsection \ref{subsec: Spread models}.

\subsection{\label{subsec: K'} A set of representatives for the sphere }

In order to construct a set of representatives $K^{\prime}$ for $\sphere{n-1}$,
we observe the following. 
\begin{fact}
\label{rem: intersection and union of BCS}Since $\del\left(A\cup B\right),\del\left(A\cap B\right)\subseteq\del A\cup\del B$,
the union, intersection and subtraction of BCSs are in themselves
BCS's. Also, a finite product of BCSs is a BCS in the product of
the ambient manifolds, and a diffeomorphic image of a BCS is a BCS.
\end{fact}

Now the existence of a transversal $K^{\prime}$ for $\sphere{n-1}$
is a consequence of the lemma below. 
\begin{lem}
\label{lem: K'}Let $K$ be a Lie group. Assume that $K^{\dprime}<K$
a closed subgroup such that the quotient space $K/K^{\dprime}$ is
compact. There exists subset $K^{\prime}\subseteq K$ which is a BCS
such that:
\begin{enumerate}
\item $\pi|_{K^{\prime}}:K^{\prime}\to K/K^{\dprime}$ is a bijection;
\item if $\sphereset\subseteq K/K^{\dprime}$ and $\Kdprimeset\subseteq K^{\dprime}$
are BCS, then the product $\exd{\pi|_{K^{\prime}}^{-1}\brac{\sphereset}}{\subset K^{\prime}}\cdot\exd{\Kdprimeset}{\subset K^{\dprime}}$
in $K$ is also a BCS.
\end{enumerate}
\end{lem}

\begin{proof}[Proof of Lemma \ref{lem: K'}]
Since $\pi:K\to K/K^{\dprime}$ is a principal $K^{\dprime}$ fiber
bundle, there exists an open covering $\left\{ U_{\alpha}\right\} $
of $K/K^{\dprime}$ with $K^{\dprime}$-equivariant diffeomorphisms
\[
\tau_{\alpha}:\pi^{-1}\brac{U_{\alpha}}\to U_{\alpha}\times K^{\dprime},
\]
where $\tau_{\alpha}(x)=(\pi(x),*)$. We can assume that there is
a BCS covering $\left\{ W_{\a}\right\} $ of $K/K^{\dprime}$ such
that $\overline{W_{\alpha}}\subseteq U_{\alpha}$ (e.g., by reducing
to open balls contained in $U_{\a}$); by compactness, we may also
assume that this covering is finite. Finally, by replacing every $W_{\a}$
with $W_{\alpha}\setminus\cup_{i=1}^{\alpha-1}W_{i}$, we may assume
that the sets $W_{\a}$ are disjoint, maintaining the BCS property
(Remark \ref{rem: intersection and union of BCS}). Set 
\[
K^{\prime}=\sqcup_{\alpha}\tau_{\alpha}^{-1}\left(W_{\alpha}\times\id_{K^{\dprime}}\right)
\]
(note that the interior is a manifold). Since the union is disjoint,
$\pi|_{K^{\prime}}:K^{\prime}\to K/K^{\dprime}$ is a bijection. Moreover,
since $W_{\alpha}$ is a BCS, then so does $W_{\alpha}\times\id_{K^{\dprime}}$,
and then so does $\tau_{\alpha}^{-1}\left(W_{\alpha}\times\id_{K^{\dprime}}\right)$;
by Remark \ref{rem: intersection and union of BCS}, $K^{\prime}$
is a BCS. 

Finally, by definition of $K^{\prime}$ one has that $k^{\prime}\in U_{\a}\cap K^{\prime}$
maps under $\tau_{\a}$ to $\left(\pi\brac{k^{\prime}},\id_{K^{\dprime}}\right)$.
If $\sphereset\subseteq K/K^{\dprime}$ and $\sphereset\subseteq K^{\dprime}$
then
\[
\pi|_{K^{\prime}}^{-1}\left(\sphereset\cap W_{\a}\right)\cdot\sphereset=\tau_{\a}^{-1}\left(\left(\sphereset\cap W_{\a}\right)\times\sphereset\right),
\]
where by Remark \ref{rem: intersection and union of BCS} the right
hand side is a BCS. Then $\pi|_{K^{\prime}}^{-1}\brac{\sphereset}\cdot\sphereset$
is a BCS, as a finite union of such. 
\end{proof}

\subsection{\label{subsec: Siegel reduced bases}Fundamental domains for $\protect\sl m\left(\protect\ZZ\right)$}

We recall a construction for fundamental domains for the $\sl m\left(\ZZ\right)$
action on $\sl m\left(\RR\right)$ and on $\so m\left(\RR\right)\backslash\sl m\left(\RR\right)$),
and list some of their properties.
\begin{defn}
\label{def: Siegel reduced basis}Let $\left\{ v_{1},\ldots,v_{m}\right\} $
be a basis for $\RR^{m}$, and let $\left\{ \phi_{1},\ldots,\phi_{m}\right\} $
be the orthonormal basis obtained from it by the Gram-Schmidt orthogonalization
algorithm. We say that $\left\{ v_{1},\ldots,v_{m}\right\} $ is \emph{reduced}
if 

\begin{enumerate}
\item the projection of $v_{j}$ to $V_{j-1}^{\perp}$ has minimal non-zero
length $a_{j}$ (here $V_{0}=\left\{ 0\right\} $), where $V_{j-1}=\sp_{\RR}\left\{ v_{1},\ldots,v_{j-1}\right\} $; 
\item the projection of $v_{j}$ to $V_{j-1}$ is $\sum_{i=1}^{j-1}n_{i,j}a_{i}\cdot\phi_{i}$
with $\left|n_{ij}\right|\leq\frac{1}{2}$ for all $i=1,\ldots,j-1$. 
\end{enumerate}
An $m\times m$ matrix with a reduced basis in its columns is also
called reduced.
\end{defn}

Observe that if a real $m\times m$ matrix $g$ is reduced, then it
lies in $\sl m\left(\RR\right)$ and satisfies $g=kan$ where $k=\sbrac{\begin{matrix}\phi_{1} & \cdots & \phi_{m}\end{matrix}}$,
$a=\diag{a_{1},\ldots,a_{m}}$ and $n=\left[\begin{smallmatrix}1 & n_{i,j}\\
0 & 1
\end{smallmatrix}\right]$, with $\phi_{j}$, $a_{j}$ and $n_{i,j}$ as in the definition above.
In particular, whether $g$ is reduced or not, depends only on $an$.
By the work of Siegel \cite{Bekka_Mayer}, the set of reduced matrices
contains a fundamental domain for the action of $\sl m\left(\ZZ\right)$.
A specific choice of such a domain was made by Schmidt \cite{Schmidt_98}
(see also \cite{Grenier_93}), and it is defined as follows; we will
use the notation  $\sym^{+}\left(\lat\right)$ for the group of orientation
preserving isometries of $\lat$ (sometimes referred to a the ``point
group'' of $\lat$).

\begin{figure}
\begin{centering}
\includegraphics[scale=0.5]{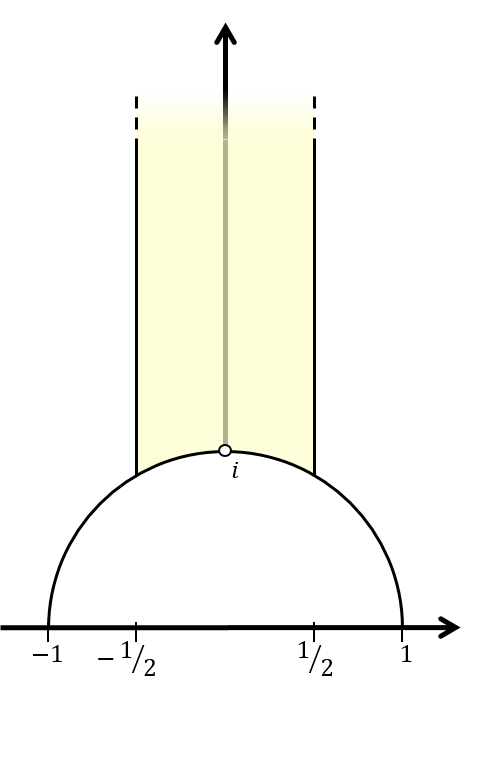}
\par\end{centering}
\caption{\label{fig: fund dom SL(2,Z)} $\protect\symfund 2$: a fundamental
domain for $\protect\sl 2\left(\protect\ZZ\right)$ in $P_{2}$ (the
hyperbolic upper half plane).}
\end{figure}

\begin{defn}
\label{def: Fund dom F_m}We let $\groupfund m\subset\sl m\left(\RR\right)=\so m\left(\RR\right)P_{m}$,
where $P_{m}$ is the subgroup consisting of upper triangular matrices,
denote a choice of a fundamental domain lying inside the set of $g=kan\in\sl m\left(\RR\right)$
such that: (i) $g$ is reduced; (ii) $n_{1,j}\geq0$ for $j>\sn\left(m\right)$
(see Notation (\ref{notation: sign of n})); (iii) $k$ lies inside
a fundamental domain of $\sym^{+}\left(\lat_{an}\right)<\so m\left(\RR\right)$,
where $\lat_{an}$ is the lattice spanned by the columns of $an$.
The projection of $\groupfund m$ to $P_{m}$ is denoted $\symfund m$
(Figure \ref{fig: fund dom SL(2,Z)}).
\end{defn}

Note that conditions (i) and (ii) are on $an$, whereas condition
(iii) is on $k$. Thus, the projection of $\groupfund m$ to $P_{m}\cong\so m\left(\RR\right)\backslash\sl m\left(\RR\right)$
is a fundamental domain for the action of $\sl m\left(\ZZ\right)$
on $\so m\left(\RR\right)\backslash\sl m\left(\RR\right)$ that lies
inside the set of triangular $m\times m$ matrices satisfying conditions
(i) and (ii) in Definition \ref{def: Fund dom F_m}, and the relation
between $\groupfund m$ and $\symfund m$ is given by: 
\begin{prop}
\label{prop: F of G'' from F of H}\label{def: Notation for K_z}The
relation between the fundamental domains $\groupfund m$ and $\symfund m$
is given by
\[
\groupfund m=\bigcup_{z\in\symfund m}K_{z}\cdot z,
\]
where $K_{z}$ is a fundamental domain for the finite group $\sym^{+}\left(\lat_{z}\right)$.
\end{prop}

Note that $\groupfund m$ is not a product of $\symfund m\subset P_{m}$
with a subset of $\so m\left(\RR\right)$, since different lattices
$\lat_{z}$ have different point groups $\sym^{+}\left(\lat_{z}\right)$.
However, there is only a finite number (that depends on $m$) of possible
fibers, since there are finitely many possible symmetry groups for
lattices in $\RR^{m}$. Moreover, for generic $z$'s the point groups
are identical:

\begin{prop}[\cite{Schmidt_98}]
\label{prop: generic fiber}For $z\in\interior{\symfund m}$,  $\sym^{+}\left(\lat_{z}\right)=Z\left(\so m\left(\RR\right)\right)$,
the center of $\so m\left(\RR\right)$.
\end{prop}

Thus suggests that for a full-measure set of $z\in\symfund m$, a
uniform fiber in $K_{m}$ can be chosen; hence $\groupfund m$ can
be approximated by $\symfund m$ times that generic fiber. 
\begin{lem}
\label{lem: lift to G''}Let $G=\sl n\left(\RR\right)$ and $P<G$
the subgroup of upper triangular matrices. Assume $\ensuremath{\uniset}\subseteq\latspace n$
is the lift of $\symset\subseteq\shapespace n$. If $\symset$ is
a BCS then $\ensuremath{\uniset}$ is, and $\text{\ensuremath{\mu_{\latspace n}}\ensuremath{\brac{\ensuremath{\uniset}}}}=\mu_{\shapespace n}\brac{\symset}\cdot{\scriptstyle \prod}_{i=1}^{n-1}\Leb{\sphere i}/\sn\left(n\right)$.
Assume $\latset\subseteq\latspace{n-1,n}$ projects to $\ensuremath{\uniset}\subseteq\unilatspace{n-1}$
and $\sphereset\subseteq\sphere{n-1}$, in the sense that $\parby{\wc}{\latset}=\parby{K^{\prime}}{\sphereset}\parby{G^{\dprime}}{\ensuremath{\uniset}}$
(e.g. if $\latset$ is the inverse image of $\ensuremath{\uniset}$).
If $\ensuremath{\uniset}$ and $\sphereset$ are BCS's, then so is
$\parby{\wc}{\latset}$, and $\mu_{\latspace{n,n-1}}\brac{\latset}=\ensuremath{\mu_{\latspace m}}\brac{\ensuremath{\uniset}}\mu_{\sphere{n-1}}\brac{\sphereset}$. 
\end{lem}

\begin{proof}
By Proposition \ref{prop: F of G'' from F of H}, $\parby G{\ensuremath{\uniset}}=\bigcup_{z\in P_{\symset}}K_{z}\cdot z$.
Since there are only finitely many possible fibers, then by Proposition
\ref{prop: generic fiber}
\[
\parby G{\ensuremath{\uniset}}=(K_{\text{gen}}\cdot\brac{\parby P{\symset}\cap\interior{\symfund n}})\cup(\bigcup_{i=1}^{\topindex\left(n\right)}K_{z_{i}}\cdot\left\{ z\in\parby P{\symset}\cap\del\symfund m:\sym^{+}\left(\lat_{z}\right)=\sym^{+}\left(\lat_{z_{i}}\right)\right\} )
\]
where $K_{\text{gen}}$ is the generic fiber and $\left\{ z\in\parby P{\symset}\cap\del\symfund n:\sym^{+}\brac{\lat_{z}}=\sym^{+}\brac{\lat_{z_{i}}}\right\} $
is contained in $\del\symfund n$, and is therefore a BCS of measure
zero in $P$. Since the fibers in $\so n\left(\RR\right)$ are BCS's
in $\so n\left(\RR\right)$ due to Lemma \ref{lem: K'}, and since
$\parby P{\symset}\cap\interior{\symfund n}$ is a BCS by Proposition
\ref{prop: spread models that we need} and Fact \ref{rem: intersection and union of BCS},
and since $\so n\left(\RR\right)\times P$ is diffeomorphic to $\sl n\left(\RR\right)$
with $\mu_{\sl n\left(\RR\right)}=\mu_{\so n\left(\RR\right)}\times\mu_{P}$,
we have that $\parby G{\ensuremath{\uniset}}$ is a BCS in $\sl n\left(\RR\right)$
and has the same measure as $K_{\text{gen}}\cdot\brac{\parby P{\symset}\cap\interior{\symfund n}}$,
which is $\mu_{\so n\left(\RR\right)}\brac{K_{\text{gen}}}\cdot\mu_{P}\brac{\parby P{\symset}}=\mu_{\so n\left(\RR\right)}\brac{\so n\left(\RR\right)}\mu_{P}\brac{\parby P{\symset}}/\sn\left(n\right)=\mu_{P}\brac{\parby P{\symset}}\cdot{\scriptstyle \prod}_{i=1}^{n-1}\Leb{\sphere i}/\sn\left(n\right)$
(recall choice of the volume of $\so n\left(\RR\right)$ in Subsection
\ref{subsec: Defining RI}). According to Proposition \ref{prop: spread models that we need},
which says that BCS's and the measures in the ``good'' sets of representatives
and in the spaces that they represent correspond, we get that we get
that $\ensuremath{\uniset}$ is a BCS and that $\text{\ensuremath{\mu_{\unilatspace n}}\ensuremath{\brac{\ensuremath{\uniset}}}}=\mu_{\shapespace n}\brac{\symset}\cdot{\scriptstyle \prod}_{i=1}^{n-1}\Leb{\sphere i}/\sn\left(n\right)$.

The proof for $\latset$ is a direct consequence of \cite[Propositions 6.15 and 6.16]{HK_WellRoundedness}
\end{proof}
For future reference, we list some properties of $\groupfund m,\symfund m$
that will be useful in the proof of our main theorem; in fact, the
following applies to every reduced matrix, and in particular to the
elements of $\groupfund m,\symfund m$. The notations for $a_{j}$
and $V_{j}$ are as in Definition \ref{def: Siegel reduced basis}.
 
\begin{lem}
\label{lem: BLC. facts about z in RS domain}Suppose $g=kan$ is reduced
and that its columns span a lattice $\lat$. Then

\begin{enumerate}
\item \label{enu: entries of n in =00005B-0.5,0.5=00005D}$n$ is a unipotent
upper triangular matrix with non-diagonal entries in $\left[-1/2,1/2\right]$;
in particular, $\norm{n^{\pm1}},\norm{n^{\pm\transpose}}\porsmall1$. 
\item \label{enu: entries of a increasing}$a=\diag{a_{1},\ldots,a_{m}}$
satisfies that $a_{1}\porsmall\cdots\porsmall a_{m}$. Specifically,
$\frac{\sqrt{3}}{2}a_{j}\leq a_{j+1}$. 
\item \label{enu: norm out of E_(j-1)}If $\lm\in\lat$  satisfies $\lm\notin V_{j-1}$,
then $\left\Vert \lm\right\Vert \geq\dist{\lm,V_{j-1}}\geq\dist{v_{j},V_{j-1}}=a_{j}$.
\item \label{enu: norm in E_j}If $x\in V_{j}$, then $\left\Vert ax\right\Vert \porsmall a_{j}\left\Vert x\right\Vert $. 
\end{enumerate}
\end{lem}

\subsection{Relation between fundamental domains and quotient spaces\label{subsec: Spread models}}

In order to deduce that the invariant measures on the fundamental
domains $K^{\prime}$, $\groupfund{n-1}$, $\symfund{n-1}$ etc.\
are the Haar measures on the spaces that they represent, we require
the following result: 
\begin{thm}[{\cite[Thm 2.2]{Jus18}}]
\label{thm: measure thing}Let $G$ be a unimodular Radon lcsc group
with a Haar measure $\mu_{G}$, and let $\nu$ be a $G$-invariant
Radon measure on an lcsc space $Y$. Assume that the $G$ action on
$Y$ is strongly proper. Then there exists a unique Radon measure
$\overline{\nu}$ on $G\backslash Y$ such that for all $f\in L^{1}\left(Y,\nu\right)$,
\[
\int_{Y}f\left(y\right)d\nu\left(y\right)=\int_{G\backslash Y}\left(\int_{G}f\left(gy\right)d\mu_{G}\left(g\right)\right)d\overline{\nu}\left(Gy\right).
\]
\end{thm}

\begin{proof}[Proof of Proposition \ref{prop: spread models that we need}]
By construction, $K^{\prime}$, $\groupfund{n-1}$ and $\symfund{n-1}$
are sets of representatives for $\sphere{n-1}$, $\unilatspace{n-1}$
and $\shapespace{n-1}$ respectively, and $K^{\prime}\subset K$ is
a BCS according to Lemma \ref{lem: K'}. $\symfund{n-1}\subset P^{\dprime}$
is a BCS since its boundary is contained in a finite union of lower-dimeansional
manifolds in $P^{\dprime}$ (see \cite[pp. 48-49]{Schmidt_98}, and
$\groupfund{n-1}\subset G^{\dprime}$ is a BCS by Lemma \ref{lem: lift to G''}.
Finally, $K^{\prime}\groupfund{n-1}\subset\wc$, it is a set of representatives
for $\latspace{n-1,n}$ since 
\[
\sl n\left(\RR\right)/N^{\prime}G^{\dprime}\left(\ZZ\right)A^{\prime}\diffeo K^{\prime}G^{\dprime}A^{\prime}N^{\prime}/G^{\dprime}\left(\ZZ\right)A^{\prime}N^{\prime}\diffeo K^{\prime}G^{\dprime}/G^{\dprime}\left(\ZZ\right)
\]
and $\groupfund{n-1}$ is a set of representatives for $G^{\dprime}/G^{\dprime}\left(\ZZ\right)\cong\unilatspace{n-1}$.\textcolor{red}{{}
} It is a BCS by Lemma \ref{lem: lift to G''}. For part (i) of the
proposition, a BCS in $\sphere{n-1}$, $\unilatspace{n-1}$, $\shapespace{n-1}$
or $\latspace{n-1,n}$ is mapped to a BCS in $K^{\prime}$, $\groupfund{n-1}$,
$\symfund{n-1}$ and $K^{\prime}\times\groupfund{n-1}$ respectively:
for $K^{\prime}$ it holds because of Lemma \ref{lem: K'}, and for
the remaining sets this is proved in \cite[Prop. 8.1]{HK_WellRoundedness}.
The correspondence of measures is a consequence of Theorem \ref{thm: measure thing}
above (but one can find more details in \cite[Prop. 6.10]{HK_WellRoundedness}).
\end{proof}

\section{\label{sec: Integral matrices representing primitive vectors}Integral
matrices representing primitive vectors }

We begin in Subsection \ref{subsec: Correspondence: SL(n,Z) <--> primitive vectors (Q and Q(Z))}
by establishing a $1$ to $1$ correspondence between primitive vectors
in $\ZZ^{n}$ and integral matrices in fundamental domains for the
discrete subgroup defined as 
\[
\disgrp:=\left(N^{\prime}\rtimes G^{\dprime}\right)\left(\ZZ\right)=\left[\begin{array}{cc}
\sl{n-1}\left(\ZZ\right) & \ZZ^{n}\\
0 & 1
\end{array}\right].
\]
Then, in Subsection \ref{subsec: Fundamental domains for Q(Z)}, we
define an explicit such fundamental domain in which the integral representative
of a primitive vector $v$, has the shortest solution $w_{v}$ in
its last column. 

\subsection{\label{subsec: Correspondence: SL(n,Z) <--> primitive vectors (Q and Q(Z))}Correspondence
between primitive vectors and matrices in $\protect\sl n\left(\protect\ZZ\right)$}

Recall $G_{v}$ was defined in Formula \ref{eq: def of G_v}. We first
prove: 
\begin{prop}
\label{prop: primitive vectors correspond to integral matrices}If
$\funddom\subset\sl n\left(\RR\right)$ is a fundamental domain for
the right action of $\disgrp$, then there exists a bijection that
depends on $\funddom$
\[
\left(\ZZ^{n}\cap\perpen v\right)\leftrightarrow v\leftrightarrow\ga_{v}\left(\funddom\right):=\mbox{the unique element in }\funddom\cap G_{v}\left(\ZZ\right),
\]
between
\[
\left\{ \substack{\mbox{primitive oriented}\\
\mbox{\ensuremath{\left(n-1\right)}-lattices in }\ZZ^{n}
}
\right\} \leftrightarrow\left\{ \substack{\mbox{primitive vectors}\\
\mbox{in }\ZZ^{n}
}
\right\} \leftrightarrow\left\{ \substack{\mbox{integral matrices}\\
\mbox{ in \ensuremath{\funddom}}
}
\right\} .
\]
\end{prop}

\begin{proof}
The correspondence $\left(\ZZ^{n}\cap\perpen v\right)\leftrightarrow v$
is explained in the Introduction, and it suffices to show the correspondence
$v\leftrightarrow\ga_{v}\left(\funddom\right).$ We first claim that
\begin{equation}
G_{v}\cap\sl n\left(\ZZ\right)\neq\left\{ \emptyset\right\} \Longleftrightarrow v\in\ZZ^{n}\mbox{ primitive. }\label{eq: v is primitive iff exists integral point in G_v}
\end{equation}
The direction $\Longrightarrow$ is a consequence of (\ref{enu: g in G_v})$\Rightarrow$(\ref{enu: upper row's projection is 1})
in Lemma \ref{lem: characterize G_v}. Conversely, if $v$ is primitive,
then there exists $w\in\ZZ^{n}$ such that $\left\langle v,w\right\rangle =1$.
Let $\left\{ v_{1},\ldots,v_{n-1}\right\} $ be an integral basis
for $\perpen v$ such that $\left\{ v_{1},\ldots,v_{n-1},v\right\} $
is a positively oriented basis for $\RR^{n}$. Then, by (\ref{enu: upper row's projection is 1})
$\Rightarrow$ (\ref{enu: g in G_v}) in Lemma \ref{lem: characterize G_v},
the resulting matrix $\sbrac{v_{1}|\cdots|v_{n-1}|w}$ is in $G_{v}$.
Since its columns are integral, it is also in $\sl n\left(\ZZ\right)$. 

Observe that $G_{v}$ is an orbit of the group $N^{\prime}\rtimes\brac{G^{\dprime}\left(\ZZ\right)}$,
acting by right multiplication on $G=\sl n\left(\RR\right)$, and
that $\disgrp$ is the subgroup of integral elements in this group.
According to (\ref{eq: v is primitive iff exists integral point in G_v}),
$v\in\ZZ^{n}$ is primitive if and only if there exists an integral
$\ga$ in $G_{v}$. This is equivalent to all the points in the orbit
$\ga\cdot\disgrp$ being integral. Since $\funddom$ is a fundamental
domain for $\disgrp$, the coset $\ga\cdot\disgrp$ intersects $\funddom$
in a single point $\left\{ \ga_{v}\right\} =\funddom\cap\brac{\ga\cdot\disgrp}$.
We claim that $\ga\cdot\disgrp=G_{v}\left(\ZZ\right)$; indeed,
\[
G_{v}\left(\ZZ\right)=G_{v}\cap\sl n\left(\ZZ\right)=\left(\ga\cdot N^{\prime}G^{\dprime}\right)\cap\sl n\left(\ZZ\right)=\ga\cdot\brac{\left(N^{\prime}G^{\dprime}\right)\cap\sl n\left(\ZZ\right)}=\ga\cdot\disgrp.\tag*{\qedhere}
\]
\end{proof}

\subsection{\label{subsec: Fundamental domains for Q(Z)}\label{subsec: Fundamental domain for shortest solutions}A
fundamental domain for $\protect\disgrp$ that captures the shortest
solutions}

Having shown that the primitive vectors in $\RR^{n}$ correspond to
integral matrices in a fundamental domain of $\disgrp$, we proceed
to construct a specific such domain, with the property that every
representative $\ga_{v}$ has in its last column the shortest solution
$w_{v}$ to the $\gcd$ equation of $v$. We begin with a more general
(even if not as general as possible) construction for a fundamental
domain of $\disgrp$; but first, a notation. 
\begin{notation}
For $g\in\sl n\left(\RR\right)$, we let $\zgt g$ denote the upper
triangular $\left(n-1\right)\times\left(n-1\right)$ matrix such that
the $P^{\dprime}$ component of $g$ is $\left[\begin{smallmatrix}\zgt g & 1\\
1 & 0
\end{smallmatrix}\right]$.
\end{notation}

\begin{prop}
\label{cor: Fund dom for Q(Z)}Let $\groupfund{}\subset\sl{n-1}\left(\RR\right)$
be a fundamental domain of $\sl{n-1}\left(\ZZ\right)$, and $\fdomN=\left\{ \domN\left(z\right)\right\} _{z\in\symfund{n-1}}$
be a family of fundamental domains for $\ZZ^{n-1}$ in $\RR^{n-1}$.
Then
\[
\funddom=\funddom_{\fdomN}:=\bigcup_{g^{\dprime}\in\groupfund{}}K^{\prime}\cdot g^{\dprime}\cdot A^{\prime}\cdot\parby{N^{\prime}}{\domN\brac{\zgt{g^{\dprime}}}}
\]
is a fundamental domain for the action of $\disgrp$ on $\sl n\left(\RR\right)$
by multiplication from the right.
\end{prop}

The proof is rather standard, and we skip it. 

\begin{rem}
\label{rem: Fund dom for Q(Z) is product set}Clearly, if all the
domains $\domN\left(z\right)$ are the same domain $\domN$, then
$\funddom$ is the product set $K^{\prime}\parby{G^{\dprime}}{\groupfund{}}A^{\prime}\parby{N^{\prime}}{\domN}$.
\end{rem}

For $g$ in $\sl n\left(\RR\right)$, consider the linear map $\linmap_{g}$
that sends the first $n-1$ columns of $g$ to the (ordered) standard
basis for $\RR^{n-1}$. Note that the $\left(n-1\right)$-lattice
$\lat_{g}^{n-1}$, spanned by the first $n-1$ columns of $g$, is
mapped under $\linmap_{g}$ onto $\ZZ^{n-1}=\sp_{\ZZ}\left\{ e_{1}\ldots,e_{n-1}\right\} $.
As a result, a fundamental domain for $\lat_{g}^{n-1}$ in $\perpen v$
is mapped under $\linmap_{g}$ onto a fundamental domain of $\ZZ^{n-1}$
in $\RR^{n-1}$. We consider the image of the Dirichlet domain for
$\lat_{g}^{n-1}$, which is $\hex{\zgt g}:=\linmap_{g}\brac{\dirdom{\lat_{g}^{n-1}}}$.
Note that indeed the right-hand side depends only on the $P^{\dprime}$
component of $g$: since $g=ka^{\prime}p^{\dprime}n^{\prime}$, then
the RHS is $\linmap_{n^{\prime}}\linmap_{p^{\dprime}}\linmap_{ka^{\prime}}\brac{\dirdom{\lat_{g}^{n-1}}}$.
Now $\linmap_{ka^{\prime}}$ acts as a rotation and multiplication
by scalar such that $\lat_{g}^{n-1}$ maps to $\lat_{p^{\dprime}}^{n-1}=\lat_{\zgt g}$
and the Dirichlet domains map to one another. Since $\linmap_{n^{\prime}}$
is identity map, then $\linmap_{g}\brac{\dirdom{\lat_{g}^{n-1}}}$
equals $\linmap_{p^{\dprime}}\brac{\dirdom{\lat_{p^{\dprime}}^{n-1}}}$.
Then 
\begin{equation}
\fdomhex_{\symfund{n-1}}:=\left\{ \hex z\right\} _{z\in\symfund{n-1}}\label{eq: Hex family}
\end{equation}
is a family of fundamental domains for $\ZZ^{n-1}$ in $\RR^{n-1}$,
and so by Proposition \ref{cor: Fund dom for Q(Z)} and by the notation
for $K_{z}^{\dprime}$ appearing in Proposition \ref{def: Notation for K_z},
the following is a fundamental domain for $\disgrp$: 
\begin{equation}
\shortdom:=\funddom_{\fdomhex}=\bigcup_{g^{\dprime}\in\groupfund{n-1}}K^{\prime}\cdot g^{\dprime}\cdot A^{\prime}N_{\hex{\zgt{g^{\dprime}}}}^{\prime}=\bigcup_{z\in\symfund{n-1}}K^{\prime}K_{z}^{\dprime}\cdot\exd{\left[\begin{smallmatrix}z & 0\\
0 & 1
\end{smallmatrix}\right]}{\in P^{\dprime}}\cdot A^{\prime}N_{\hex z}^{\prime}.\label{eq: Omega_short}
\end{equation}
Recall from Proposition \ref{prop: explicit RI coordinates of g}
that 
\[
g=\left(|\cdots|w\right)\in G_{v}\cap\funddom\brac{\fdomhex}\implies w^{\perpen v}\in\linmap_{g}^{-1}\brac{\hex{\zgt g}}\iff w^{\perpen v}\in\dirdom{\lat_{g}^{n-1}},
\]
namely $w^{\perpen v}$ (such that $w=v/\norm v^{2}+w^{\perpen v}$,
see Lemma \ref{lem: v-comp of w}) is the shortest representative
of the coset $w^{\perpen v}+\lat_{g}^{n-1}$ in the hyperplane $\perpen v$.
This means that $w$ is the shortest representative of the coset $w+\lat_{g}^{n-1}$
(which lies in the affine hyperplane $\left\{ u:\left\langle u,v\right\rangle =1\right\} $).
As a result, for every primitive vector $v$, the representative $\ga_{v}=G_{v}\left(\ZZ\right)\cap\shortdom$,
has last column $w$ which is 
\begin{eqnarray*}
w_{v} & := & \mbox{the shortest integral \ensuremath{w\;}which satisfies \ensuremath{\left\langle w,v\right\rangle =1}}.
\end{eqnarray*}
The relation between the norm of $w_{v}$, which is what we are interested
in for Theorems \ref{thm: |w_v|/|v| to zero} and \ref{thm: MainThm},
and between the norm of $w_{v}^{\perpen v}$, which is what is  captured
by $\shortdom$, is given by the following lemma.
\begin{lem}
\label{lem: w and w perp}If $\{v_{n}\}$ is a divergent sequence
of primitive vectors, then 
\[
\underset{n\to\infty}{\lim}\left|\norm{w_{v_{m}}}-\norm{w_{v_{m}}^{\perpen{v_{m}}}}\right|=O(\norm{v_{m}}^{-1}).
\]
\end{lem}

\begin{proof}
By Lemma \ref{lem: v-comp of w},
\[
|\norm{w_{v_{m}}}-\norm{w_{v_{m}}^{\perpen{v_{m}}}}|\leq\norm{w_{v_{m}}-w_{v_{m}}^{\perpen{v_{m}}}}=\frac{1}{\norm{v_{m}}}.
\]
\end{proof}

\section{\label{sec: Defining a counting problem} Defining a counting problem }

The goal of this section is to reduce the proof of Theorem \ref{thm: MainThm}
into a problem of counting integral matrices in subsets of $\sl n\left(\RR\right)$,
and specifically of $\shortdom$. We begin by defining these subsets.
First, consider the covering radius of the lattice spanned by the
first $n-1$ columns:
\begin{eqnarray*}
\rad\brac{\lat_{g}^{n-1}} & = & \mbox{radius of bounding circle for \ensuremath{\dirdom{\lat_{g}^{n-1}}}}.
\end{eqnarray*}
Clearly, the norm $\Vert w^{\perpen v}\Vert$ lies in the interval
$\sbrac{0,\rad\brac{\lat_{g}}}$, i.e. 
\[
\frac{\Vert w^{\perpen v}\Vert}{\rad\left(\lat_{g}\right)}\in\left[0,1\right].
\]
We consider sub-families $\fdomhex^{\,\a}\subseteq\fdomhex$ for which
this quotient is restricted to a sub-interval $\left[0,\a\right]$,
with $0\leq\a\leq1$. Let $\ball r{\perpen v}$ denote an origin-centered
$n-1$ dimensional ball in $\perpen v$ whose radius is $r$. For
$\a\in\left[0,1\right]$, let
\begin{eqnarray}
\ellipse{\zgt g}{\a} & = & \linmap_{g}\brac{\ball{\a\rad\left(\lat_{g}\right)}{\perpen v}\cap\dirdom{\lat_{g}}}\nonumber \\
\fdomellipse{\a}{\symfund{n-1}} & = & \left\{ \ellipse z{\a}\right\} _{z\in\symfund{n-1}}.\label{eq: family of elipses int hexagons}
\end{eqnarray}
We now turn to define the subsets of $\shortdom$ such that the integral
matrices inside them represent via the bijection $\ga_{v}\leftrightarrow v$
the primitive vectors that are counted in Theorem \ref{thm: MainThm}. 
\begin{notation}
\label{nota: Definitoin of subsets to count in}For $T>0$, $\sphereset\subseteq\sphere{n-1}$,
$\ensuremath{\uniset}\subseteq\unilatspace{n-1}$, $\symset\subseteq\shapespace{n-1}$,
$\latset\subseteq\latspace{n-1,n}$ and $\a\in\left[0,1\right]$,
recall the notation in \ref{eq: Omega_short} and consider
\[
\left(\shortdom\right)_{T}\brac{\sphereset,\symset,\a}=\,\shortdom\,\cap\,\left\{ g=k^{\prime}k^{\dprime}p^{\dprime}a^{\prime}n^{\prime}:\begin{matrix}k^{\prime}\in\parby{K^{\prime}}{\sphereset},a^{\prime}\in\parby{A^{\prime}}{\left[0,T\right]},\\
p^{\dprime}\in\parby{P^{\dprime}}{\symset},n^{\prime}\in\parby{N^{\prime}}{\ellipse{\zgt g}{\a}}
\end{matrix}\right\} =\bigcup_{p^{\dprime}\in\parby{P^{\dprime}}{\symset}}\parby{K^{\prime}}{\sphereset}\cdot K_{z^{p^{\dprime}}}^{\dprime}\cdot p^{\dprime}\cdot A_{T}^{\prime}\parby{N^{\prime}}{\ellipse{\zgt{p^{\dprime}}}{\a}},
\]
\[
\left(\shortdom\right)_{T}\brac{\sphereset,\ensuremath{\uniset},\a}=\,\shortdom\cap\left\{ g=k^{\prime}g^{\dprime}a^{\prime}n^{\prime}:\begin{matrix}k^{\prime}\in\parby{K^{\prime}}{\sphereset},a^{\prime}\in\parby{A^{\prime}}{\left[0,T\right]},\\
g^{\dprime}\in\parby{G^{\dprime}}{\ensuremath{\uniset}},n^{\prime}\in\parby{N^{\prime}}{\ellipse{\zgt g}{\a}}
\end{matrix}\right\} =\bigcup_{g^{\dprime}\in\parby{G^{\dprime}}{\ensuremath{\uniset}}}\parby{K^{\prime}}{\sphereset}\cdot g^{\dprime}\cdot A_{T}^{\prime}\parby{N^{\prime}}{\ellipse{\zgt{g^{\dprime}}}{\a}},
\]
and 
\[
\left(\shortdom\right)_{T}\brac{\latset,\a}=\,\shortdom\,\cap\,\left\{ g=qa^{\prime}n^{\prime}:\begin{matrix}q\in\parby{\wc}{\latset},a^{\prime}\in\parby{A^{\prime}}{\left[0,T\right]},\\
n^{\prime}\in\parby{N^{\prime}}{\ellipse{\zgt g}{\a}}
\end{matrix}\right\} =\bigcup_{q\in\parby{\wc}{\latset}}q\cdot A_{T}^{\prime}\parby{N^{\prime}}{\ellipse{\zgt q}{\a}}.
\]
\end{notation}

\begin{figure}
\begin{centering}
\includegraphics[scale=0.5]{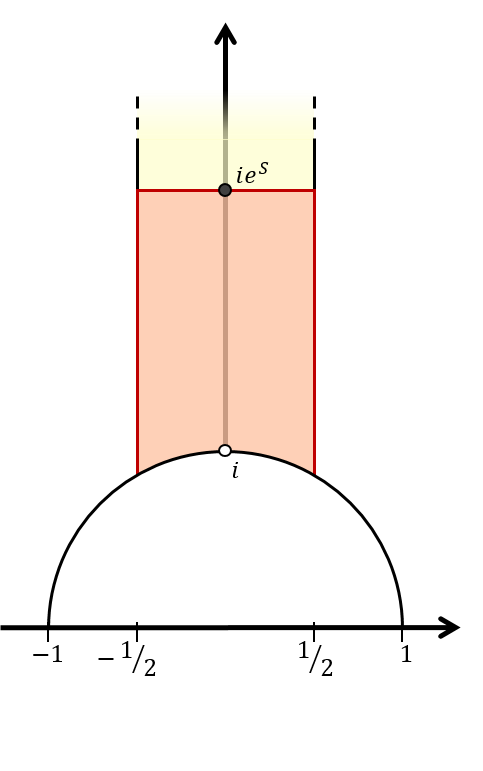}
\par\end{centering}
\caption{\label{fig: Fund dom for SL(2,Z): Truncated } The domain $\protect\trunc{\protect\symfund 2}S$. }
\end{figure}

The following notation is for sets in $G$ whose $A^{\dprime}$ component
is restricted to a compact box. 
\begin{notation}
\label{nota: truncated}For every $\Svec=\brac{S_{1},\ldots,S_{n-2}}>\underline{0}$
and a subset $\Gset\subset G$, let $\trunc{\Gset}{\Svec}$ denote
the subset $\Gset\cap\cbrac{g:\pi_{A_{i}^{\prime\prime}}\left(g\right)\leq S_{i}\,\forall i}$,
where $\pi_{A_{i}^{\prime\prime}}$ is the projection to the $A_{i}^{\prime\prime}$
component (see Figure \ref{fig: Fund dom for SL(2,Z): Truncated }
for $\trunc{\symfund 2}S$).
\end{notation}

Recall from the Introduction that $\lat_{v}=\lat_{\ga_{v}}^{n-1}$
and $\rad_{v}=\rad\brac{\lat_{v}^{n-1}}$. The following is now immediate
from Proposition \ref{prop: primitive vectors correspond to integral matrices},
Proposition \ref{prop: explicit RI coordinates of g}, and the construction
of $\shortdom$, and concludes the translation of Theorem \ref{thm: MainThm}
into a problem of counting lattice points in $\sl n\left(\RR\right)$: 
\begin{cor}
\label{cor: the sets we should count in}Consider the correspondence
$v\leftrightarrow\ga_{v}$ where $\ga_{v}=\left(G_{v}\left(\ZZ\right)\right)\cap\shortdom$.
For $T>0$, $\sphereset\subseteq\sphere{n-1}$ $\ensuremath{\uniset}\subseteq\unilatspace{n-1}$
, $\symset\text{\ensuremath{\subseteq\shapespace{n-1}}}$, $\latset\subseteq\latspace{n-1,n}$
and $\a\in\left[0,1\right]$:
\begin{enumerate}
\item The $\sl n\left(\ZZ\right)$ matrices in $\left(\shortdom\right)_{T}\brac{\sphereset,\symset,\a}$
correspond under $\ga_{v}\leftrightarrow v$ to the elements of 
\[
\left\{ v\in\ZZ^{n}\mbox{ primitive}:\norm v\leq e^{T},\hat{v}\in\sphereset,\shape{\lat_{v}}\in\symset,\lilnorm{w^{v^{\perp}}}/\rad_{v}\in\left[0,\a\right]\right\} .
\]
\item The $\sl n\left(\ZZ\right)$ matrices in $\left(\shortdom\right)_{T}\brac{\sphereset,\ensuremath{\uniset},\a}$
correspond under $\ga_{v}\leftrightarrow v$ to the elements of 
\[
\left\{ v\in\ZZ^{n}\mbox{ primitive}:\norm v\leq e^{T},\hat{v}\in\sphereset,\unisimlat{\lat_{v}}\in\ensuremath{\uniset},\lilnorm{w^{v^{\perp}}}/\rad_{v}\in\left[0,\a\right]\right\} .
\]
\item The $\sl n\left(\ZZ\right)$ matrices in $\left(\shortdom\right)_{T}\brac{\latset,\a}$
correspond under $\ga_{v}\leftrightarrow v$ to the elements of
\[
\left\{ v\in\ZZ^{n}\mbox{ primitive}:\norm v\leq e^{T},\unilat{\lat_{v}}\in\latset,\lilnorm{w^{v^{\perp}}}/\rad_{v}\in\left[0,\a\right]\right\} .
\]
\end{enumerate}
For the cases where $\symset$ (resp.\  $\ensuremath{\uniset}$,
$\latset$) is the set parametertized by $\trunc{\symfund{n-1}}{\Svec}$
(resp.\  $\trunc{\groupfund{n-1}}{\Svec}$, $K^{\prime}\trunc{\groupfund{n-1}}{\Svec}$)
with $\Svec=\brac{S_{1},\dots,S_{n-2}}$, then the condition $\shape{\lat_{v}}\in\symset$
(resp.\ $\unisimlat{\lat_{v}}\in\ensuremath{\uniset}$, $\unilat{\lat_{v}}\in\latset$)
is equivalent to $\frac{\left\Vert \left(v_{1}\wedge\cdots\wedge v_{i}\right)\right\Vert }{\norm v^{i/\left(n-1\right)}}\geq e^{-S_{i}/2}$
for every $1\leq i\leq n-2$. 
\end{cor}

\section{\label{sec: Number of integral points up the cusp}Simplifying the
counting problem by restricting to compacts}

In the previous section we reduced the proof of Theorem \ref{thm: MainThm}
to counting $\sl n\left(\ZZ\right)$ points inside the subsets $\left(\shortdom\right)_{T}$
as $T\to\infty$. These sets have the disadvantage of not being compact,
despite their finite volume; this is apparent from the fact that they
contain $A^{\dprime}$, which is unbounded. Since our counting method
(described in Subsection \ref{subsec: GN method}) does not allow
non-compact sets, the aim of this section is to reduce counting in
$\left(\shortdom\right)_{T}$ to counting in a compact subset of it.
Here we will allow a fundamental domain of $\disgrp$ as general as
in Corollary \ref{cor: Fund dom for Q(Z)}, and not restrict just
to $\shortdom$.
\begin{notation}
\label{nota: def of H(T,S)}Let $\funddom$ be a fundamental domain
for $\disgrp$ in $\sl n\left(\RR\right)$ as in Corollary \ref{cor: Fund dom for Q(Z)}.
For every $T,\underline{S}>0$, define (in accordance with Notation
\ref{nota: truncated}) 
\[
\funddom_{T}^{\underline{S}}:=\funddom\cap\left\{ g=ka_{t}^{\prime}a_{\svec}^{\dprime}n:\begin{matrix}s_{j}\leq S_{j},\:t\in\left[0,T\right]\end{matrix}\right\} .
\]
Note that $\mu\brac{\funddom_{T}-\funddom_{T}^{\underline{S}}}$ is
in $O\brac{e^{nT-S_{\min}}}$, where $S_{\min}=\min_{j}S_{j}$. 
\end{notation}

The goal of this section is to prove the following:
\begin{prop}
\label{prop: very few SL(n,Z) points up the cusp}Let $\funddom$
a fundamental domain for $\disgrp$ in $G$, and $\underline{\s}=\left(\s_{1},\ldots,\s_{n-2}\right)$
where $0<\s_{i}<1$ $\forall i$. Denote $\funddom_{T}^{\left[\underline{\s}T;\infty\right]}:=\funddom_{T}-\funddom_{T}^{\underline{\s}T}$
and $\sigma_{\min}=\min\left(\sigma_{1},\dots,\sigma_{n-1}\right)$.
Then for every $\e>0$
\[
\#\brac{\,\funddom_{T}^{\left[\underline{\s}T;\infty\right]}\cap\sl n\left(\ZZ\right)}=O_{\e}\,\brac{e^{T\left(n-\sigma_{\min}+\e\right)}}.
\]
\end{prop}

Two auxiliary claims are required for the proof. 

\begin{figure}
\hspace*{\fill}\subfloat[\label{fig: Butz Lemma - in plane v perp}$v_{1}$, $v_{2}$ and $a_{2}$
in the plane $\protect\sp_{\protect\RR}\left\{ v_{1},v_{2}\right\} =\protect\sp_{\protect\RR}\left(\protect\lat\right)$]{\includegraphics[scale=0.6]{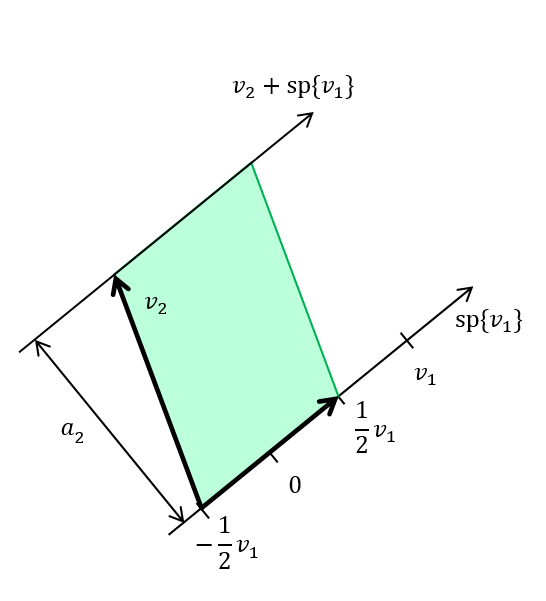}}\hspace*{\fill}\subfloat[{\label{fig: Butz Lemma - cylinder}A Dirichlet domain for $\tilde{\protect\lat}_{1}$,
$\left[-1/2,1/2\right]v_{1}$, multiplied by a ball of radius $R_{2}$.
}]{\includegraphics[scale=0.55]{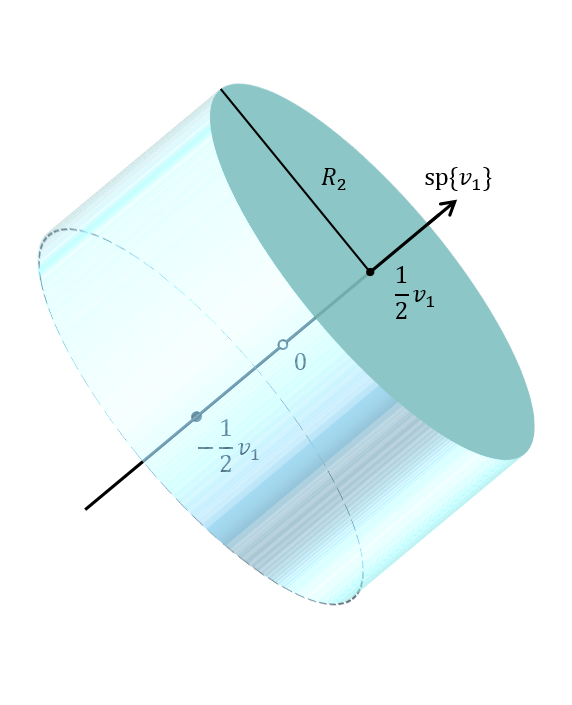}}\hspace*{\fill}

\caption{Lemma \ref{lem: Butz's Lemma}.}
\end{figure}

\begin{lem}
\label{lem: Butz's Lemma}Assume that $\lat_{\#}$ is a full lattice
in $\RR^{n}$ with covolume $\mathbf{v}$, and set $1\leq d\leq n$.
Let there be $d-1$ intervals $\left[\a_{i},\b_{i}\right]$ with $0\leq\a_{i}<\b_{i}$.
The number of rank $d$ subgroups of $\lat_{\#}$ whose covolume is
$\leq\cov$ and who satisfy that $\covol{\lat^{i}}\in\sbrac{\cov^{\a_{i}},\cov^{\b_{i}}}$
for some reduced basis $\cbrac{v_{1},\ldots,v_{d}}$ such that $\lat^{i}:=\sp_{\ZZ}\cbrac{v_{1},\ldots,v_{i}}$,
is $O_{n}\brac{\mathbf{v}^{-d}\cov^{e\left(\underline{\a},\underline{\b}\right)}}$,
where 
\[
e\brac{\underline{\a},\underline{\b}}=n-d+1+2\sum_{i=1}^{d-1}\b_{i}+\sum_{i=1}^{d-1}\brac{n-i}\brac{\b_{i}-\a_{i}}.
\]
\end{lem}

\begin{proof}
Let $\lat<\lat_{\#}$ be a rank $d$ subgroup, and write $\lat=\sp_{\ZZ}\left\{ v_{1},\ldots,v_{n-1}\right\} $
where $\left\{ v_{1},\ldots,v_{n-1}\right\} $ is a reduced basis
for $\lat$. We use the notations introduced\textcolor{magenta}{{} }in
Definition \ref{def: Siegel reduced basis}: $\left\{ \phi_{1},\dots,\phi_{n-1}\right\} $
is the Gram-Schmidt basis obtained from $\left\{ v_{1},\ldots,v_{n-1}\right\} $,
$V_{i}$ is $\sp\left\{ v_{1},\ldots,v_{i}\right\} =\sp\left\{ \phi_{1},\dots,\phi_{i}\right\} $,
and $a_{i}$ is the projection of $v_{i}$ on the line orthogonal
to $V_{i-1}$ inside the space $V_{i}$, where $\sp\left\{ \emptyset\right\} $
is set to be the trivial subspace $\left\{ 0\right\} $. In other
words, $a_{i}$ is the distance of $v_{i}$ from the subspace $V_{i-1}$
(Figure \ref{fig: Butz Lemma - in plane v perp}). If $\lat$ is such
that $\covol{\lat^{i}}\in\left[\cov^{\a_{i}},\cov^{\b_{i}}\right]$,
then $a_{i}\leq R_{i}=X^{\b_{i}-\a_{i-1}}$. Denote the number of
possibilities for choosing $v_{i}$ given that $\lat^{i-1}$ is known
by $\#v_{i}\vert_{\lat^{i-1}}$. We first claim that for every $1\leq i\leq d$
\begin{equation}
\#v_{i}\vert_{\lat^{i-1}}=O\left(\brac{R_{i}}^{n-i+1}\cdot\mathbf{v}^{-1}\cdot\covol{\lat^{i-1}}\right).\label{eq: =000023 of possibilities for v_i}
\end{equation}
Indeed, for $i=1$, the number $\#v_{i}\vert_{\lat^{i-1}}$ is simply
the number of possibilities for choosing a $\lat_{\#}$ vector $v_{1}$
inside a ball of radius $a_{1}=\norm{v_{1}}$ in $\RR^{n}$, and therefore
\[
\#v_{1}\vert_{\lat^{0}}=\#\brac{\lat_{\#}\cap\ball{R_{1}}{}}=O\brac{\mathbf{v}^{-1}\cdot R_{1}^{n}}.
\]
For $i>1$, the orthogonal projection of $v_{i}$ to the subspace
$V_{i-1}$ must lie inside a Dirchlet domain of the lattice $\tilde{\lat}^{i-1}:=\sp_{\ZZ}\left\{ a_{1}\phi_{1},\dots,a_{i-1}\phi_{i-1}\right\} $.
Thus, $v_{i}$ has to be chosen from the set of $\lat_{\#}$ points
which are of distance $\leq a_{i}\leq R_{i}$ from the Dirichlet domain
for $\tilde{\lat}^{i-1}$ in $\sp_{\RR}\brac{\lat^{i-1}}$. These
are the $\lat_{\#}$ points that lie in a domain which is the product
of the Dirichlet domain for $\tilde{\lat}^{i-1}$ (in $V_{i-1}$)
with a ball of radius $R_{i}$ in the $n-\left(i-1\right)$ dimensional
subspace $V_{i-1}^{\perp}$ (Figure \ref{fig: Butz Lemma - cylinder}).
Denote this ball by $\ball{R_{i}}{n-\left(i-1\right)}$, and then

\begin{eqnarray*}
\#v_{i}\vert_{\lat^{i-1}} & \leq & \#\brac{\lat_{\#}\cap\cbrac{\ball{R_{i}}{n-\left(i-1\right)}\times\mbox{Dirichlet domain for \ensuremath{\tilde{\lat}^{i-1}}}}}\\
 & = & O\brac{\mathbf{v}^{-1}\cdot\vol(\ball{R_{i}}{n-\left(i-1\right)})\cdot\covol{\lat^{i-1}}}\\
 & = & O\brac{\mathbf{v}^{-1}\cdot\brac{R_{i}}^{n-i+1}\cdot\covol{\lat^{i-1}}}.
\end{eqnarray*}
This establishes Equation (\ref{eq: =000023 of possibilities for v_i}).
Now, the number of possibilities for $\lat$ is given by:
\[
\prod_{i=1}^{d}\brac{\#v_{i}\vert_{\lat^{i-1}}}=O\brac{\prod_{i=1}^{d}\brac{\mathbf{v}^{-1}\cdot\brac{R_{i}}^{n-i+1}\cdot\covol{\lat^{i-1}}}}
\]
\[
=O\brac{\mathbf{v}^{-d}\prod_{i=1}^{d}\brac{X^{\left(\b_{i}-\a_{i-1}\right)\left(n-i+1\right)}\cdot X^{\b_{i-1}}}}
\]
where $\a_{0}=0$ and $\b_{d}=1$ (as $\covol{\lat^{1}}=\left\Vert v_{1}\right\Vert \geq\cov^{0}$,
and $\covol{\lat^{d}}=\covol{\lat}\leq\cov^{1}$). Since 
\[
\sum_{i=1}^{d}\brac{\brac{n-i+1}\brac{\b_{i}-\a_{i-1}}+\b_{i-1}}=n-d+1+\sum_{i=1}^{d-1}\brac{n-i}\brac{\b_{i}-\a_{i}}+2\sum_{i=1}^{d-1}\b_{i}=e\brac{\underline{\a},\underline{\b}}
\]
then the number of lattices $\lat$ is bounded by $O\brac{\mathbf{v}^{-d}X^{e\brac{\underline{\a},\underline{\b}}}}$.
\end{proof}

\begin{cor}
\label{cor: very few lattices with very short vector}Assume that
$\lat_{\#}$ is a full lattice in $\RR^{n}$ with covolume $\mathbf{v}$,
let $1\leq d\leq n$, and $0<\om_{1}<\cdots<\om_{d-1}$. For every
$\e>0$, the number of rank $d$ subgroups of $\lat_{\#}$ with covolume
$\leq\cov$ which satisfy $\covol{\lat^{i}}\in\sbrac{1,\cov^{\om_{i}}}$
is $O_{\e}\,\brac{\mathbf{v}^{-d}\cdot\cov^{n-d+1+2\om+\e}}$, where
$\om:=\sum_{i=1}^{d-1}\om_{i}$.
\end{cor}

\begin{proof}
Divide every interval $\sbrac{0,\om_{i}}$ into $N_{i}=N_{i}\brac{\om_{i}}$
sub-intervals
\[
0=\b_{0}^{i}<\b_{1}^{i}<\ldots<\b_{N_{i}}^{i}=\om_{i}
\]
such that $\vbrac{\b_{j}^{i}-\b_{j-1}^{i}}\leq\e$ for every $j=1,\ldots,N_{i}$.
By refining these partitions, we may assume without loss of generality
that $N_{1}=\ldots=N_{d-1}:=N$. Fix $j\in\left\{ 1,\ldots,N\right\} $;
according to Lemma \ref{lem: Butz's Lemma}, the number of rank $d$
subgroups $\lat$ of $\lat_{\#}$ with $\covol{\lat}\leq\cov$ and
$\covol{\lat^{i}}\in\sbrac{\cov^{\b_{j-1}^{i}},\cov^{\b_{j}^{i}}}$
for every $i=1\ldots d-1$ is of order $\cov$ to the power of 
\begin{eqnarray*}
n-d+1+2\sum_{i=1}^{d-1}\b_{j}^{i}+\sum_{i=1}^{d-1}\left(n-i\right)\brac{\b_{j}^{i}-\b_{j-1}^{i}} & \leq & n-d+1+2\sum_{i=1}^{d-1}\om_{i}+\sum_{i=1}^{d-1}\left(n-i\right)\cdot\e\\
 & = & n-d+1+2\om+\e\cdot\left(d-1\right)\left(n-d/2\right),
\end{eqnarray*}
where we have used $\vbrac{\b_{j}^{i}-\b_{j-1}^{i}}\leq\e$ and $\b_{j}^{i}\leq\om_{i}$. 

Let $\lat<\lat_{\#}$ be as in the statement. Since $\covol{\lat^{i}}$
lies in $\sbrac{\cov^{0},\cov^{\om_{i}}}$ for every $i=1,\ldots,d-1$,
then for every $i$ there exist $j_{1}^{i},\ldots j_{d-1}^{i}$ such
that $\covol{\lat_{i}}\in\sbrac{X^{\b_{j-1}^{i}},X^{\b_{j}^{i}}}$.
It follows that 
\[
\#\lat=O_{\e}\brac{\mathbf{v}^{-d}\sum_{\substack{\left\{ j_{1},\ldots j_{n-2}\right\} \\
\subset\left\{ 1,\ldots,N\right\} 
}
}X^{n-d+1+2\om+\e\cdot O_{n}\left(1\right)}}=O_{\e}\brac{\mathbf{v}^{-d}X^{n-d+1+2\om+\e}}.\tag*{\qedhere}
\]
\end{proof}

\begin{proof}[Proof of Proposition \ref{prop: very few SL(n,Z) points up the cusp}]
Let $\ga_{v}=ka_{\underline{s}}^{\prime\prime}a_{t}^{\prime}n\in\ga\in\funddom_{T}\cap\sl n\left(\ZZ\right)$.
By definition of $\funddom_{T}^{\left[\underline{\s}T;\infty\right]}$
we have that $\ga\in\funddom_{T}^{\left[\underline{\s}T;\infty\right]}$
if and only if $t\in\left[0,T\right]$ and $s_{i}\geq\s_{i}T$ for
some $i$. According to Proposition \ref{prop: explicit RI coordinates of g},
(by which $\covol{\lat_{v}^{i}}=e^{\frac{it}{n-1}-\frac{s_{i}}{2}}$),
and since $\lat_{v}^{i}$ is integral, we have that $1\leq\covol{\lat_{v}^{i}}\leq e^{(\frac{i}{n-1}-\frac{\s_{i}}{2})T}$.
Thus, the number of $\sl n\left(\ZZ\right)$-elements $\ga$ in $\funddom_{T}^{\left[\underline{\s}T;\infty\right]}$
is bounded by  the number of $\left(n-1\right)$-dimensional subgroups
$\lat_{v}$ of $\ZZ^{n}$ of co-volume $\leq e^{T}:=\cov$, for which
there exists $i\in\left\{ 1,\ldots,n-2\right\} $ such that $\covol{\lat_{v}^{i}}\in\sbrac{1,X^{\frac{i}{n-1}-\frac{\s_{i}}{2}}}$,
where for $j\neq i$ $\covol{\lat_{v}^{j}}\in\sbrac{1,X^{\frac{j}{n-1}}}$.
In other words, 
\[
\#\brac{\funddom_{T}^{\left[\underline{\s}T;\infty\right]}\cap\sl n\left(\ZZ\right)}=\#\brac{\bigcup_{\substack{\underline{u}=\brac{u_{1},\ldots,u_{n-2}}\\
\in\left\{ 0,1\right\} ^{n-2}-\left\{ 0\right\} 
}
}\cbrac{\lat_{v}:\forall i,\,\covol{\lat_{v}^{i}}\in\sbrac{1,X^{\frac{i}{n-1}-\frac{\sigma_{i}u_{i}}{2}}}}}
\]
which by Corollary \ref{cor: very few lattices with very short vector}
with $e^{T}=\cov$, $d=n-1$ and $\frac{i}{n-1}-\frac{\sigma_{i}u_{i}}{2}=\om_{i}$
equals to 
\[
\sum_{\underline{u}\in\left\{ 0,1\right\} ^{n-2}-\left\{ 0\right\} }O_{\e}\brac{X^{2+n-2+\e-\sum_{i=1}^{n-2}\sigma_{i}u_{i}}}=O_{\e}\brac{X^{n-\sigma_{\min}+\e}}
\]
where $\sigma_{\min}=\min\left\{ \sigma_{i}\right\} $. 
\end{proof}

\section{\label{sec: Almost-a-proof}Almost a proof for Theorems A and B}

In Section \ref{sec: Defining a counting problem} (see Corollary
\ref{cor: the sets we should count in}), the proof of Theorem \ref{thm: MainThm}
was reduced to counting integral matrices in three families of subsets
of $\shortdom\subset\sl n\left(\RR\right)$. But, as we shall see,
it is in fact sufficient to count integral matrices in only one of
these families, the one corresponding to part \ref{enu: MainThm_restricted Q}
of Theorem \ref{thm: MainThm}: $\left(\shortdom\right)_{T}\brac{\latset,\a}$.
The content of this section is a proof of Theorem \ref{thm: MainThm},
assuming the following (yet to be proved) counting statement in this
family:
\begin{prop}
\textcolor{green}{\label{prop: Counting with Hexagons (A counting)}}For
$\a\in\left(0,1\right)$, assume that $\latset\subseteq\latspace{n-1,n}$
is BCS. Set $\lm_{n}=n^{2}/\left(2\left(n^{2}-1\right)\right)$, and
let $\errexp_{n}$ be as in Theorem \ref{thm: MainThm}. 
\begin{enumerate}
\item For $\e\in\brac{0,\errexp_{n}}$, $\Svec=\brac{S_{1},\dots,S_{n-2}}$,
$\sumS=\sum_{i=1}^{n-2}S_{i}$ and every $T\geq\frac{\sumS}{n\lm_{n}\errexp_{n}}+O\left(1\right)$,
\[
\#\brac{\trunc{\brac{\shortdom}_{T}}{\Svec}\brac{\latset,\a}\cap\sl n\left(\ZZ\right)}=\frac{\mu\brac{\trunc{\brac{\shortdom}_{T}}{\Svec}\brac{\latset,\a}}}{\mu\left(\sl n\left(\RR\right)/\sl n\left(\ZZ\right)\right)}+O_{\e,\latset}\brac{e^{\sumS/\lm_{n}}e^{nT\left(1-\errexp_{n}+\e\right)}}.
\]
\item For $\e>0$, $\dl\in\left[0,\errexp_{n}-\e\right)$, $T\geq O\left(1\right)$
and $\Svec\left(T\right)$ such that $\sum S_{i}\left(T\right)\leq n\dl\lm_{n}T+O_{\latset}(1)$,
\[
\#\brac{\trunc{\brac{\shortdom}_{T}}{\Svec\left(T\right)}\brac{\latset,\a}\cap\sl n\left(\ZZ\right)}=\frac{\mu\brac{\trunc{\brac{\shortdom}_{T}}{\Svec\left(T\right)}\brac{\latset,\a}}}{\mu\left(\sl n\left(\RR\right)/\sl n\left(\ZZ\right)\right)}+O_{\e,\latset}\brac{e^{nT\left(1-\errexp_{n}+\dl+\e\right)}}.
\]
\end{enumerate}
\end{prop}

Notice that the difference between parts 1 and 2 of the proposition
above is that in the first part $\Svec$ is fixed, while in the second
part, at the cost of compromising the error term, we allow the sum
of $S_{i}$-s to grow proportionally to $T$. 
\begin{rem}
\label{rem: Counting when shapes are restricted}If $\latset\subset\latspace{n-1,n}$
is also bounded, then for suitable $\Svec$ one has that $\brac{\shortdom}_{T}^{\Svec}\brac{\latset,\a}=\brac{\shortdom}_{T}\brac{\latset,\a}$;
thus in this case part 1 of Proposition \ref{prop: Counting with Hexagons (A counting)}
can be written as
\[
\#\left(\brac{\shortdom}_{T}\cap\sl n\left(\ZZ\right)\right)=\mu\left(\brac{\shortdom}_{T}\right)/\mu\left(\sl n\left(\RR\right)/\sl n\left(\ZZ\right)\right)+O_{\e}(e^{nT\left(1-\errexp_{n}+\e\right)}),
\]
where the implied constant depends on $\latset$. 
\end{rem}

The proof of Proposition \ref{prop: Counting with Hexagons (A counting)}
is in Section \ref{sec: Concluding the proofs}. Let us now prove
Theorem \ref{thm: MainThm} based on this proposition:
\begin{proof}[Proof of Theorem \ref{thm: MainThm}]
According to Corollary \ref{cor: the sets we should count in} and
to Lemma \ref{lem: w and w perp}, the quantities we seek to estimate
in parts (\ref{enu: MainThm_restricted shape}), (\ref{enu: MainThm_restricted G''})
and (\ref{enu: MainThm_restricted Q}) of the theorem is in one to
one correspondence with the integral matrices in the following subsets
of $\sl n\left(\RR\right)$: (1) $\brac{\shortdom}_{T}\brac{\sphereset,\symset,\a}$,
(2) $\brac{\shortdom}_{T}\brac{\sphereset,\ensuremath{\uniset},\a}$,
or (3) $\brac{\shortdom}_{T}\brac{\latset,\a}$. Observe that, indeed
the main terms in the theorem are the volumes of these sets, divided
by the measure of $\sl n\left(\RR\right)/\sl n\left(\ZZ\right)$.
Let us demonstrate the computation for the case of the family (1),
for which we recall the notation for the fibers $K_{z_{i}}^{\dprime}$
and the generic fiber $K_{\text{gen}}^{\dprime}$ appearing in the
proof of Lemma \ref{lem: lift to G''}:
\[
\mu\left(\left(\shortdom\right)_{T}\brac{\sphereset,\symset,\a}\right)=\mu\left(\bigcup_{p^{\dprime}\in\parby{P^{\dprime}}{\symset}}\parby{K^{\prime}}{\sphereset}\cdot K_{z^{p^{\dprime}}}^{\dprime}\cdot p^{\dprime}\cdot A_{T}^{\prime}\parby{N^{\prime}}{\ellipse{\zgt{p^{\dprime}}}{\a}}\right)=
\]
\[
=\mu_{K^{\prime}}\left(K_{\sphereset}^{\prime}\right)\mu_{K^{\dprime}}\left(K_{\text{gen}}^{\dprime}\right)\mu_{A^{\prime}}\left(A_{T}^{\prime}\right)\int_{\parby{P^{\dprime}}{\symset}\cap\interior{\symfund{n-1}}}\mu_{N^{\prime}}\left(\parby{N^{\prime}}{\ellipse{\zgt{p^{\dprime}}}{\a}}\right)d\mu_{P^{\dprime}}\left(p^{\dprime}\right)
\]
\[
+\sum_{i}\mu_{K}\left(K_{\sphereset}^{\prime}K_{z_{i}}^{\dprime}\right)\mu_{A^{\prime}}\left(A_{T}^{\prime}\right)\int_{\substack{\parby{P^{\dprime}}{\symset}\cap\del\symfund{n-1}\cap\\
\left\{ p^{\dprime}:\sym^{+}\left(\lat_{p^{\dprime}}\right)=\sym^{+}\left(\lat_{z_{i}}\right)\right\} 
}
}\mu_{N^{\prime}}\left(\parby{N^{\prime}}{\ellipse{\zgt{p^{\dprime}}}{\a}}\right)d\mu_{P^{\dprime}}\left(p^{\dprime}\right),
\]
where we have used: the definition \ref{nota: Definitoin of subsets to count in}
for $\left(\shortdom\right)_{T}\brac{\sphereset,\symset,\a}$, Formula
\ref{eq: Haar measure on G} for the decomposition of $\mu$ to RI
components and Proposition \ref{prop: generic fiber} which tells
us that all the interior points in $\symfund{n-1}$ have the generic
fiber. Now, the second summand is of measure zero, since the boundary
of $\symfund{n-1}$ is such, so we are left only with the first summand. 

Since, by ``Measures on the RI components'' in Section \ref{sec: RI components},
$\mu_{N^{\prime}}$ is the Lebesgue measure on $\RR^{n-1}$, $\mu_{K^{\prime}}$
is the Lebesgue measure on $\sphere{n-1}$, the volume of $\so{n-1}\left(\RR\right)$
is $\prod_{i=1}^{n-2}\Leb{\sphere i}$ (implying that the measure
of $K_{\text{gen}}^{\dprime}$ is $\prod_{i=1}^{n-2}\Leb{\sphere i}/\sn\left(n-1\right)$),
and the $\mu_{A^{\prime}}\left(A_{T}^{\prime}\right)=e^{nT}/n$,
and since by Proposition \ref{prop: spread models that we need}
we can pass from integration on $\symfund{n-1}$ to integration on
$\shapespace{n-1}$, we have that the above equals to 
\[
\frac{e^{nT}\cdot\mu_{\sphere{n-1}}\left(\sphereset\right)\cdot\prod_{i=1}^{n-2}\Leb{\sphere i}}{n\cdot\sn\left(n-1\right)}\int_{\symset}\dirfunc(z)d\mu_{\shapespace{n-1}}\left(z\right),
\]
as wanted.

We claim that it is sufficient to prove part (\ref{enu: MainThm_restricted Q})
of the theorem, since parts (\ref{enu: MainThm_restricted shape})
and (\ref{enu: MainThm_restricted G''}) are special cases. Indeed,
family (1) is a special case of family (2), when taking $\ensuremath{\uniset}\subseteq\unilatspace{n-1}$
to be the inverse image of $\symset\subseteq\unilatspace{n-1}$.
This is because Lemma \ref{lem: lift to G''} gives that the lift
$\ensuremath{\uniset}$ is a BCS when $\symset$ is, so the assumption
of part 1 of the theorem implies the assumption in part 2 for the
lifts; moreover, this Lemma gives that $\text{\ensuremath{\mu_{\unilatspace{n-1}}\brac{\ensuremath{\uniset}}}}=\mu_{\shapespace{n-1}}\brac{\symset}{\scriptstyle \prod}_{i=1}^{n-2}\Leb{\sphere i}/\sn\left(n-1\right)$,
so the main term provided in part 2 for $\ensuremath{\uniset}$ and
$\sphereset$ coincides with the one provided in part 1 of this theorem
for $\symset$ and $\sphereset$. Similarly, family (2) is a special
case of family (3), when taking $\latset$ such that $\parby{\wc}{\latset}=\parby{K^{\prime}}{\sphereset}\parby{G^{\dprime}}{\ensuremath{\uniset}}$.
By Lemma \ref{lem: lift to G''}, $\latset$ is a BCS when $\ensuremath{\uniset}$
and $\sphereset$ are, and $\mu_{\latspace{n,n-1}}\brac{\latset}=\mu_{\unilatspace{n-1}}\brac{\ensuremath{\uniset}}\mu_{\sphere{n-1}}\brac{\sphereset}$.
We therefore prove only part 3 of Theorem \ref{thm: MainThm}.  

\textbf{(i)} Let us first consider the case where $\latset$ is \textbf{\emph{not
bounded}}, and therefore $\brac{\shortdom}_{T}\brac{\latset,\a}$
is not bounded also. Fix $\e\in\brac{0,\errexp_{n}}$, $\delta\in\brac{0,\errexp_{n}-\e}$,
and $\sigmn:=\brac{\frac{\dl n\lm_{n}}{n-2}-\e}\cdot\Onevec$, where
$\lm_{n}=\frac{n^{2}}{2\left(n^{2}-1\right)}$; note that the sum
of the coordinates of $\sigmn$ is $\dl n\lm_{n}-\left(n-2\right)\e$,
which is smaller than $\dl n\lm_{n}+O_{\latset}(1)/T$ for $T$ large
enough. Using Proposition \ref{prop: very few SL(n,Z) points up the cusp},
we reduce to counting in compact sets $\brac{\shortdom}_{T}^{\sigmn T}\brac{\latset,\a}$,
and pay with an error term of $O_{\e}(e^{nT\brac{1-\frac{\dl\lm_{n}}{n-2}+\frac{\e}{n}+\e}})$,
which we can write as $O_{\e}(e^{nT\brac{1-\frac{\dl\lm_{n}}{n-2}+\e}})$
since $\e$ is arbitrary. \textbf{(ii)} Counting integral matrices
in the sets $\brac{\shortdom}_{T}^{\sigmn T}\brac{\latset,\a}$ will
complete the proof, and it is performed using (the second part of)
Proposition \ref{prop: Counting with Hexagons (A counting)}, based
on which it equals
\[
\mu(\,\brac{\shortdom}_{T}^{\sigmn T}\brac{\latset,\a})+O_{\e}(e^{nT\left(1-\errexp_{n}+\dl+\e\right)}).
\]
Since $\mu\left(\brac{\shortdom}_{T}\brac{\latset,\a}\right)=\mu(\brac{\shortdom}_{T}^{\sigmn T}\brac{\latset,\a})+O(e^{nT(1-\frac{\dl\lm_{n}}{n-2})})$
(see remark about the measure in Notation \ref{nota: def of H(T,S)}),
and the error term is swallowed in the one obtained in step (i), we
obtain
\[
\#\left(\brac{\shortdom}_{T}\brac{\latset,\a}\cap\sl n\left(\ZZ\right)\right)=\mu\left(\brac{\shortdom}_{T}\brac{\latset,\a}\right)+O_{\e}(e^{nT\left(1-\errexp_{n}+\dl+\e\right)})+O_{\e}(e^{nT(1-\frac{\dl\lm_{n}}{n-2}+\e)}).
\]
\textbf{(iii)} We now choose $\dl$ that will balance the two error
terms above: $1-\errexp_{n}+\dl=1-\frac{\dl\lm_{n}}{n-2}$ if and
only if $\dl=\errexp_{n}/(1+\frac{\lm_{n}}{n-2})=\errexp_{n}\cdot\left(1-\frac{n^{2}}{2n^{3}-3n^{2}-2n+4}\right)$
. Then the final error term for non bounded $\latset$ is $e^{nT\left(1-\errexp_{n}\cdot\frac{n^{2}}{2n^{3}-3n^{2}-2n+4}\right)}$.
\textbf{(iv)} Moving forward to \textbf{\emph{bounded}} $\latset$,
we repeat a similar strategy as in the unbounded case, performing
only step (ii). Fix $\e\in\left(0,\errexp_{n}\right)$ and apply Proposition
\ref{prop: Counting with Hexagons (A counting)} (case of Remark \ref{rem: Counting when shapes are restricted})
to obtain that the number of integral matrices in $\brac{\shortdom}_{T}\brac{\latset,\a}$
is $\mu\left(\brac{\shortdom}_{T}\brac{\latset,\a}\right)+O_{\e}\left(e^{nT\left(1-\errexp+\e\right)}\right)$.
This completes the proof for the bounded case in the theorem. 
\end{proof}

\begin{proof}[Proof of Theorem \ref{thm: |w_v|/|v| to zero}]
Let 
\[
\Gset=\,\shortdom\,\cap\,\left\{ g=ka_{\svec}^{\dprime}a_{t}^{\prime}n^{\prime}:s_{n-2}\in[0,t/2]\right\} .
\]
We define 
\[
\cA=\cbrac{v\in\ZZ_{\prim}^{n}:\ga_{v}\in\Gset},
\]
and claim that it is a set of full density in $\ZZ_{\prim}^{n}$.
In fact, we show that $\ZZ_{\prim}^{n}-\cA$ is a set of density zero.
For this, we note that the set $\shortdom-\Gset$ is contained in
the set 
\[
\widetilde{\Gset}=\lim_{T\to\infty}\widetilde{\Gset}_{T}
\]
where
\begin{align*}
\widetilde{\Gset}_{T} & =\coprod_{t=1}^{T}\brac{\left(\shortdom\right)_{t}-\left(\shortdom\right)_{t-1}}\cap\left\{ g=ka^{\prime}a_{\svec}^{\dprime}n:\begin{matrix}0\leq s_{1},\ldots,s_{n-3},\\
\frac{1}{2}(t-1)\leq s_{n-2}
\end{matrix}\right\} .
\end{align*}
 Note that $\widetilde{\Gset}_{T}$ can also be written as the disjoint
union 
\begin{equation}
\widetilde{\Gset}_{T}=\coprod_{t=1}^{T}\brac{\left(\shortdom\right)_{t}-\trunc{\left(\shortdom\right)_{t}}{(0,\ldots,0,\frac{1}{2}(t-1))}}-\brac{\left(\shortdom\right)_{t-1}-\trunc{\left(\shortdom\right)_{t-1}}{(0,\ldots,0,\frac{1}{2}(t-1))}}.\label{eq: Thm_A_aux}
\end{equation}
As a result, the volume of $\widetilde{\Gset}_{T}$ can be bounded
as follows: 
\begin{align*}
\mu\left(\widetilde{\Gset}_{T}\right) & \leq\sum_{t=1}^{T}\mu\brac{\left(\shortdom\right)_{t}-\trunc{\left(\shortdom\right)_{t}}{(0,\ldots,0,\frac{1}{2}(t-1))}}\\
 & \porsmall\sum_{t=1}^{T}\left(\int_{\frac{1}{2}\brac{t-1}}^{\infty}e^{-s_{n-2}}ds_{n-2}\right)\int_{t-1}^{t}e^{n\tau}d\tau\\
 & =\frac{1}{n}\sum_{t=1}^{T}e^{-\frac{1}{2}t+\frac{1}{2}}(e^{nt}-e^{n(t-1)})\\
 & \leq Te^{(n-\frac{1}{2})T}.
\end{align*}
The presentation in (\ref{eq: Thm_A_aux}) can also be used to estimate
the number of $\sl n\left(\ZZ\right)$ elements in $\widetilde{\Gset}_{T}$,
by counting $\sl n\left(\ZZ\right)$ elements in each of the summands
separately. For this, Let $\dl$ and $\sigmn=\brac{\s_{1},\ldots,\s_{n-2}}$
 be as in the proof of the unbounded case in Theorem \ref{thm: MainThm}.
Going along the lines of this proof, we can reduce counting in each
(non-compact) summand to counting in the truncated set 
\begin{eqnarray*}
 &  & \trunc{\brac{\shortdom}_{t}}{\sigmn t}-\trunc{\brac{\shortdom}_{t}}{(\s_{1}t,\ldots,\s_{n-3}t,\min\cbrac{\s_{n-2}t,\frac{1}{2}(t-1)})}\\
 &  & \left.\right.\hfill\left.\right.\hfill-\left(\trunc{\brac{\shortdom}_{t-1}}{\sigmn t}-\trunc{\brac{\shortdom}_{t-1}}{(\s_{1}t,\ldots,\s_{n-3}t,\min\cbrac{\s_{n-2}t,\frac{1}{2}(t-1)})}\right),
\end{eqnarray*}
since according to Proposition \ref{prop: very few SL(n,Z) points up the cusp},
the difference in the amount of lattice points inside each summand
and its truncation lies in $O_{\e}(e^{nt\brac{1-\frac{\dl\lm_{n}}{n-2}+\e}})$,
and the difference between their measures is also swallowed in this
error estimate (see remark about the measure in Notation \ref{nota: def of H(T,S)}).
Using Proposition \ref{prop: Counting with Hexagons (A counting)}(ii)
to estimate the amount of lattice points in each truncated summand,
we obtain that the number of lattice points in each (full) summand
is its measure divided by $\mu(\sl n(\RR)/\sl n(\ZZ))$, up to an
error term of order $O_{\e}(e^{nt\brac{1-\frac{\dl\lm_{n}}{n-2}+\e}})$.
As a result, 
\[
\#\brac{\widetilde{\Gset}_{T}\cap\sl n(\ZZ)}=\frac{\mu(\widetilde{\Gset}_{T})}{\mu(\sl n(\RR)/\sl n(\ZZ))}+O_{\e}(Te^{nT\brac{1-\frac{\dl\lm_{n}}{n-2}+\e}}).
\]

According to Corollary \ref{cor: the sets we should count in},
\[
\lim_{T\to\infty}\frac{\#\brac{\brac{\ZZ_{\prim}^{n}-\cA}\cap\ball{e^{T}}{}}}{\#\brac{\ZZ_{\prim}^{n}\cap\ball{e^{T}}{}}}=\lim_{T\to\infty}\frac{\#\brac{\brac{\brac{\shortdom}_{T}-\Gset}\cap\sl n(\ZZ)}}{\#\brac{\brac{\shortdom}_{T}\cap\sl n(\ZZ)}}
\]
which by definition of $\widetilde{\Gset}$ is at most
\[
\leq\lim_{T\to\infty}\frac{\#\brac{\widetilde{\Gset}_{T}\cap\sl n(\ZZ)}}{\#\brac{\brac{\shortdom}_{T}\cap\sl n(\ZZ)}}.
\]
The denominator in the above limit is, according to Proposition \ref{prop: primitive vectors correspond to integral matrices}
and Theorem \ref{thm: MainThm} , asymptotic to 
\[
\frac{\mu(\brac{\shortdom}_{T})}{\mu(\sl n(\RR)/\sl n(\ZZ))}\porequal e^{nT}.
\]
We then have that 
\[
\lim_{T\to\infty}\frac{\#\brac{\brac{\ZZ_{\prim}^{n}-\cA}\cap\ball{e^{T}}{}}}{\#\brac{\ZZ_{\prim}^{n}\cap\ball{e^{T}}{}}}\leq\lim_{T\to\infty}\frac{\mu(\widetilde{\Gset}_{T})}{\mu(\brac{\shortdom}_{T})}=\lim_{T\to\infty}\frac{Te^{(n-\frac{1}{2})T}}{e^{nT}}=0,
\]
which establishes that $\cA$ is a set of full density in $\ZZ_{\prim}^{n}$. 

For a primitive vector $v$ with large enough norm, we have by Lemma
\ref{lem: w and w perp} and definition of $w_{v}$ that
\[
0<\frac{\norm{w_{v}}}{\norm v}\porsmall\frac{\norm{\perpen{w_{v}}}}{\norm v}\leq\frac{\rad_{v}}{\norm v}=\frac{\rad_{v}}{\covol{\lat_{v}}}.
\]
Minkowski's 2nd Theorem gives us that $\rad_{v}=\rad(\lat_{v})\porequal\mathfrak{m}_{n-1}(\lat_{v})$,
where $\mathfrak{m}_{i}$ denotes the $i$th successive minima. From
\cite[Theorem 7.9]{GM02} we have that $\covol{\lat_{v}}\porequal\mathfrak{m}_{1}(\lat_{v})\cdots\mathfrak{m}_{n-1}(\lat_{v})$.
Thus the above can be further estimated as: 
\[
\porsmall\frac{\mathfrak{m}_{n-1}(\lat_{v})}{\mathfrak{m}_{1}(\lat_{v})\cdots\mathfrak{m}_{n-1}(\lat_{v})}=\frac{1}{\mathfrak{m}_{1}(\lat_{v})\cdots\mathfrak{m}_{n-2}(\lat_{v})}\porsmall\covol{\lat_{v}^{n-2}}^{-1},
\]
where by Proposition \ref{prop: explicit RI coordinates of g}(iii),
\[
\porsmall e^{\frac{s_{n-2}}{2}-\frac{n-2}{n-1}t}\leq e^{-\frac{3n-6}{4(n-1)}t}.
\]
The above decays to $0$ if (and only if) $n>2$, and we are done
\textendash{} since if $\{v_{m}\}\subset\cA$ diverges, then $\ga_{v_{m}}=k_{m}a_{m}^{\dprime}a_{t_{m}}^{\prime}n_{m}$
with $t_{m}\to\infty$ as $m\to\infty$, implying that $\norm{w_{v_{m}}}/\norm{v_{m}}\to0$. 
\end{proof}

\part{Counting lattice points\label{part: Technical Part}}

This second part is the technical part of the paper, where we prove
Proposition \ref{prop: Counting with Hexagons (A counting)}, in order
to conclude the proof of Theorem \ref{thm: MainThm}. This proposition
concerns counting lattice points in $\sl n\left(\RR\right)$; our
main tool for this purpose is a method introduced in \cite{GN1} for
counting lattice points in increasing families of sets inside semisimple
Lie groups. The advantages of this method is that it produces an error
term, and that it allows counting in quite general families, requiring
only that these families are \emph{well rounded}, which is a regularity
condition. The cost of this generality is that the property of well
roundedness is often hard to verify. In \cite{HK_WellRoundedness}
we develop a machinery to somewhat simplify this process, mainly by
allowing us to replace the underlying simple group $G=KAN$ with the
much-easier-to-work-in Cartesian product $K\times A\times N$; we
will refer to some technical results from there in the course of Part
II. 

\section{\label{sec: Well roundededness} Counting lattice points in well
rounded families of sets inside Lie groups }

We begin in Subsection \ref{subsec: GN method} by describing the
counting lattice points method that we will use, and proceed in Subsection
\ref{subsec: Plan of proof} with laying out a plan of proof for Proposition
\ref{prop: Counting with Hexagons (A counting)}. From now on, we
use $\gam$ to denote a general lattice in a Lie group, hence abandoning
the notation in  Section \ref{sec: Integral matrices representing primitive vectors}.

\subsection{A method for lattice points counting in Well rounded families \label{subsec: GN method}}

In this subsection we briefly describe the counting method developed
in \cite{GN1}. This approach, aimed at counting lattice points in
increasing families of sets inside non-compact algebraic simple Lie
groups, consists of two ingredients: a regularity condition on the
sets involved, and a spectral estimate concerning the unitary $G$
representation $\pi_{G/\Lat}^{0}:G\to L_{0}^{2}\left(G/\Lat\right)$
(the orthogonal complement of the $G$ invariant $L^{2}$ functions).
Before stating the counting theorem \ref{thm: GN Counting thm} from
\cite{GN1}, we describe the two ingredients, starting with the regularity
condition. 
\begin{defn}
\label{def: well--roundedness}Let $G$ be a Lie group with a Borel
measure $\mu$, and let $\cbrac{\nbhd{\e}{}}_{\e>0}$ be a family
of identity neighborhoods in $G$. Assume $\left\{ \Gset_{T}\right\} _{T>0}\subset G$
is a family of measurable domains and denote 
\[
\Gset_{T}^{\left(+\e\right)}:=\nbhd{\e}{}\Gset_{T}\nbhd{\e}{}=\bigcup_{u,v\in\nbhd{\e}{}}u\,\Gset_{T}\,v,
\]
\[
\Gset_{T}^{\left(-\e\right)}:=\bigcap_{u,v\in\nbhd{\e}{}}u\,\Gset_{T}\,v
\]
(see Figure \ref{fig: Well-Roundedness}). The family $\left\{ \Gset_{T}\right\} $
is \emph{Lipschitz well rounded (LWR)} with (positive) parameters
$\brac{\mathcal{C},T_{0}}$ if for every $0<\e<1/\mathcal{C}$ and
$T>T_{0}$: 
\begin{equation}
\mu\left(\Gset_{T}^{\left(+\e\right)}\right)\leq\left(1+\mathcal{C}\e\right)\:\mu\left(\Gset_{T}^{\left(-\e\right)}\right).\label{eq:LWReq}
\end{equation}
The parameter $\mathcal{C}$ is called the \emph{Lipschitz constant}
of the family $\left\{ \Gset_{T}\right\} $. 
\end{defn}

\begin{figure}
\hspace*{\fill}\subfloat[\foreignlanguage{british}{The set $\protect\Gset_{T}$}]{\includegraphics[scale=0.5]{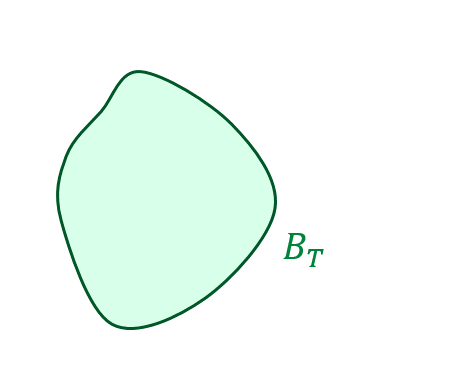}}\hspace*{\fill}\subfloat[\foreignlanguage{british}{The set $\protect\Gset_{T}$ is perturbed by ${\cal O}_{\protect\e}$}]{\hspace*{5mm}\includegraphics[scale=0.5]{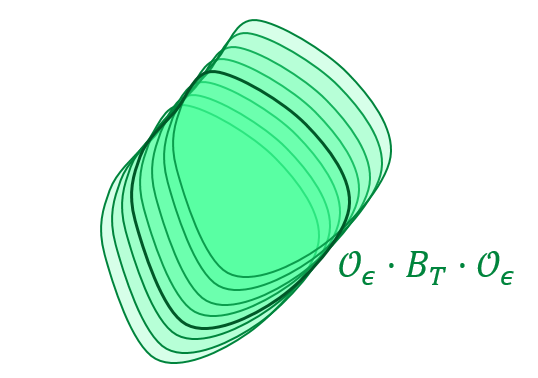}\hspace*{5mm}}\hspace*{\fill}\subfloat[$\protect\Gset_{T}^{\left(-\protect\e\right)}$ and $\protect\Gset_{T}^{\left(+\protect\e\right)}$]{\includegraphics[scale=0.5]{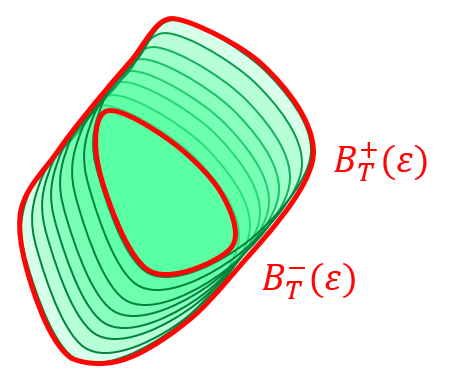}}\hspace*{\fill}

\caption{Well roundedness.\label{fig: Well-Roundedness}}
\end{figure}

The definition above allows any family $\left\{ \nbhd{\e}{}\right\} _{\e>0}$
of identity neighborhoods; in this paper we shall restrict to the
following:
\begin{assumption}
\label{assu: Our O_e }We will assume that $\nbhd{\e}G=\exp\left(\ball{\e}{}\right)$,
where $B_{\e}$ is an origin-centered $\e$-ball inside the Lie algebra
of $G$, and $\exp$ is the Lie exponent. 
\end{assumption}

\begin{rem}
\label{rem: LWR set}We allow the case of a constant family $\left\{ \Gset_{T}\right\} =\Gset$:
we say that $\Gset$ is a Lipschitz well rounded \emph{set} (as apposed
to a Lipschitz well rounded \emph{family}) with parameters $\brac{\mathcal{C},\e_{0}}$
if  $\mu\brac{\Gset^{\left(+\e\right)}}\leq\brac{1+\mathcal{C}\e}\:\mu\brac{\Gset^{\left(-\e\right)}}$
for every $0<\e<\e_{0}$. It is proved in \cite[Prop. 3.5]{HK_WellRoundedness}
that if a set $\Gset$ is BCS and bounded, then it is LWR. 
\end{rem}

We now turn to describe the second ingredient, which is the spectral
estimation. In certain Lie groups, among which algebraic simple Lie
groups $G$, there exists $p\in\mathbb{N}$ for which the matrix coefficients
$\langle\pi_{G/\Lat}^{0}u,v\rangle$ are in $L^{p+\e}\left(G\right)$
for every $\e>0$, with $u,v$ lying in a dense subspace of $L_{0}^{2}\left(G/\Lat\right)$
(see \cite[Thm 5.6]{GN_book}). Let $p\left(\Lat\right)$ be the smallest
among these $p$'s, and denote 
\[
m\left(\Lat\right)=\begin{cases}
1 & \text{if \ensuremath{p=2},}\\
2\left\lceil p\left(\Lat\right)/4\right\rceil  & \text{otherwise.}
\end{cases}
\]
The parameter $m\left(\Lat\right)$ appears in the error term exponent
of the counting theorem below, which is the cornerstone of the counting
results in this paper. 
\begin{thm}[{\cite[Theorems 1.9, 4.5, and Remark 1.10]{GN1}}]
\label{thm: GN Counting thm}Let $G$ be an algebraic simple Lie
group with Haar measure $\mu$, and let $\Lat<G$ be a lattice. Assume
that $\{\Gset_{T}\}\subset G$ is a family of finite-measure domains
which satisfy $\mu(\Gset_{T})\to\infty$ as $T\to\infty$. If the
family $\{\Gset_{T}\}$ is Lipschitz well rounded with parameters
$(C_{\Gset},T_{0})$, then $\exists T_{1}>0$ such that for every
$\delta>0$ and $T>T_{1}$:
\[
\#\brac{\Gset_{T}\cap\Lat}-\mu\brac{\Gset_{T}}/\mu\left(G/\Lat\right)\underset{G,\Lat,\delta}{\porsmall}C_{\Gset}^{\frac{\dim G}{1+\dim G}}\cdot\mu\brac{\Gset_{T}}^{1-\errexp\left(\Lat\right)+\delta},
\]
where $\mu\left(G/\Lat\right)$ is the measure of a fundamental domain
of $\Lat$ in $G$ and
\[
1-\errexp\left(\Lat\right)=1-\brac{2m\left(\Lat\right)\left(1+\dim G\right)}^{-1}\in\left(0,1\right).
\]
The parameter $T_{1}$ is such that $T_{1}\geq T_{0}$ and for every
$T\geq T_{1}$ 
\begin{equation}
\mu\left(\Gset_{T}\right)^{\errexp\left(\Lat\right)}\underset{G,\Lat}{\porbig}C_{\Gset}^{\frac{\dim G}{1+\dim G}}.\label{eq: def of T_1 in GN thm}
\end{equation}
\end{thm}

Bounds on the parameter $p\left(\Lat\right)$ (i.e. on $m\left(\Lat\right)$)
clearly imply bounds on the parameter $\errexp\left(\Lat\right)$
appearing in the error term exponent. We refer to \cite{Li95}, \cite{Li_Zhu_96}
and \cite{Scaramuzzi90} for upper bounds on $p\left(\Lat\right)$
in simple Lie groups. Specifically for the group $\sl n\left(\RR\right)$,
the current known bound for $n>2$ and any lattice $\Lat$ in $\sl n\left(\RR\right)$
is $2\leq p\left(\Lat\right)\leq2n-2$ \cite{Li95}. For the lattice
$\Lat=\sl n\left(\ZZ\right)$, $p\left(\sl n\left(\ZZ\right)\right)=2n-2$
\cite{DRS93} which implies that $m\left(\sl n\left(\ZZ\right)\right)=2\left\lceil \left(n-1\right)/2\right\rceil $
and therefore $\errexp\left(\sl n\left(\ZZ\right)\right)$ is exactly
$\errexp_{n}$ from Theorem \ref{thm: MainThm}. 

\subsection{\label{subsec: Plan of proof}Plan of proof for Proposition \ref{prop: Counting with Hexagons (A counting)}}

Proposition \ref{prop: Counting with Hexagons (A counting)} is concerned
with counting in the sets: 
\[
\left(\shortdom\right)_{T}^{\Svec}\brac{\latset,\a}=\bigcup_{q\in\trunc{\brac{\wc_{\latset}}}{\Svec}}q\cdot A_{T}^{\prime}\parby{N^{\prime}}{\ellipse{\zgt q}{\a}}.
\]
According to Theorem \ref{thm: GN Counting thm}, in order to prove
Proposition \ref{prop: Counting with Hexagons (A counting)}, it is
sufficient to claim that the families above are LWR with parameters
that do not depend on $\Svec$. This will be done by following the
two steps below. In each step, we mention technical results from \cite{HK_WellRoundedness},
and conclude with a summary of how and where the goal of the step
will be proved in this paper, and which role will it assume in the
proof of Propostion \ref{prop: Counting with Hexagons (A counting)}. 

\subsubsection*{Step 1: Reduction from LWR in $\protect\sl n\left(\protect\RR\right)=KA^{\prime}A^{\protect\dprime}N^{\protect\dprime}N^{\prime}$
to LWR in $K\times A^{\prime}\times A^{\protect\dprime}\times N^{\protect\dprime}\times N^{\prime}$.}

It is much easier to verify well roundedness in the (resp.\: compact,
abelian, unipotent) subgroups $K,A,N$ of $\sl n\left(\RR\right)$,
and their subgroups, than in the simple $\sl n\left(\RR\right)$.
Let $\roundo$ denote the map from $\sl n\left(\RR\right)$ to the
product, that sends $g=ka^{\prime}a^{\dprime}n^{\dprime}n^{\prime}$
to $\left(k,a^{\prime},a^{\dprime},n^{\dprime},n^{\prime}\right)$\footnote{When a component is omitted, it means that it is the identity.}.
Then
\[
\roundo\left(\left(\shortdom\right)_{T}^{\Svec}\brac{\latset,\a}\right)=\bigcup_{\substack{\left((k,a^{\dprime},n^{\dprime}),a^{\prime}\right)\in\\
\roundo(\trunc{\brac{\wc_{\latset}}}{\Svec})\times A_{T}^{\prime}
}
}\left(k,a^{\dprime},n^{\dprime},a^{\prime}\right)\times\parby{N^{\prime}}{\ellipse{a^{\dprime}n^{\dprime}}{\a}}.
\]
We will apply the following result from \cite{HK_WellRoundedness},
that will enable us to reduce to verifying the well roundedness of
$\roundo\left(\left(\shortdom\right)_{T}^{\Svec}\brac{\latset,\a}\right)$;
but first, a definition. 
\begin{defn}[{\cite[Def. 4.1]{HK_WellRoundedness}}]
\label{def: roundomorphism}Let $G$ and $\sbgrp$ be two Lie groups
with Borel measures $\mu_{G}$ and $\mu_{\sbgrp}$. A Borel measurable
map $\roundogen:G\to\sbgrp$ will be called an \emph{$f$-roundomorphism}
if it is:
\begin{enumerate}
\item \textbf{Measure preserving:} $\roundogen_{*}\brac{\,\mu_{G}}=\mu_{\sbgrp}$.
\item \textbf{Locally Lipschitz:} $\roundogen(\nbhd{\e}Gg\nbhd{\e}G)\subseteq\nbhd{f\e}{\sbgrp}\roundogen(g)\nbhd{f\e}{\sbgrp}$
for some continuous $f=f\left(g\right):G\to\RR_{>0}$ and for every
$0<\e<\frac{1}{f}$.
\end{enumerate}

\end{defn}

In \cite[Prop. 4.2]{HK_WellRoundedness} we prove that if a family
$\Gset_{T}\subseteq\sbgrp$ is LWR and $\roundogen:G\to\sbgrp$ is
a roudomorphism such that $\roundogen^{-1}\left(\Gset_{T}\right)$
is bounded uniformly in $T$, than $\roundogen^{-1}\left(\Gset_{T}\right)$
is LWR. Here we only need the case where $\sbgrp$ is a direct product
of groups:
\begin{prop}[{\cite[Corollary 4.3]{HK_WellRoundedness}}]
\label{cor: roundo to a produt grp}\label{rem: product of LWR is LWR}Let
$\roundogen:G\to\sbgrp=\sbgrp^{1}\times\cdots\times\sbgrp^{\topindex}$
be an $f$-roundomorphism and let $\Gset_{T}=\Gset_{T}^{1}\times\cdots\times\Gset_{T}^{\topindex}\subseteq\sbgrp$.
Set $\mu_{\sbgrp}=\mu_{\sbgrp_{1}}\times\cdots\times\mu_{\sbgrp_{\topindex}}$,
$\nbhd{\e}{\sbgrp}=\nbhd{\e}{\sbgrp_{1}}\times\cdots\times\nbhd{\e}{\sbgrp_{\topindex}}$
and assume that
\begin{enumerate}
\item For $j=1,\dots,\topindex$: $\Gset_{T}^{j}\subseteq\sbgrp^{j}$ is
LWR w.r.t. the parameters $\brac{T_{j},C_{j}}$;
\item $f$ is bounded uniformly by a real number $F$ on the sets $\roundogen^{-1}\brac{\Gset_{T}}$.
\end{enumerate}
Then $\roundogen^{-1}\left(\Gset_{T}\right)$ is LWR, w.r.t. the parameters
\[
T=\max\left\{ T_{1},\dots,T_{\topindex}\right\} ,\;C\asymp_{\topindex}F\cdot\max\left\{ C_{1},\dots,C_{\topindex},1\right\} .
\]
In particular, a direct product of LWR families is LWR in the direct
product of the corresponding group.
\end{prop}

\paragraph*{For the proof of Proposition \ref{prop: Counting with Hexagons (A counting)}:}

In Section \ref{sec: effective Iwasawa and GI decompositions} we
will prove that the map $\roundo$ is a roundomorphism and establish
a bound on $f$, reducing well roundedness of $\left(\shortdom\right)_{T}^{\Svec}\brac{\latset,\a}$
to well roundedness of $\roundo(\left(\shortdom\right)_{T}^{\Svec}\brac{\latset,\a})$. 

\subsubsection*{Step 2: Verifying LWR property in a product of groups. }

The sets $\roundo\left(\left(\shortdom\right)_{T}^{\Svec}\brac{\latset,\a}\right))$
in Step 1 are of the general form 
\[
\Gset_{T}=\bigcup_{z\in\symset_{T}}z\times\domN_{z}\subseteq P\times\RR^{m},\quad\left(\star\right)
\]
where $P$ is a Lie group. We require the following Lipschitzity condition
on the family $\left\{ \domN_{z}\right\} $: 
\begin{defn}[{\cite[Definition 5.1 and Proposition 5.6]{HK_WellRoundedness}}]
\label{def: BLC family}Let $P$ be a Lie group and $\nbhd{\e}{}$
a family of coordinate balls. Let $\symset$ be a subset of $P$,
and consider the family $\fam{\mathcal{\fdomN}}{\symset}=\left\{ \domN_{z}\right\} _{z\in\symset}$,
where $\domN_{z}\subseteq\mathbb{R}^{m}$ ($m$ is uniform for all
$z$). We say that the family $\fam{\mathcal{\fdomN}}{\symset}$ is
\emph{bounded Lipschitz continuous }(or \textbf{BLC}) w.r.t $\nbhd{\e}{}$
if there exists $C>0$ such that for every $0<\e<C^{-1}$ the following
hold:
\begin{enumerate}
\item For a norm ball $\ball{\e}{}\subset\mathbb{R}^{m}$ of radius $\e$,
$\domN_{z}+\ball{\e}{}\subseteq\brac{1+C\e}\domN_{z}$.
\item If $z^{\prime}\subseteq\nbhd{\e}{}z\nbhd{\e}{}$ for $z,z^{\prime}\in\symset$,
then $\domN_{z^{\prime}}\subseteq\brac{1+C\e}\domN_{z}$.
\item The Lebesgue volume of $\domN_{z}$ is bounded uniformly from below
by a positive constant \textbf{$\volmin$}. 
\item $\domN_{z}\subseteq\ball{\radmax}{}$ for some uniform $\radmax>0$
and every $z\in\symset$. 
\end{enumerate}
\end{defn}

The following result relates the BLC property of the family $\left\{ \domN_{z}\right\} $,
to the LWR property of the sets in $\left(\star\right)$. 
\begin{prop}[{\cite[Proposition 5.5]{HK_WellRoundedness}}]
\label{prop: Fibered sets are LWR in direct product}Let $\left\{ \symset_{T}\right\} _{T>0}$
be an increasing family inside a Lie group $P$, and $\symset:=\cup_{T>0}\symset_{T}$.
Let $\fam{\fdomN}{\symset}=\left\{ \domN_{z}\right\} _{z\in\symset}$
where $\domN_{z}\subset\RR^{m}$, and consider the family 
\[
\Gset_{T}=\bigcup_{z\in\symset_{T}}z\times\domN_{z}\subseteq P\times\RR^{m}.
\]
If $\left\{ \symset_{T}\right\} _{T>0}$ is LWR with parameters $\left(T_{0},C_{\symset}\right)$,
and $\mathcal{\fdomN}_{\symset}$ is BLC w.r.t. the family $\{\nbhd{\e}P\}_{\e>0}$
and with parameters $\left(C_{\mathcal{\fdomN}},\volmin,\radmax\right)$,
then $\Gset_{T}$ is LWR w.r.t the family $\nbhd{\e}P\times\ball{\e/2}{\RR^{m}}\subset P\times\RR^{m}$
and with parameters $\left(T_{0},C_{\Gset}\right)$ where
\[
C_{\Gset}\prec C_{\fdomN}+\brac{\volmax/\volmin}C_{\symset}+1
\]
and $\volmax=\mu_{\RR^{m}}\brac{\ball{\radmax}{}}$. 
\end{prop}

\paragraph*{For the proof of Proposition \ref{prop: Counting with Hexagons (A counting)}: }

Following Proposition \ref{prop: Fibered sets are LWR in direct product},
in order to prove that the sets $r(\left(\shortdom\right)_{T}^{\Svec}\brac{\latset,\a})$
from Step 1 are LWR, one should show that:
\begin{itemize}
\item The family $\left\{ \ellipse{a^{\dprime},n^{\dprime}}{\a}\right\} _{\left(a^{\dprime},n^{\dprime}\right)\in\roundo\brac{\trunc{\symfund{n-1}}{\Svec}}}$
is BLC (Definition \ref{def: BLC family}), which is done in Section
\ref{sec: Family for gcd solution is BLC}.
\item The family $\symset_{T}=\roundo(\trunc{\wc_{\latset}}{\Svec})\times A_{T}^{\prime}$
over which the union is taken is LWR (Definition \ref{def: well--roundedness}).
For this, by Remark \ref{rem: product of LWR is LWR}, it is sufficient
to show that each of the factors is LWR. The two factors will be handled
as follows: 
\begin{itemize}
\item In Section \ref{sec: Regularity-results-forA} we show that $\cbrac{A_{T}^{\prime}}$
is LWR;
\item in Section \ref{sec: The base sets are LWR} we show that $\roundo\brac{\trunc{\wc_{\latset}}{\Svec})}$
 is LWR. 
\end{itemize}
\end{itemize}
The proof of Proposition \ref{prop: Counting with Hexagons (A counting)}
is completed in Section \ref{sec: Concluding the proofs}.

\section{\label{sec: Regularity-results-forA}Well roundedness in subgroups
of $A$ }

In this section and the one that follows, we extend our discussion
from $G=\sl n\left(\RR\right)$ to $G$ being a real semi-simple Lie
group with finite center and Iwasawa decomposition $G=KAN$.  Here
we focus on the subgroup $A$, and consider subgroups of it that are
the image of subspaces in $\lieA$, the Lie algebra of $A$, under
the exponent map. To introduce them, we first set some notations. 
\begin{notation}
\label{nota: "scalar product"}For vectors $\lieAelement_{1},\dots,\lieAelement_{\topindex}\in\lieA$,
we write
\[
\lieAvec:=\left(\lieAelement_{1},\dots,\lieAelement_{\topindex}\right)\in\lieA^{\topindex}.
\]
If $\svec=\left(s_{1},\dots,s_{\topindex}\right)\in\RR^{\topindex}$
we let $\svec\cdot\lieAvec:=\sum_{i=1}^{\topindex}s_{i}\lieAelement_{i}$.
We say that $\lieAvec$ is linearly independent if $\lieAelement_{1},\dots,\lieAelement_{\topindex}$
are. 
\end{notation}

We let $\left\{ \phi_{1},\ldots\phi_{\dimN}\right\} \subset\lieA^{*}$
denote the \emph{positive} roots, counted with multiplicities,  and
we use the standard notation for their sum: 
\[
2\rho=\sum_{i=1}^{\dimN}\phi_{i}\in\lieA^{*}.
\]

\begin{defn}
\label{def: H sbgrps of A}Given linearly independent $\lieAvec=\left(\lieAelement_{1},\dots,\lieAelement_{\topindex}\right)$,
we define the subgroup $\Asub{\lieAvec}{}<A$ to be 
\[
\Asub{\lieAvec}{}:=\left\{ \exp\left(\svec\cdot\lieAvec\right):\svec\in\RR^{\topindex}\right\} ,
\]
and endow it with the (non-Haar!) measure 
\[
\mu_{\Asub{\lieAvec}{}}:=e^{2\rho\left(\lieAelement_{1}\right)s_{1}}\cdots e^{2\rho\left(\lieAelement_{\topindex}\right)s_{\topindex}}ds_{1}\cdots ds_{\topindex}.
\]
When $\topindex=1$, we omit the underlines: $\lieAvec=\lieAelement$
and $\svec=s$.
\end{defn}

\begin{rem}
\label{rem: H sbgrps of A}Every closed connected subgroup of $A$
is of the form $\Asub{\lieAvec}{}$. Furthermore, $\Asub{\lieAvec}{}\cap\Asub{\lieAvec^{\prime}}{}=\cbrac{1_{A}}$
if and only if $\lieAvec$ is linearly independent of $\lieAvec^{\prime}$.
In that case, $\Asub{\lieAvec\times\lieAvec^{\prime}}{}=\Asub{\lieAvec}{}\times\Asub{\lieAvec^{\prime}}{}$
as both groups and measure spaces. In particular, if $\lieAvec$ is
a basis for $\lieA$, then $\Asub{\lieAvec}{}=A$ and $\mu_{\Asub{\lieAvec}{}}=\mu_{A}$. 
\end{rem}

\begin{example}
\label{exa: SL(n,R)  1}In the case of $G=\sl n\left(\RR\right)$,
$N=\left[\begin{smallmatrix}1 & \cdots & \RR\\
 & \ddots & \vdots\\
0 &  & 1
\end{smallmatrix}\right]$  and $A=\left[\begin{smallmatrix}e^{\a_{1}} &  & 0\\
 & \ddots\\
0 &  & e^{\a_{n}}
\end{smallmatrix}\right]$, where $\sum\a_{i}=0$. The roots $\phi_{i,j}\in\lieA^{*}$ are defined
via $\phi_{i,j}\brac{{\displaystyle {\scriptstyle \sum}}_{k=1}^{n}\alpha_{k}e_{k,k}}=\alpha_{j}-\alpha_{i}$,
where the positive roots (w.r.t. which $N$ is defined) are the ones
with $j<i$. For $\lieAelement={\displaystyle {\scriptstyle \sum}}_{k=1}^{n}\alpha_{k}e_{k,k}\in\lieA$,
\[
2\rho\left(\lieAelement\right)=2\rho\left(\sum_{k=1}^{n}\alpha_{k}e_{k,k}\right)=\sum_{k=1}^{n}\left(n+1-2k\right)\alpha_{k}.
\]
For $A^{\prime}$ and $A^{\dprime}$ as defined in Section \ref{sec: RI components},
the bases for the Lie algebras are $\lieAelement^{\prime}=(1/\left(n-1\right),\dots,1/\left(n-1\right),-1)$
and $\lieAelement_{i}^{\dprime}=(-e_{i,i}+e_{i+1,i+1})/2$ for $i=1,\dots,n-2$.
For $A^{\prime}$, according to the formula above for $2\rho$, we
have that $2\rho\brac{H^{\prime}}=n$ and therefore 
\[
\mu_{A^{\prime}}=\mu_{\Asub{\lieAvec^{\prime}}{}}=e^{nt}dt,
\]
and for $A^{\dprime}$, $2\rho\brac{\lieAelement_{i}^{\dprime}}=-1$
for all $i$ and therefore 
\[
\mu_{A^{\dprime}}=\prod_{i=1}^{n-2}e^{-s_{i}}ds_{i}.
\]
\end{example}

\begin{defn}
We consider the following subsets of $A$: 
\begin{enumerate}
\item For $\Svec=\brac{S_{1},\ldots,S_{\topindex}}$, 
\[
\Asub{\lieAvec}{\Svec}=\cbrac{\exp\brac{\svec\cdot\lieAvec}:\svec\in{\scriptstyle \prod\limits _{i=1}^{\topindex}}\sbrac{0,S_{i}}}\subseteq\Asub{\lieAvec}{}.
\]
\item When all $S_{i}$ are equal to $T$, we simply write $\Asub{\lieAvec}T\subseteq\Asub{\lieAvec}{}$.
\end{enumerate}
\end{defn}

The goal of this subsection is to prove the following:
\begin{prop}
\label{prop: A cubes are well rounded}The \textbf{family} $\{\Asub{\lieAvec}T\}_{T>0}$
is LWR with parameters which depend only on $\lieAvec$, and the \textbf{fixed
set} $\Asub{\lieAvec}{\Svec}$ is well rounded with parameters which
depend only on $\lieAvec$, when $S_{1},\dots,S_{\topindex}$ are
larger from some $\dl>0$. E.g. $\dl=4/2\rho\left(\lieAelement_{i}\right)$
if $2\rho\left(\lieAelement_{i}\right)\neq0$, and $\dl=1$ otherwise.
\end{prop}

\begin{rem}
Notice that the sets $\Asub{\lieAvec}{\Svec}$ are clearly BCS and
bounded, and are therefore (Remark \ref{rem: LWR set}) LWR; hence
the content of the proposition for these sets is that their LWR parameters
are uniform (i.e., do not depend on $\Svec$).
\end{rem}

\begin{proof}
We only prove the proposition for the family $\{\Asub{\lieAvec}T\}_{T>\dl}$
since the proof for the set $\Asub{\lieAvec}{\Svec}$ is identical.
Moreover, it is sufficient to consider the case of $\topindex=1$,
and then the general case follows from Proposition \ref{cor: roundo to a produt grp}.
Notice that 
\begin{eqnarray*}
\ln\brac{\projset{(\Asub HT}{+\e}} & = & \left[-\e\,,T+\e\right],\\
\ln\brac{\projset{(\Asub HT)}{-\e}} & = & \left[\e\,,T-\e\right].
\end{eqnarray*}
We shall prove LWR of $\cbrac{\Asub HT}_{T>0}$ computationally, by
splitting to different cases according to the sign of $\rho\left(\lieAelement\right)$.
Assume first that $2\rho\left(\lieAelement\right)\neq0$, and then
\[
\mu_{\Asub H{}}\brac{\projset{(\Asub HT)}{+\e}}=\int_{t=-\e}^{t=T+\e}e^{2\rho\left(\lieAelement\right)t}dt=\brac{e^{2\rho\left(\lieAelement\right)\left(T+\e\right)}-e^{-2\rho\left(\lieAelement\right)\e}}/2\rho\left(\lieAelement\right),
\]
and
\[
\mu_{\Asub H{}}\brac{\projset{(\Asub HT)}{-\e}}=\int_{t=\e}^{t=T-\e}e^{2\rho\left(\lieAelement\right)t}dt=\brac{e^{2\rho\left(\lieAelement\right)\left(T-\e\right)}-e^{2\rho\left(\lieAelement\right)\e}}/2\rho\left(\lieAelement\right).
\]
It follows that,
\begin{eqnarray*}
\frac{\mu_{\Asub H{}}\brac{\projset{(\Asub HT)}{+\e}}-\mu_{A_{\lieAelement}}\brac{\projset{(\Asub HT)}{-\e}}}{\mu_{\Asub H{}}\brac{\projset{(\Asub HT)}{-\e}}} & = & \frac{\left(e^{2\rho\left(\lieAelement\right)\left(T+\e\right)}-e^{-2\rho\left(\lieAelement\right)\e}\right)-\left(e^{2\rho\left(\lieAelement\right)\left(T-\e\right)}-e^{2\rho\left(\lieAelement\right)\e}\right)}{e^{2\rho\left(\lieAelement\right)\left(T-\e\right)}-e^{2\rho\left(\lieAelement\right)\e}}.
\end{eqnarray*}
\begin{itemize}
\item If $2\rho\left(\lieAelement\right)>0$ we continue in the following
way 
\[
=\exd{\frac{e^{2\rho\left(\lieAelement\right)T}+1}{e^{2\rho\left(\lieAelement\right)T}}}{\leq2}\cdot\frac{e^{2\rho\left(\lieAelement\right)\e}-e^{-2\rho\left(\lieAelement\right)\e}}{e^{-2\rho\left(\lieAelement\right)\e}-e^{-2\rho\left(\lieAelement\right)T}\cdot e^{2\rho\left(\lieAelement\right)\e}}.
\]
For $\e\leq\frac{1}{2\cdot2\rho\left(\lieAelement\right)}$ and $T\geq\frac{4}{2\rho\left(\lieAelement\right)}$
it holds that $e^{2\rho\left(\lieAelement\right)\e}-e^{-2\rho\left(\lieAelement\right)\e}\leq3\cdot2\rho\left(\lieAelement\right)\e$
and $e^{-2\rho\left(\lieAelement\right)\e}-e^{-2\rho\left(\lieAelement\right)T}\cdot e^{2\rho\left(\lieAelement\right)\e}\geq1/2$;
then, 
\[
\frac{\mu_{\Asub H{}}\brac{\projset{\brac{\Asub HT}}{+\e}}-\mu_{\Asub H{}}\left(\brac{\projset{\brac{\Asub HT}}{-\e}}\right)}{\mu_{\Asub H{}}\brac{\projset{\left(\Asub HT\right)}{-\e}}}\leq2\cdot\frac{3\cdot2\rho\left(\lieAelement\right)\e}{1/2}=12\cdot2\rho\left(\lieAelement\right).
\]
\item If $2\rho\left(\lieAelement\right)<0$, we have
\begin{eqnarray*}
 & = & \frac{(e^{-2\rho\left(\lieAelement\right)\e}-e^{2\rho\left(\lieAelement\right)\left(T+\e\right)})-(e^{2\rho\left(\lieAelement\right)\e}-e^{2\rho\left(\lieAelement\right)\left(T-\e\right)})}{e^{2\rho\left(\lieAelement\right)\e}-e^{2\rho\left(\lieAelement\right)\left(T-\e\right)}}\\
 & = & \frac{(e^{2\rho\left(-\lieAelement\right)\e}-e^{-2\rho\left(-\lieAelement\right)\e})+(e^{-2\rho\left(-\lieAelement\right)\left(T-\e\right)}-e^{-2\rho\left(-\lieAelement\right)\left(T+\e\right)})}{e^{2\rho\left(\lieAelement\right)\e}-e^{2\rho\left(\lieAelement\right)\left(T-\e\right)}}\\
 & = & \exd{(1+e^{-2\rho\left(-\lieAelement\right)T})}{\leq2}\cdot\frac{e^{2\rho\left(-\lieAelement\right)\e}-e^{-2\rho\left(-\lieAelement\right)\e}}{e^{-2\rho\left(-\lieAelement\right)\e}-e^{-2\rho\left(-\lieAelement\right)T}\cdot e^{2\rho\left(-\lieAelement\right)\e}}.
\end{eqnarray*}
So, the same computation as in the previous case shows that the last
expression is $\leq2\cdot\frac{3\cdot2\left|\rho\left(\lieAelement\right)\right|\e}{1/2}=12\cdot2\left|\rho\left(\lieAelement\right)\right|\e$
when $\e\leq\frac{1}{2\cdot2\rho\left(\lieAelement\right)}$ and $T\geq\frac{4}{2\rho\left(\lieAelement\right)}$.
\end{itemize}
Finally, when $2\rho\left(\lieAelement\right)=0$,
\[
\frac{\mu_{\Asub H{}}\brac{\projset{(\Asub HT)}{+\e}}}{\mu_{\Asub H{}}\brac{\projset{(\Asub HT)}{-\e}}}=\frac{T+2\e}{T-2\e}=1+\frac{4}{T-2\e}\e\leq1+4\e,
\]
 when $T-2\e>1$, which holds when for $\e<1/4$ and $T>1$.
\end{proof}

\section{\label{sec: effective Iwasawa and GI decompositions}The Iwasawa
roundomorphism}

In Subsection \ref{subsec: Plan of proof} we defined maps called
roundomorphisms, for which the pre-image of a well rounded family
is in itself well rounded. We also introduced a map $\roundo$ on
$\sl n\left(\RR\right)$, and the aim of this section is to prove
that $\roundo$ is a roundomorphism, allowing us to reduce the well
roundedness of families in $\sl n\left(\RR\right)$ to well roundedness
of their projections to $K$, $A^{\prime}$, $A^{\dprime}$ $N^{\prime}$
and $N^{\dprime}$. We begin by showing that (the more crude) map
$G\to K\times A\times N$ projecting to the Iwasawa coordinates of
a semisimple group is a roundomorphism. 

\subsection{Effective Iwasawa decomposition}

Recall that we let $G$ denote a semisimple Lie group with finite
center and Iwasawa decomposition $G=KAN$. The subgroups $K$, $A$
and $N$ are equipped with measures $\mu_{K}$, $\mu_{A}$ and $\mu_{N}$
respectively, such that for a given Haar measure $\mu_{G}$ of $G$,
$\mu_{G}=\mu_{K}\times\mu_{A}\times\mu_{N}$. Note that while $\mu_{K}$
and $\mu_{N}$ are Haar measures of their corresponding group, $\mu_{A}$
is not (see Definition \ref{def: H sbgrps of A} for $\mu_{A}$). 

Let $\lieA$ be the Lie algebra of $A$, $\lieN$ the Lie algebra
of $N$, and recall that $\left\{ \phi_{1},\ldots\phi_{\dimN}\right\} \subset\lieA^{*}$
are the positive (restricted) roots w.r.t. $\lieN$. Here the $\phi_{i}$'s
are not necessarily different, but with multiplicities.  For $a=\exp\left(\lieAelement\right)\in A$
define 
\begin{equation}
\begin{array}{c}
\maxexp\left(\lieAelement\right):=\max_{i}\left\{ -\phi_{i}\left(\lieAelement\right),0\right\} ,\\
\ff a:=\cnorm^{2}e^{\maxexp\left(\lieAelement\right)},\qquad\quad
\end{array}\label{eq: m(H) and err(a)}
\end{equation}
where $\cnorm\geq1$ is a constant which depends on the specific
choice of norm $\norm{\cdot}$ on $\lieN$ in the following manner:
$\left(1/\cnorm\right)\norm{\lieNelement}_{\infty}\leq\norm{\lieNelement}\leq\cnorm\norm{\lieNelement}_{\infty}$
for every $\lieNelement\in\lieN$.
\begin{rem}
\label{rem: err sub multiplicative}Notice that $\ff{\cdot}$ is sub-multiplicative:
\[
\ff{a_{1}a_{2}}\leq\ff{a_{1}}\ff{a_{2}}.
\]
 
\end{rem}

The goal of this section is to prove the following proposition.
\begin{prop}[Effective Iwasawa decomposition]
\label{prop: Effective Iwasawa decomposition} Let $G$ be a semisimple
Lie group with finite center. The diffeomorphism defining the Iwasawa
decomposition $\roundogen:G\to K\times A\times N$, $\roundogen\left(g\right)=\left(k,a,n\right)$
is a $f$-roundomorphism w.r.t. $\nbhd{\e}G$, $\nbhd{\e}{K\times A\times N}$
and 
\[
f\left(g\right)\porsmall C\left(n\right)\cdot\ff a^{2},
\]
where $C\left(n\right)=\left\Vert \mbox{Ad}\,n\right\Vert _{\mbox{op}}$\textup{. }
\end{prop}

The proof requires the following auxiliary lemma. 
\begin{lem}
\label{lem: A Conj  N }Let $N^{-}:=\Theta\left(N\right)$, where
$\Theta$ is a global Cartan involution compatible with the given
Iwasawa decomposition. Then $A$ acts on both $N,N^{-}$ by conjugation
such that the following holds: 
\begin{align*}
a^{-1}\nbhd{\e}Na & \subseteq\nbhd{\ff a\e}N,\\
a\nbhd{\e}{N^{-}}a^{-1} & \subseteq\nbhd{\ff a\e}{N^{-}}.
\end{align*}
\end{lem}

\begin{proof}
 First we introduce some notations. Let $\lieNelement_{1},\ldots,\lieNelement_{\dimN}$
be the corresponding linearly independent eigenvectors in $\lieG$
of $\phi_{1},\ldots\phi_{\dimN}$ respectively. Denote 
\[
n_{\underline{x}}=n_{\sbrac{x_{1},\ldots,x_{\dimN}}}:=\exp(\,{\scriptstyle \sum\limits _{i=1}^{p}}x_{i}Z_{i}).
\]
Then
\[
N=\{n_{\underline{x}}:\underline{x}\in\RR^{\dimN}\};\quad N^{-}=\{n_{\underline{x}}^{-}=\Theta(n_{\underline{x}}):\underline{x}\in\RR^{\dimN}\}.
\]

For every $\lieAelement\in\lieA$ and $\lieNelement\in\lieN$ the
action of $a^{-1}=\exp\left(-\lieAelement\right)$ on $\exp\left(\lieNelement\right)$
is given by 
\[
\mbox{Conj}_{\,\exp\left(-\lieAelement\right)}\left(\exp\left(\lieNelement\right)\right)=\exp(\Ad{e^{-\lieAelement}}\left(\lieNelement\right))=\exp(e^{\ad{-\lieAelement}}\left(\lieNelement\right)).
\]
In particular, if $\lieNelement=\sum_{i=1}^{\dimN}x_{i}\lieNelement_{i}$
then (since $\ad{-\lieAelement}\left(\lieNelement_{i}\right)=\left[-\lieAelement,\lieNelement_{i}\right]=\phi_{i}\left(-\lieAelement\right)\cdot\lieNelement_{i}$
and therefore $e^{\ad{-\lieAelement}}\left(\lieNelement_{i}\right)=e^{\phi_{i}\left(-\lieAelement\right)}\cdot\lieNelement_{i}$):
\begin{align*}
\mbox{Conj}_{\,\exp\left(-\lieAelement\right)}(\exp\brac{\ensuremath{{\scriptstyle \sum\limits _{i=1}^{p}}x_{i}Z_{i}}}) & =\exp(\Ad{e^{-\lieAelement}}\brac{\ensuremath{{\scriptstyle \sum\limits _{i=1}^{p}}x_{i}Z_{i}}})=\exp(\ensuremath{{\scriptstyle \sum\limits _{i=1}^{p}}x_{i}}\Ad{e^{-\lieAelement}}\left(\lieNelement_{i}\right))\\
= & \exp(\ensuremath{{\scriptstyle \sum\limits _{i=1}^{p}}x_{i}}\cdot e^{\ad{-\lieAelement}}\left(\lieNelement_{i}\right))=\exp(\ensuremath{{\scriptstyle \sum\limits _{i=1}^{p}}x_{i}}\cdot e^{\phi_{i}\left(-\lieAelement\right)}\cdot\lieNelement_{i}).
\end{align*}
As a result, 
\[
a^{-1}\cdot n_{\underline{x}}\cdot a=\exp\left(-\lieAelement\right)\cdot n_{\underline{x}}\cdot\exp\left(\lieAelement\right)=n_{[x_{1}e^{\phi_{1}\left(-\lieAelement\right)},\ldots,x_{\dimN}e^{\phi_{\dimN}\left(-\lieAelement\right)}]}=n_{\dbrac{\underline{x},(e^{-\phi_{i}\left(\lieAelement\right)})_{i=1}^{\dimN}}}.
\]
 If $a^{-1}\cdot n_{\underline{x}}\cdot a=n_{\underline{y}}$, then
for $n_{\underline{x}}\in\nbhd{\e}N$ and $\norm x<\e$ it holds for
$\underline{y}$ that 
\[
\lilnorm{\underline{y}}=\lilnorm{\dbrac{\underline{x},(e^{-\phi_{i}\left(\lieAelement\right)})_{i=1}^{\dimN}}}\leq\cnorm\lilnorm{\dbrac{\underline{x},(e^{-\phi_{i}\left(\lieAelement\right)})_{i=1}^{\dimN}}}_{\infty}\leq\cnorm\lilnorm{\underline{x}}_{\infty}\lilnorm{(e^{-\phi_{i}\left(\lieAelement\right)})_{i=1}^{\dimN}}_{\infty}\leq\e\cdot\ff a.
\]
Thus, 
\[
a^{-1}\nbhd{\e}Na\subseteq\nbhd{\ff a\e}N.
\]
The second part follows from the first by applying $\Theta$ (the
global Cartan involution) to the above.
\end{proof}
As a final preparation to the proof of Proposition \ref{prop: Effective Iwasawa decomposition},
we list some properties of the families of identity neighborhoods
$\nbhd{\e}G=\exp_{G}\left(\ball{\e}{}\right)$ appearing in the statement
of the proposition. We let $G$ be a general Lie group. \textbf{Then
$\nbhd{\e}G$ has the following properties}:
\begin{enumerate}
\item \label{lem: Conjugation inflates by the norm of Ad}(\textbf{Conjugation
by $g$ dilates by $\left\Vert \Ad g\right\Vert $}) If the Lie algebra
of $G$ is $\mathfrak{g}$ then for every $g\in G$, 
\[
g^{-1}\,\nbhd{\e}G\,g\subseteq\nbhd{\e\cdot\left\Vert \Ad g\right\Vert _{\text{op}}}G=\exp\cbrac{\lieNelement\in\mathfrak{g}:\left\Vert \lieNelement\right\Vert \leq\e\cdot\left\Vert \Ad g\right\Vert _{\text{op}}},
\]
where $\norm{\cdot}$ is any euclidean norm on $\mathfrak{g}$ and
$\left\Vert \cdot\right\Vert _{\text{op}}$ is the norm on the space
of linear $\mathfrak{g}$-operators. 
\item (\textbf{Connectivity}) $\nbhd{\e}G$ is a connected subset of $G$. 
\item \label{lem: Connectivity + additivity of coord. balls}(\textbf{Additivity})
for small enough $\e$ and $\delta$, there exists $c>0$ such that
$\nbhd{\e}G\nbhd{\dl}G\subseteq\nbhd{c\left(\e+\delta\right)}G$.
\item \label{def: equivalence of nbhds}\label{rem: coordinate balls equivalent}(\textbf{Decomposition
of $G$ allows decomposition of $\nbhd{\e}G$}) If $G$ is semi-simple
(as it is in Proposition \ref{prop: Effective Iwasawa decomposition}),
hence has Iwasawa decomposition, the family $\nbhd{\e}G$ is equivalent
to the family $\nbhd{\e}K\nbhd{\e}A\nbhd{\e}N=\exp_{K\times A\times N}\left(\ball{\e}{}\right)$
of identity neighborhoods in $G$ in the sense that there exist $\e_{1},c,C>0$
such that for every $0<\e<\e_{1}$ it holds that $\nbhd{c\e}G\subseteq\nbhd{\e}K\nbhd{\e}A\nbhd{\e}N\subseteq\nbhd{C\e}G$.
 Using Bruhat coordinates on identity neighborhood in $G$, the family
$\nbhd{\e}G$ is also equivalent to the family $\nbhd{\e}M\nbhd{\e}{N^{-}}\nbhd{\e}A\nbhd{\e}N$,
where $M=\left(Z_{K}\left(A\right)\right)_{0}$; we may assume that
the parameter $\e_{1}$ is the same. 
\end{enumerate}
\begin{proof}[proof of Proposition \ref{prop: Effective Iwasawa decomposition}]
Clearly, we only need to show that 
\[
\roundo\left(\nbhd{\e}Gg\nbhd{\e}G\right)\subseteq\nbhd{f\e}{K\times A\times N}\,\roundo(g)\nbhd{f\e}{K\times A\times N},
\]
where $f$ is as in the statement. This will be accomplished in three
steps. 

\textbf{Step 1: Left perturbations.} ~According to Properties \ref{rem: coordinate balls equivalent}
and \ref{lem: Conjugation inflates by the norm of Ad}, there exist
$\e_{1},c_{1},c_{2}>0$ such that for all $\e<\e_{1}$ 
\begin{align*}
\nbhd{\e}Gkan & =k\left(k^{-1}\nbhd{\e}Gk\right)an\subseteq k\nbhd{c_{1}\e}Gan\subseteq k\nbhd{c_{2}\e}K\nbhd{c_{2}\e}A\nbhd{c_{2}\e}Nan\\
 & =k\nbhd{c_{2}\e}K\cdot\nbhd{c_{2}\e}Aa\cdot a^{-1}\nbhd{c_{2}\e}Nan.
\end{align*}
 By Lemma \ref{lem: A Conj  N }, $a^{-1}\nbhd{\e}Na\subseteq\nbhd{\ff a\e}N$,
hence
\[
\roundo\left(\nbhd{\e}Gg\right)\subseteq\nbhd{c_{2}\ff a\e}{K\times A\times N}\roundo\left(g\right)\nbhd{c_{2}\e}{K\times A\times N}.
\]

\textbf{Step 2: Right perturbations.} ~By Properties \ref{rem: coordinate balls equivalent}
(for the Bruhat coordinates) and \ref{lem: Conjugation inflates by the norm of Ad},
\begin{align*}
kan\nbhd{\e}G & =ka\left(n\nbhd{\e}Gn^{-1}\right)n\subseteq ka\nbhd{C\left(n\right)\e}Gn\\
 & \subseteq ka\cdot\nbhd{c_{3}C\left(n\right)\e}M\nbhd{c_{3}C\left(n\right)\e}{N^{-}}\nbhd{c_{3}C\left(n\right)\e}A\nbhd{c_{3}C\left(n\right)\e}N\cdot n\\
 & =k\nbhd{c_{3}C\left(n\right)\e}M\cdot a\cdot\nbhd{c_{3}C\left(n\right)\e}{N^{-}}\nbhd{c_{3}C\left(n\right)\e}A\nbhd{c_{3}C\left(n\right)\e}Nn\\
 & =k\nbhd{c_{3}C\left(n\right)\e}M\left(a\nbhd{c_{3}C\left(n\right)\e}{N^{-}}a^{-1}\right)a\,\nbhd{c_{3}C\left(n\right)\e}A\nbhd{c_{3}C\left(n\right)\e}Nn.
\end{align*}
\textcolor{black}{{} }By Lemma \ref{lem: A Conj  N }, $a\nbhd{c_{3}C\left(n\right)\e}{N^{-}}a^{-1}\subseteq\nbhd{c_{3}C\left(n\right)\ff a\e}{N^{-}}\subseteq\nbhd{c_{3}C\left(n\right)\ff a\e}G$.
Moreover, for $\e\leq\e_{1}/(c_{3}C\left(n\right)\ff a)$ we have
\begin{align*}
\nbhd{c_{3}C\left(n\right)\e}M\nbhd{c_{3}C\left(n\right)\ff a\e}G & \subseteq\nbhd{c_{3}C\left(n\right)\e}K\nbhd{c_{4}C\left(n\right)\ff a\e}K\nbhd{c_{4}C\left(n\right)\ff a\e}A\nbhd{c_{4}C\left(n\right)\ff a\e}N\\
 & \subseteq\nbhd{c_{5}C\left(n\right)\ff a\e}K\nbhd{c_{4}C\left(n\right)\ff a\e}A\nbhd{c_{4}C\left(n\right)\ff a\e}N.
\end{align*}
As a result,
\[
kan\nbhd{\e}G\subseteq k\nbhd{c_{5}C\left(n\right)\ff a\e}K\nbhd{c_{4}C\left(n\right)\ff a\e}A\nbhd{c_{4}C\left(n\right)\ff a\e}N\,a\,\nbhd{c_{3}C\left(n\right)\e}A\nbhd{c_{3}C\left(n\right)\e}Nn.
\]
Let $a_{\e}\in\nbhd{c_{3}C\left(n\right)\e}A$. Write $a_{1}=aa_{\e}$.
By sub-multiplicativity of $\ff{\cdot}$ (Remark \ref{rem: err sub multiplicative})
 we get,
\[
\nbhd{c_{4}C\left(n\right)\ff a\e}Na_{1}=a_{1}a_{1}^{-1}\nbhd{c_{4}C\left(n\right)\ff a\e}Na_{1}\subseteq a_{1}\nbhd{c_{4}C\left(n\right)\ff a\ff{a_{1}}\e}N\subseteq a_{1}\nbhd{c_{5}C\left(n\right)\ff a^{2}\e}N.
\]
Combining all of the above, we conclude
\[
kan\nbhd{\e}G\subseteq k\nbhd{c_{5}C\left(n\right)\ff a\e}K\nbhd{c_{4}C\left(n\right)\ff a\e}Aa\nbhd{c_{3}C\left(n\right)\e}A\nbhd{c_{5}C\left(n\right)\ff a^{2}\e}N\nbhd{c_{3}C\left(n\right)\e}Nn.
\]
In other words,
\[
r\left(g\nbhd{\e}G\right)\subseteq\nbhd{c_{6}C\left(n\right)\left(\ff a^{2}+1\right)\e}{K\times A\times N}r\left(g\right)\nbhd{c_{6}C\left(n\right)\left(\ff a+1\right)\e}{K\times A\times N}.
\]

\paragraph*{Step 3: Combining left and right perturbations. }

Finally, using the additivity property \ref{lem: Connectivity + additivity of coord. balls}
on $\nbhd{\e}{K\times A\times N}$ we conclude that
\[
r\left(\nbhd{\e}Gg\nbhd{\e}G\right)\subseteq\nbhd{f\left(g\right)\e}{K\times A\times N}r\left(g\right)\nbhd{f\left(g\right)\e}{K\times A\times N}
\]
for $\e\leq1/f\left(g\right)$ and $f\left(g\right)\porsmall C\left(n\right)\cdot\ff a^{2}$. 
\end{proof}

\subsection{Effective Refined Iwasawa decomposition}

After having established that the map $G\to K\times A\times N$ projecting
to the $KAN$ coordinates is a roundomorphism, we deduce it for the
$KA^{\prime}A^{\dprime}N$ and RI decompositions as well (Corollary
\ref{cor: roundo of GI decomposition}). 
\begin{lem}
\label{lem: roundo for decomposing N}Let N be a connected nilpotent
Lie group with Haar measure $\mu_{N}$. Suppose that $N=N_{1}\ltimes N_{2}$,
where $N_{1}$ and $N_{2}$ are two closed subgroups of $N$ equipped
with Haar measures $\mu_{N_{1}}$ and $\mu_{N_{2}}$.Then each element
in $N$ can be decomposed in a unique way as $n=n_{1}n_{2}$, and
the map 
\[
r\left(n\right)=\left(n_{1},n_{2}\right)\in N_{1}\times N_{2}
\]
 is a $f$-roundomorphism for some continuous $f:N\to\RR^{\geq0}$.
If $N$ is abelian, then  $f\equiv1$.
\end{lem}

\begin{proof}
The first condition in the definition of a roundomorphism is a consequence
of the nilpotency assumption (see \cite[Corollary 8.31, Theorem 8.32]{Knapp}).
The second condition, local Lipschitzity, follows from the fact that
$\roundo$ in the lemma is a diffeomorphism (see \cite[Prop. 4.6]{HK_WellRoundedness}).
\end{proof}
\begin{cor}[Effective RI decomposition]
\label{cor: roundo of GI decomposition} Let $G$ be a semisimple
Lie group with finite center and Iwasawa decomposition $G=KAN$. Assume
that $N^{\prime}$ and $N^{\dprime}$ are closed subgroups of $N$
equipped with Haar measures $\mu_{N^{\prime}},\mu_{N^{\dprime}}$
such that $N=N^{\dprime}\ltimes N^{\prime}$ and $\mu_{N}=\mu_{N^{\prime}}\times\mu_{N^{\dprime}}$.
Similarly, let $A^{\prime}$ and $A^{\dprime}$ be closed subgroups
of $A$ such that $A=A^{\prime}\times A^{\dprime}$ and $\mu_{A}=\mu_{A^{\prime}}\times\mu_{A^{\dprime}}$.
The projection map
\[
\roundo:G\to K\times A^{\prime}\times A^{\dprime}\times N^{\dprime}\times N^{\prime},\quad\roundo\left(g\right)=\left(k,a^{\prime},a^{\dprime},n^{\dprime},n^{\prime}\right)
\]
is an $f$-roundomorphisms w.r.t. 
\[
f\left(g\right)\porsmall c\left(n^{\prime},n^{\dprime}\right)\cdot\ff{a^{\prime}a^{\dprime}}^{2}
\]
where $c\left(n^{\prime},n^{\dprime}\right)$ is a continuous functions
on $N^{\prime}\times N^{\dprime}$. 
\end{cor}

\begin{proof}
This follows from Proposition \ref{prop: Effective Iwasawa decomposition}
combined with Lemma \ref{lem: roundo for decomposing N} and the fact
that a composition of roundomorphisms is a roundomorphism (\cite[Lemma 4.5]{HK_WellRoundedness}).
\end{proof}

\subsection{\label{subsec: Computing err(a)} Computing $f$ for the Iwasawa
roundomorphism }

Assume the setting of Corollary \ref{cor: roundo of GI decomposition},
where we have shown that ``the Refined Iwasawa decomposition map'',
$\roundo$, is a roundomorphism, and expressed the error function
$f$ in terms of $a^{\prime},a^{\dprime}$. We now proceed to compute
$f$ under some assumptions on $A^{\prime},A^{\dprime}$, which are
satisfied for the $A^{\prime},A^{\dprime}$ introduced in Section
\ref{sec: RI components} and are relevant for the counting problem
in Proposition \ref{prop: Counting with Hexagons (A counting)}. The
discussion is concluded in Lemma \ref{lem: err(a) in SLn}, where
we deduce the correct $f$ for our counting problem, and it will be
used in the proof of the proposition.

Denote $\dimA:=\dim\left(A\right)$. Let $\lieAelement_{1}^{\prime},\dots,\lieAelement_{\dimAp}^{\prime},\lieAelement_{1}^{\dprime},\dots,\lieAelement_{\dimA-\dimAp}^{\dprime}$
be a basis for $\lieA$, and denote 
\[
A^{\prime}=\Asub{\lieAvec^{\prime}}{},A^{\dprime}=\Asub{\lieAvec^{\dprime}}{},
\]
 where $\lieAvec^{\prime}=\brac{\lieAelement_{1}^{\prime},\dots,\lieAelement_{\dimAp}^{\prime}}$
and $\lieAvec^{\dprime}=\brac{\lieAelement_{1}^{\dprime},\dots,\lieAelement_{\dimA-\dimAp}^{\dprime}}$.
We compute $f$ under the following assumption\textcolor{green}{.}:
\begin{assumption}
\label{assum: assumption that 2rho(H'')<0}For every $i=\ensuremath{1,\dots,\dimAp}$
assume $\lieAelement_{i}^{\prime}\in\overline{\mathcal{C}}-\left\{ 0\right\} $,
where $\mathcal{C}$ is the positive Weil chamber w.r.t. $N$. For
every $\ensuremath{i=1,\dots,\dimA-\dimAp}$, assume $2\rho\brac{\lieAelement_{i}^{\dprime}}<0$.
The latter can be achieved, for example, by requiring that $\lieAelement_{i}^{\dprime}\in-\overline{\mathcal{C}}-\left\{ 0\right\} $
for every $i$. 
\end{assumption}

\begin{notation}
\label{nota: m of H''}For $\lieAvec=\brac{\lieAelement_{1},\dots,\lieAelement_{\topindex}}$
and $\maxexp\brac{\lieAelement_{j}}$ as defined in Formula (\ref{eq: m(H) and err(a)}),
denote 
\[
\maxexp_{\lieAvec}=\underset{j}{\max}\cbrac{\maxexp\brac{\lieAelement_{j}}}=\max_{i,j}\cbrac{-\phi_{i}\brac{\lieAelement_{j}},0}.
\]
\end{notation}

The content of the following Lemma is that under assumption \ref{assum: assumption that 2rho(H'')<0},
the error function of the Iwasawa roundomorphism is only affected
by the $A^{\dprime}$ component of $A$. 
\begin{lem}
\label{lem: err(a) in SLn}Under assumption \ref{assum: assumption that 2rho(H'')<0},
$a_{\underline{t}}^{\prime}a_{\svec}^{\dprime}=\exp(\underline{t}\cdot\lieAvec^{\prime}+\svec\cdot\lieAvec^{\dprime})$
satisfies that $\ff{a_{t}^{\prime}a_{\svec}^{\dprime}}\leq\cnorm^{2}e^{\maxexp_{\lieAvec^{\dprime}}\sums}$,
where $\sums:=\svec\cdot\left(1,\ldots,1\right)=\sum s_{i}$. In particular,
for $G=\sl n\left(\RR\right)$ and $A^{\prime},A^{\dprime}$ as defined
in Section \ref{sec: RI components}, $\ff{a_{t}^{\prime}a_{\svec}^{\dprime}}\leq\cnorm^{2}e^{\sums}$. 
\end{lem}

\begin{proof}
If the elements $\lieAelement_{j}^{\prime}$ are in $\overline{\mathcal{C}}-\left\{ 0\right\} $,
then $\maxexp(\underline{t}\cdot\lieAvec^{\prime}+\svec\cdot\lieAvec^{\dprime})=\maxexp(\svec\cdot\lieAvec^{\dprime})\leq\maxexp_{\lieAvec^{\dprime}}\sums$.
As for $G=\sl n\left(\RR\right)$ and $A^{\prime},A^{\dprime}$ as
defined in Section \ref{sec: RI components}, the basis elements in
$\lieA$ that correspond to $A^{\prime},A^{\dprime}$ are $\lieAelement^{\prime}=(1/\brac{n-1},\dots,1/\brac{n-1},-1)$
and $\lieAelement_{j}^{\dprime}=(-e_{j,j}+e_{j+1,j+1})/2$ for $j=1,\dots,n-2$
(see Example \ref{exa: SL(n,R)  1}). With the positive roots as in
Example \ref{exa: SL(n,R)  1}, we have that $\maxexp_{\lieAelement_{j}^{\dprime}}=\max\cbrac{0,1=\frac{1}{2}-(-\frac{1}{2}),\frac{1}{2}-0,-\frac{1}{2}-0}=1$
for every $j=1,\dots,n-2$, hence $\maxexp_{\lieAvec^{\dprime}}=1$. 
\end{proof}

\section{\label{sec: The base sets are LWR}The base sets}

We return our focus to $G=\sl n\left(\RR\right)$. The aim of this
section is to prove that $\roundo\brac{\trunc{\brac{\wc_{\latset}}}{\Svec}}$
is LWR, and therefore (see second step in the plan on Section \ref{subsec: Plan of proof})
the base set in 
\[
\roundo\brac{\brac{\shortdom}_{T}^{\Svec}\brac{\latset,\a}}=\bigcup_{\substack{\left(\brac{k,a^{\dprime},n^{\dprime}},a^{\prime}\right)\in\\
\roundo(\trunc{\latset}{\Svec})\times A_{T}^{\prime}
}
}\left(k,a^{\dprime},n^{\dprime},a^{\prime}\right)\times\parby{N^{\prime}}{\ellipse{a^{\dprime}n^{\dprime}}{\a}}
\]
is LWR independently of $\Svec$. From now on, $H^{\prime}$ and $H_{j}^{\dprime}$
for $j=1,\ldots,n-2$ are as in Example \ref{exa: SL(n,R)  1}. 

\begin{lem}
\label{lem: pushforward of Fm is LWR AxN}For any $\latset\subseteq K\symfund{n-1}\subset\wc$
that is a BCS, the set $\roundo\left(\latset\right)$ is LWR in $K\times A^{\dprime}\times N^{\dprime}$.
As a result,  $\roundo(\trunc{\latset}{\Svec})$) is LWR with parameters
that do not depend on $\Svec$.
\end{lem}

\begin{rem}
\textbf{}The set $K\symfund{n-1}$ (resp.\ $\symfund{n-1}$) itself
is \textbf{not }LWR in $\wc$ (resp.\ $P^{\dprime}$), only its image
under $\roundo$ is.
\end{rem}

Since Lemma \ref{lem: pushforward of Fm is LWR AxN} is about counting
in a group that is a direct product, it is proved by working in each
of the components separately. Among the two components $A^{\dprime}$
and $N^{\dprime}$, the problematic one is of course $A^{\dprime}$;
the role of the following two lemmas is to handle this component. 

\begin{lem}
\label{lem: height of F_n bounded from below}The projection to the
$\Asub{H_{i}^{\dprime}}{}$ component of $\symfund{n-1}$ is bounded
from below for every $i=1,\dots,n-2$. 
\end{lem}

\begin{proof}
We need to show that for every $\lieAelement\in\mathfrak{a}^{\dprime}$
such that $\exp\left(\lieAelement\right)\cdot n^{\dprime}\in\symfund{n-1}$,
it holds that the coefficients of $\lieAelement$ in its presentation
of a linear combination of $\{\lieAelement_{j}^{\dprime}\}$ are bounded
from below. These coefficients are given by linear functionals: $\lieAelement=\sum_{j=1}^{n-2}\psi_{j}\left(\lieAelement\right)\lieAelement_{j}^{\dprime}$
(actually, $\left\{ \psi_{j}\right\} _{j=1}^{n-2}\subset\brac{\mathfrak{a}^{\dprime}}^{*}$
is the dual basis to $\{\lieAelement_{j}^{\dprime}\}_{j=1}^{n-2}\subset\mathfrak{a}^{\dprime}$).
Denote $\phi_{i}:=\phi_{i+1,i}$ where $\left\{ \phi_{i,j}\right\} $
are the roots for $\sl n\left(\RR\right)$ defined in Example \ref{exa: SL(n,R)  1}.
Clearly $\left\{ \phi_{i}\right\} $ form a basis to $\brac{\mathfrak{a}^{\dprime}}^{*}$,
and by Lemma \ref{lem: BLC. facts about z in RS domain} they satisfy
that $\phi_{i}\left(\lieAelement\right)\geq\ln\brac{\sqrt{3}/2}$
for every $i=1,\dots,n-2$ and $\lieAelement$ as above. It is therefore
sufficient to show that in the presentation of every $\psi_{j}$ as
a linear combination of $\left\{ \phi_{i}\right\} $, the coefficients
are non-negative. Write $\psi_{i}=2\sum_{j=1}^{n-2}x_{i,j}\phi_{j}$
and evaluate at each of $\lieAelement_{1}^{\dprime},\dots,\lieAelement_{n-2}^{\dprime}$
we obtain the following system of linear equations
\[
\left[\begin{smallmatrix}2 & -1 & 0 &  & 0\\
-1 & 2 & -1\\
 & -1 & \ddots & \ddots\\
 &  & \ddots & 2 & -1\\
0 &  &  & -1 & 2
\end{smallmatrix}\right]\left[\begin{array}{c}
x_{i,1}\\
\vdots\\
x_{i,n-2}
\end{array}\right]=e_{i}.
\]
A computation shows that the solution $\brac{x_{i,j}}_{j=1}^{n-2}$
is indeed non-negative. 
\end{proof}
To see how the following lemma concerns the $A^{\dprime}$ component,
notice that the group $\brac{A^{\lieAelement_{i}^{\dprime}},d\mu_{A_{\lieAelement_{i}^{\dprime}}}}$
is measure preserving isomorphic to $\brac{\RR^{>0},\cdot,dx/x^{2}}$
for every $i=1,\dots,n-2$. 
\begin{lem}
The map $\psi:\brac{\RR^{>0},\cdot,dx/x^{2}}\to\brac{\RR,+,\One_{\left(0,\infty\right)}\left(x\right)\cdot dx}$
given by $\psi\left(x\right)=1/x$ is a $f$-roundomorphism with $f\left(x\right)=2/x$. 
\end{lem}

\begin{proof}
A standard computation  shows that $\varphi$ pushes $dx/x^{2}$
to $\One_{\left(0,\infty\right)}\left(x\right)\cdot dx$. Moreover,
for $\e<1/12$:
\begin{align*}
\psi(\nbhd{\e}{\RR^{>0}}x\nbhd{\e}{\RR^{>0}}) & \subseteq\psi\left(x\cdot\left[1-3\e,1+3\e\right]\right)\subseteq x^{-1}\cdot\left[1-4\e,1+4\e\right]\\
 & =\psi\left(x\right)+2f\left(x\right)\left[-\e,\e\right]=\nbhd{f\e}{\RR}\psi\left(x\right)\nbhd{f\e}{\RR}.\tag*{\qedhere}
\end{align*}
\end{proof}
\begin{proof}[Proof of Lemma \ref{lem: pushforward of Fm is LWR AxN}]
We start by showing that $\roundo\brac{K\symfund{n-1}}\subset K\times A^{\dprime}\times N^{\dprime}$
is LWR. Consider the map 
\[
\varphi:K\times A^{\dprime}\times N^{\dprime}\to K\times\brac{\RR,+,\One_{\left(0,\infty\right)}\left(x\right)\cdot dx}^{\left(n-2\right)}\times N^{\dprime},
\]
induced by the map given in the previous Lemma. It is an $f$-roundomorphism
with $f\left(k,x_{1},\dots,x_{n-2},n^{\dprime}\right)=\frac{2^{n-2}}{x_{1}\cdots x_{n-2}}$.
Since, by Lemma \ref{lem: height of F_n bounded from below}, the
projection to $A^{\lieAelement_{i}^{\dprime}}$ of $r\left(\symfund{n-1}\right)$
 (hence of $r\left(K\symfund{n-1}\right)$) is bounded from below
for every $i$, we conclude that $\varphi\left(r\left(K\symfund{n-1}\right)\right)$
is a bounded set. 

By Proposition \ref{prop: spread models that we need}, $\del\varphi\left(\roundo\left(\symfund{n-1}\right)\right)\subseteq\varphi\left(\roundo\del\left(\symfund{n-1}\right)\right)\cup K\times\del\brac{\RR_{>0}^{\times(n-2)}}\times N^{\dprime}$
is contained in a finite union of lower dimensional embedded submanifolds,
and therefore so is the boundary of $K\symfund{n-1}$; so, according
to Remark \ref{rem: LWR set}, $\varphi\left(\roundo\left(K\symfund{n-1}\right)\right)$
is LWR. Finally, since $f|_{\roundo\left(K\symfund{n-1}\right)}$
is bounded, then by Proposition \ref{cor: roundo to a produt grp}
we conclude that $\roundo\left(K\symfund{n-1}\right)$ is LWR. 

As the boundary of $\latset$ is also contained in a finite union
of lower dimensional embedded submanifolds, then $\roundo\left(\latset\right)$
is LWR by the same considerations. 

We now turn to prove that the set $\roundo(K\trunc{\symfund{n-1}}{\Svec})$
is LWR; this set is the intersection of $\roundo\left(K\symfund{n-1}\right)$
with the set $K\times A_{\Svec}^{\dprime}\times\pi_{N^{\dprime}}\left(\symfund{n-1}\right)$,
where $\pi_{N^{\dprime}}\left(\symfund{n-1}\right)$ is the projection
of $\symfund{n-1}$ to $N^{\dprime}$. According to \cite[Lemma 3.4]{HK_WellRoundedness},
LWR property is maintained under intersections, and so it is sufficient
to show that $K\times A_{\Svec}^{\dprime}\times\pi_{N^{\dprime}}\left(\symfund{n-1}\right)$
is LWR. This is indeed the case since $A_{\Svec}^{\dprime}$ is LWR
with a parameter independent of $\Svec$ (by Proposition \ref{prop: A cubes are well rounded}),
$\pi_{N^{\dprime}}\left(\symfund{n-1}\right)$ and $K$ are LWR since
they are bounded BCS (see Lemma \ref{lem: BLC. facts about z in RS domain}),
and LWR is maintained under taking products by Remark \ref{rem: product of LWR is LWR}.
Thus $\roundo\left(\trunc{\latset}{\Svec}\right)=\roundo\left(\latset\right)\cap\roundo(K\trunc{\symfund{n-1}}{\Svec})$
is again LWR, as the intersection of two such sets. 
\end{proof}

\section{\label{sec: Family for gcd solution is BLC}The family $\protect\fdomellipse{\protect\a}{\protect\roundo\protect\brac{\protect\symfund{n-1}}}$
is BLC }

The goal of this Section is to show that the family $\fdomellipse{\a}{\roundo\brac{\symfund{n-1}}}$
is BLC for all $0<\a\leq1$, according to the plan of proof for Proposition
\ref{prop: Counting with Hexagons (A counting)}, described in Subsection
\ref{subsec: Plan of proof}. 

The domain $\symfund{n-1}$ is a subset of $P^{\dprime}$, which is
a diffeomorphic and group isomorphic copy of $P_{n-1}$, the group
of $\left(n-1\right)\times\left(n-1\right)$ upper triangular matrices
with positive diagonal entries and determinant $1$. To simplify the
notation, we consider the situation in general dimension with $\symfund m\subset P_{m}$,
and write $P_{m}=A_{m}N_{m}$ where $A_{m}$ is the diagonal subgroup
of $\sl m\left(\RR\right)$ and $N_{m}$ is the subgroup of upper
triangular unipotent matrices. In particular, we abandon the notations
of $P^{\dprime},A^{\dprime},N^{\dprime}$ and keep in mind that for
our purpose, one takes $m=n-1$. The roundomorphism $\roundo$ introduced
in Corollary \ref{cor: roundo of GI decomposition} now becomes 
\begin{align*}
\roundo:P_{m} & \to A_{m}\times N_{m}\\
z=an & \mapsto\left(a,n\right).
\end{align*}
Let us recall some further notations that were introduced previously,
perhaps with $n-1$ instead of $m$. For $z=[z_{1}|\cdots|z_{m}]\in\symfund m$
we let $\lat_{z}$ denote the lattice spanned by the columns of $z$,
and consider the linear map $\linmap_{z}:\mathbb{R}^{m}\to\mathbb{R}^{m}$
given by $z_{j}\mapsto e_{j}$ for every $j=1,\ldots,m$. Note that
$\linmap_{z}$ maps $\lat_{z}$ to $\ZZ^{n}$. 
\begin{rem}
\label{rem: L_z}$\linmap_{z}^{-1}\left(x\right)=zx$ for every $x\in\mathbb{R}^{m}$
(i.e., the linear map $\linmap_{z}^{-1}$ is given by the matrix $z$).
Hence, $\linmap_{z}\left(zx\right)=x$, namely the image under $\linmap_{z}$
of a vector is its coordinates w.r.t. the basis $\left\{ z_{1},\ldots z_{m}\right\} $,
which is also clear from the definition of $\linmap_{z}$.
\end{rem}

We begin by considering the case of $\a=1$.
\begin{prop}
\label{prop: Hexagons are BLC}The family $\fam{\fdomhex}{\roundo\brac{\symfund m}}=\left\{ \hex{an}\right\} _{\left(a,n\right)\in\roundo\left(\symfund m\right)}=\left\{ \hex z\right\} _{\roundo\left(z\right)\in\roundo\left(\symfund m\right)}$
is \textbf{BLC} w.r.t. $\nbhd{\e}{A_{m}\times N_{m}}$. 
\end{prop}

In the proof, Lemma \ref{lem: BLC. facts about z in RS domain}, will
play a key role. In particular, we note that the last part of this
lemma implies shrinking property of conjugation of upper triangular
matrices by elements of $\symfund m$, and we formulate this in the
following corollary. 

\begin{cor}
\label{cor: conjugation by z shrinks}Let $[z_{1}|\cdots|z_{m}]=z=a_{z}n_{z}\in\symfund m$.
Then for any upper triangular matrix $p$,

\begin{enumerate}
\item $\lilnorm{a_{z}pa_{z}^{-1}}\porsmall\norm p$;
\item $\lilnorm{zpz^{-1}},\norm{\trans zp\mtrans z}\porsmall\norm p$.
\end{enumerate}
\end{cor}

\begin{proof}
part 1 follows from the fact that if $i\leq j$ then $a_{i}\porsmall a_{j}$
and therefore 
\[
\left|a_{i}p_{i,j}a_{j}^{-1}\right|=\left|a_{i}\right|\left|p_{i,j}\right|\left|a_{j}\right|^{-1}\prec\left|a_{j}\right|\left|p_{i,j}\right|\left|a_{j}\right|^{-1}=\left|p_{i,j}\right|.
\]
Since $p_{ij}=0$ for $i>j$, then 
\[
\lilnorm{a_{z}pa_{z}^{-1}}\prec\lilnorm{a_{z}pa_{z}^{-1}}_{1}\prec\norm p_{1}\porsmall\norm p.
\]
For the second part notice that:
\[
\lilnorm{zpz^{-1}}=\lilnorm{a_{z}n_{z}pn_{z}^{-1}a_{z}^{-1}}\leq\exd{\lilnorm{a_{z}n_{z}a_{z}^{-1}}}{\prec1}\exd{\lilnorm{a_{z}pa_{z}^{-1}}}{\prec\norm p}\exd{\lilnorm{a_{z}n_{z}^{-1}a_{z}^{-1}}}{\prec1}\prec\norm p,
\]
 and 
\[
\lilnorm{\trans zp\mtrans z}=\lilnorm{\trans{n_{z}}a_{z}pa_{z}^{-1}\mtrans{n_{z}}}\leq\exd{\lilnorm{\trans{n_{z}}}}{\prec1}\exd{\lilnorm{a_{z}pa_{z}^{-1}}}{\prec\norm p}\exd{\lilnorm{\mtrans{n_{z}}}}{\prec1}\porsmall\norm p.\tag*{\qedhere}
\]
\end{proof}
The following fact indicates the relation between the norms of $z=a_{z}n_{z}$
and its columns, to the entries of $a_{z}$ and the covering radius
of $\lat_{z}$. 
\begin{fact}
\label{fact: diagonal entries and covering radius}Let $[z_{1}|\cdots|z_{m}]=z=a_{z}n_{z}$
in $\symfund m$. 

\begin{enumerate}
\item \label{enu: z_j almost a_j}For $j=1,\ldots,m$, $\norm{z_{j}}\asymp a_{j}$
.
\item \label{enu:. covering radius almost a_m}$\rad\brac{\lat_{z}}\asymp a_{m}\asymp\norm z$. 
\end{enumerate}
\end{fact}

\begin{notation*}
Set Let $E_{j}:=\sp_{\RR}\left\{ e_{1},\ldots,e_{j}\right\} $, where
$\left\{ e_{1},\ldots,e_{m}\right\} $ is the standard basis to $\RR^{m}$. 
\end{notation*}
\begin{proof}
According to Corollary \ref{cor: conjugation by z shrinks} and Lemma
\ref{lem: BLC. facts about z in RS domain},
\[
a_{i}=\dist{z_{i},E_{i-1}}\leq\norm{z_{i}}=\norm{a_{z}n_{z}e_{i}}=\lilnorm{a_{z}n_{z}a_{z}^{-1}a_{z}e_{i}}\leq\exd{\lilnorm{a_{z}n_{z}a_{z}^{-1}}}{\prec1}\exd{\norm{a_{z}e_{i}}}{a_{i}}\porsmall a_{i},
\]
which proves the first part. As for the second part, we have on the
one hand that (by Lemma \ref{lem: BLC. facts about z in RS domain},
parts (\ref{enu: entries of n in =00005B-0.5,0.5=00005D}) and (\ref{enu: entries of a increasing}))
\[
\norm z=\norm{a_{z}n_{z}}\prec\norm{a_{z}}\asymp a_{m}
\]
and on the other hand that 
\[
a_{m}\asymp\norm{a_{z}}=\lilnorm{a_{z}n_{z}n_{z}^{-1}}\prec\norm{a_{z}n_{z}}=\norm z.
\]
The fact that $a_{m}\asymp\rad_{z}$ is proved in \cite[Theorem 7.9]{GM02}.
\end{proof}
\textcolor{green}{}
\begin{lem}
\textcolor{green}{\label{lem: z' is in OzO}}Let $\left(a^{\prime},n^{\prime}\right)\in\nbhd{\e}{A_{m}\times N_{m}}\left(a,n\right)\nbhd{\e}{A_{m}\times N_{m}}$.
If $z=an,\,z^{\prime}=a^{\prime}n^{\prime}$ and $z\in\symfund m$,
then $\lilnorm{z^{\prime}z^{-1}},\lilnorm{z^{-1}z^{\prime}}\leq1+C_{1}\e$
for some $C_{1}>0$. 
\end{lem}

\begin{proof}
Clearly $\left(a^{\prime},n^{\prime}\right)\in\nbhd{\e}{A_{m}\times N_{m}}\left(a,n\right)\nbhd{\e}{A_{m}\times N_{m}}$
is equivalent\footnote{See fourth property of $\nbhd{\e}G$ in Section \ref{sec: effective Iwasawa and GI decompositions}}
to $z^{\prime}\in\nbhd{\e}{A_{m}}a\,\nbhd{\e}{A_{m}}\nbhd{\e}{N_{m}}n\,\nbhd{\e}{N_{m}}$.
Using the fact that $\nbhd{\e}{P_{m}}$ is equivalent to $\nbhd{\e}{A_{m}}\nbhd{\e}{N_{m}}$
and Corollary \ref{cor: conjugation by z shrinks} we obtain,
\[
\nbhd{\e}{A_{m}}a\,\nbhd{\e}{A_{m}}\nbhd{\e}{N_{m}}n\,\nbhd{\e}{N_{m}}=an\left(n^{-1}\nbhd{2\e}{A_{m}}\nbhd{\e}{N_{m}}n\right)\nbhd{\e}{N_{m}}\subseteq an\cdot n^{-1}\nbhd{c_{1}\e}{P_{m}}n\cdot\nbhd{\e}{N_{m}}\subseteq an\nbhd{c_{2}\e}{P_{m}}\nbhd{\e}{N_{m}}\subseteq z\nbhd{c_{3}\e}{P_{m}}.
\]
 Again using Corollary \ref{cor: conjugation by z shrinks}, one
also obtains
\[
z\nbhd{c_{3}\e}{P_{m}}=\left(z\nbhd{c_{3}\e}{P_{m}}z^{-1}\right)z\subseteq\nbhd{c_{4}\e}{P_{m}}\,z.
\]
Finally, fix $C_{1}>0$ such that 
\[
\nbhd{c_{4}\e}{P_{m}}\subseteq\left\{ p\in P_{m}:\norm p\leq1+C_{1}\e\right\} .\tag*{\qedhere}
\]
\end{proof}
The following lemma is the technical core of the proof of Proposition
\ref{prop: Hexagons are BLC}.
\begin{lem}
\label{lem: bounds needed to prove BLC of hexagons}Suppose $z,z^{\prime}\in\symfund m$
and that $\roundo\left(z^{\prime}\right)\in\nbhd{\e}{}\roundo\left(z\right)\nbhd{\e}{}$.
Let $v\in\mathbb{Z}^{m}$ and write $\lm=zv,\lm^{\prime}=z^{\prime}v$.
Then the following hold:

\begin{enumerate}
\item $\lilnorm{z^{\transpose}\lm}\porsmall\left\Vert \lm\right\Vert ^{2}$;
\item $\lilnorm{\lm^{\prime}}\leq\left(1+C_{1}\e\right)\left\Vert \lm\right\Vert $
for the constant $C_{1}>0$ from Lemma \ref{lem: z' is in OzO};
\item $\lilnorm{z^{\transpose}\lm-z^{\prime\transpose}\lm^{\prime}}\porsmall\e\left\Vert \lm\right\Vert ^{2}$.
\end{enumerate}
\end{lem}

\begin{proof}
For the first part, recall that $\linmap_{z}^{-1}\left(x\right)=zx$
and then 
\[
\lilnorm{\mtrans{\linmap_{z}}\left(\lm\right)}=\lilnorm{\trans z\lm}=\lilnorm{\trans{n_{z}}a_{z}\lm}\leq\exd{\lilnorm{\trans{n_{z}}}}{\prec1}\left\Vert a_{z}\lm\right\Vert \prec\left\Vert a_{z}\lm\right\Vert .
\]
Next, let $j\in\left\{ 1,\ldots,m\right\} $ such that $\lm\in E_{j}\backslash E_{j-1}$.
By parts (\ref{enu: norm in E_j}) and (\ref{enu: norm out of E_(j-1)})
respectively of Lemma \ref{lem: BLC. facts about z in RS domain}:
\[
\left\Vert a\lm\right\Vert \prec a_{j}\left\Vert \lm\right\Vert \leq\left\Vert \lm\right\Vert ^{2}.
\]
 All in all, $\lilnorm{\mtrans{\linmap_{z}}\left(\lm\right)}\porsmall\left\Vert \lm\right\Vert ^{2}$.

For the second part, use Lemma \ref{lem: z' is in OzO}:
\[
\lilnorm{\lm^{\prime}}=\lilnorm{z^{\prime}v}=\lilnorm{z^{\prime}z^{-1}zv}\leq\lilnorm{z^{\prime}z^{-1}}\left\Vert zv\right\Vert \leq\left(1+C_{1}\varepsilon\right)\norm{\lm}.
\]
 For the third part, it is clear that
\[
\lilnorm{\trans z\lm-z^{\prime\transpose}\lm^{\prime}}\leq\lilnorm{\trans z\left(\lm-\lm^{\prime}\right)}+\lilnorm{\brac{\trans z-z^{\prime\transpose}}\lm^{\prime}}
\]
and we shall bound each of these two summands. The first one is bounded
by
\[
\lilnorm{\trans z\brac{\lm-\lm^{\prime}}}=\lilnorm{\trans z\brac{z-z^{\prime}}v}=\lilnorm{\trans z\exd{\brac{I-z^{\prime}z^{-1}}}{p\in P_{m}}zv}=\lilnorm{\brac{\trans zpz^{-\transpose}}\trans z\exd{zv}{\lm}}\leq\lilnorm{\trans zp\mtrans z}\lilnorm{\trans z\lm}
\]
where by Corollary \ref{cor: conjugation by z shrinks}, Lemma \ref{lem: z' is in OzO},
and the first part of the current Lemma, 
\[
\prec\norm p\lilnorm{\trans z\lm}\porsmall\e\norm{\lm}^{2}.
\]
The second summand is bounded by

\[
\lilnorm{\brac{\trans z-z^{\prime\transpose}}\lm^{\prime}}=\lilnorm{\brac{\trans zz^{\prime-\transpose}-I}z^{\prime\transpose}\lm^{\prime}}=\lilnorm{\brac{\trans{\brac{z^{\prime-1}z}}-I}z^{\prime\transpose}\lm^{\prime}}\leq\lilnorm{\trans{\brac{z^{\prime-1}z}}-I}\cdot\lilnorm{z^{\prime\transpose}\lm^{\prime}}.
\]
By Lemma \ref{lem: z' is in OzO} and the first part of the current
Lemma, the above is $\prec C_{1}\e\cdot\e\lilnorm{\lm^{\prime}}^{2}$,
and by the second part of the current lemma the latter is 
\[
\leq C_{1}\e\cdot\e\cdot\left(\brac{1+C_{1}\e}\norm{\lm}\right)^{2}\prec\e\norm{\lm}^{2}.\tag*{\qedhere}
\]
\end{proof}
Towards proving Proposition \ref{prop: Hexagons are BLC}, stating
that the family $\fdomhex_{\symfund m}$ is BLC, we prove that this
family satisfies the fourth property of BLC.
\begin{lem}
\label{lem: Boundedness of L_z(Dir(z))}The family $\fdomhex$ is
bounded uniformly from above. Namely, there exists $\radmax>0$ that
depends only on $m$ such that $\hex z=\linmap_{z}\left(\text{Dir}\left(z\right)\right)$
is contained in $\ball{\radmax}{}$ for every $z\in\symfund m$. 
\end{lem}

We introduce a notation, to be used in the proofs of Lemma \ref{lem: Boundedness of L_z(Dir(z))}
and Proposition \ref{prop: Hexagons are BLC}. For $\lm\in\lat_{z}$,
write $\striplane{\lm}$ for the strip
\[
\striplane{\lm}:=\left\{ x:\left|\left\langle x,\lm\right\rangle \right|\leq\left\Vert \lm\right\Vert ^{2}/2\right\} .
\]
It is easy to check that it consists of all the vectors in $\mathbb{R}^{m}$
which are closer to the origin than to $\pm\lm$. As a result, 
\begin{equation}
\dirdom{\lat_{z}}=\bigcap_{0\neq\lm\in\lat_{z}}\striplane{\lm}.\label{eq: Dir(z) is intersection of strips}
\end{equation}

\begin{proof}
According to (\ref{eq: Dir(z) is intersection of strips}) and definition
of $\striplane{\lm}$, an element $x\in\text{Dir}\left(\lat_{z}\right)$
satisfies that $\left|\left\langle \lm,x\right\rangle \right|\leq\left\Vert \lm\right\Vert ^{2}/2$
for every $0\neq\lm\in\lat_{z}$. In particular, this holds for $\lm\in\left\{ z_{1},\ldots,z_{m}\right\} \subset\lat_{z}$
(the columns of $z$). Recall that by Remark \ref{rem: L_z}, $x=z\linmap_{z}\left(x\right)$.
The inequality $\left|\dbrac{z_{j},x}\right|\leq\left\Vert z_{j}\right\Vert ^{2}/2$
therefore translates into the inequality $\vbrac{\dbrac{z_{j}/\left\Vert z_{j}\right\Vert ^{2},z\linmap_{z}\left(x\right)}}\leq1/2$,
i.e. 
\[
\vbrac{\dbrac{\trans zz_{j}/\left\Vert z_{j}\right\Vert ^{2},\linmap_{z}\left(x\right)}}\leq1/2
\]
or
\[
\vbrac{\exd{\left\Vert z_{j}\right\Vert ^{-2}\trans{z_{j}}z}{\mbox{row}}\cdot\exd{\linmap_{z}\left(x\right)}{\mbox{column}}}\leq1/2.
\]
Considering all $m$ inequalities, we obtain 
\[
\left|\left[\begin{array}{ccc}
- & \left\Vert z_{1}\right\Vert ^{-2}\trans{z_{1}} & -\\
 & \vdots\\
- & \left\Vert z_{m}\right\Vert ^{-2}\trans{z_{m}} & -
\end{array}\right]\cdot z\cdot\exd{\linmap_{z}\left(x\right)}{\mbox{column}}\right|\leq\trans{\left(1/2,\ldots,1/2\right)}
\]
(where one should understand $\leq$ and $\left|\cdot\right|$ as
referring to the components), namely
\[
\left|\lildiag{\left\Vert z_{j}\right\Vert ^{2}}_{j=1}^{m}\cdot\trans zz\cdot\linmap_{z}\left(x\right)\right|\leq\trans{\left(1/2,\ldots,1/2\right)}.
\]
Let $g:=\lildiag{\left\Vert z_{j}\right\Vert ^{2}}_{j=1}^{m}\cdot\trans zz$;
based on the last inequality, in order to show that $\lilnorm{\linmap_{z}\left(x\right)}$
is bounded by some constant $\radmax=\radmax\left(m\right)$, it is
sufficient to prove that $\lilnorm{g^{-1}}\porsmall1$ where the implied
constant depends only on $m$. Indeed, 
\[
\lilnorm{g^{-1}}=\lilnorm{z^{-1}\mtrans z\lildiag{\left\Vert z_{j}\right\Vert ^{2}}_{j=1}^{m}}\overset{\substack{\mbox{\mbox{Fact \ref{fact: diagonal entries and covering radius}}}\\
\mbox{part (\ref{enu: z_j almost a_j})}
}
}{\porsmall}\lilnorm{z^{-1}\mtrans z\lildiag{a_{j}^{2}}_{j=1}^{m}}=\lilnorm{z^{-1}z^{-\transpose}a_{z}^{2}}
\]
\[
=\lilnorm{n_{z}^{-1}a_{z}^{-2}\mtrans{n_{z}}a_{z}^{2}}=\lilnorm{n_{z}^{-1}\trans{\left(a_{z}^{2}n_{z}^{-1}a_{z}^{-2}\right)}}\leq
\]
\[
\leq\exd{\lilnorm{n_{z}^{-1}}}{\porsmall1}\cdot\exd{\lilnorm{a_{z}^{2}n_{z}^{-1}a_{z}^{-2}}}{\porsmall\lilnorm{n_{z}^{-1}}\porsmall1}\porsmall1
\]
where the estimation $\lilnorm{a_{z}^{2}n_{z}^{-1}a_{z}^{-2}}\porsmall\lilnorm{n_{z}^{-1}}$
is also due to Corollary \ref{cor: conjugation by z shrinks}.
\end{proof}
We are now ready to prove Proposition \ref{prop: Hexagons are BLC}.
\begin{proof}[proof of Proposition \ref{prop: Hexagons are BLC}]
We begin by verifying property \textbf{BLC} \textbf{(I)}. According
to (\ref{eq: Dir(z) is intersection of strips}), it is sufficient
to prove that this property holds for each strip $\striplane{\lm}$
separately, namely that 
\[
\linmap_{z}\left(\striplane{\lm}\right)+\ball{\e}{}\subseteq\left(1+C\e\right)\linmap_{z}\left(\striplane{\lm}\right).
\]
Since (Remark \ref{rem: L_z})
\[
\linmap_{z}\left(\striplane{\lm}\right)=\left\{ y:\left|\left\langle \mtrans{\linmap_{z}}\left(\lm\right),y\right\rangle \right|\leq\left\Vert \lm\right\Vert ^{2}/2\right\} =\left\{ y:\left|\left\langle \trans z\lm,y\right\rangle \right|\leq\left\Vert \lm\right\Vert ^{2}/2\right\} ,
\]
and 
\[
\linmap_{z}\left(\striplane{\lm}\right)+\ball{\e}{}\subseteq\left\{ x:\left|\left\langle x,\mtrans{\linmap_{z}}\left(\lm\right)\right\rangle \right|\leq\left\Vert \lm\right\Vert ^{2}/2+\left\Vert \mtrans{\linmap_{z}}\left(\lm\right)\right\Vert \cdot\e\right\} ,
\]
the desired inclusion is equivalent to
\[
\left\Vert \lm\right\Vert ^{2}/2+\e\left\Vert \mtrans{\linmap_{z}}\left(\lm\right)\right\Vert \leq\left(1+C\e\right)\left\Vert \lm\right\Vert ^{2}/2.
\]
This indeed holds, since by part 1 of Lemma \ref{lem: bounds needed to prove BLC of hexagons},
$\lilnorm{\mtrans{\linmap_{z}}\left(v\right)}=\lilnorm{\trans zv}\porsmall\left\Vert v\right\Vert ^{2}$.

We turn to prove property \textbf{BLC} \textbf{(II)}. As with property
\textbf{BLC} \textbf{(I)}, it is sufficient to verify it for each
strip $\striplane{\lm}$ separately. Assume that $r_{P^{\dprime}}\left(z^{\prime}\right)\in\nbhd{\e}{}r_{P^{\dprime}}\left(z\right)\nbhd{\e}{}$.
Let $y\in\dirdom{z^{\prime}}\subset\RR^{m}$, namely 
\[
\left|\left\langle z^{\prime\transpose}\lm^{\prime},y\right\rangle \right|\leq\lilnorm{\lm^{\prime}}^{2}/2
\]
for every $0\neq\lm^{\prime}\in\lat_{z^{\prime}}.$ We need to prove
that $y\in\left(1+C\e\right)\linmap_{z}\left(\striplane{\lm}\right)$,
for all $0\neq\lm\in\lat_{z}$, namely that 
\[
\left|\left\langle \trans z\lm,y\right\rangle \right|\leq\left(1+C\e\right)\left\Vert \lm\right\Vert ^{2}/2.
\]
Now, 
\[
\left|\left\langle \trans z\lm,y\right\rangle \right|\leq\left|\left\langle z^{\prime\transpose}\lm^{\prime},y\right\rangle \right|+\left|\left\langle \trans z\lm-z^{\prime\transpose}\lm^{\prime},y\right\rangle \right|\leq\lilnorm{\lm^{\prime}}^{2}/2+\norm y\cdot\lilnorm{\trans z\lm-z^{\prime\transpose}\lm^{\prime}}.
\]
According to Lemma \ref{lem: Boundedness of L_z(Dir(z))}, 
\[
\leq\lilnorm{\lm^{\prime}}^{2}/2+\radmax\cdot\lilnorm{\trans z\lm-z^{\prime\transpose}\lm^{\prime}}=\brac{\lilnorm{\lm^{\prime}}^{2}/\left\Vert \lm\right\Vert ^{2}+2\radmax\norm{\trans z\lm-z^{\prime\transpose}\lm^{\prime}}/\left\Vert \lm\right\Vert ^{2}}\cdot\left\Vert \lm\right\Vert ^{2}/2
\]
and according to parts 2 and 3 of Lemma \ref{lem: bounds needed to prove BLC of hexagons},
\[
=\brac{\,\exd{\lilnorm{\lm^{\prime}}^{2}/\left\Vert \lm\right\Vert ^{2}}{\leq1+C_{1}\e}+2\radmax\,\exd{\norm{\trans z\lm-z^{\prime\transpose}\lm^{\prime}}/\left\Vert \lm\right\Vert ^{2}}{\prec\e}\,}\left\Vert \lm\right\Vert ^{2}/2\leq\brac{1+C\e}\cdot\left\Vert \lm\right\Vert ^{2}/2.
\]
The \textbf{BLC} \textbf{(III)} is trivial since $\hex z=\linmap_{z}\left(\dirdom z\right)$
are fundamental domains for $\ZZ^{m}$ in $\RR^{m}$, hence their
volume is exactly $1$. Property \textbf{BLC} \textbf{(IV)} for the
family $\fdomhex_{\roundo\brac{\symfund m}}$ is the content of Lemma
\ref{lem: Boundedness of L_z(Dir(z))}.
\end{proof}
The following is the main result of this section. 
\begin{prop}
\textcolor{green}{\label{prop: Intersections of hexagons with ellipses are BLC}}For
every $0<\alpha\leq1$ the family $\fdomellipse{\a}{\roundo\left(\symfund m\right)}$
 defined in Formula (\ref{eq: family of elipses int hexagons}) is
\textbf{BLC} w.r.t. $\nbhd{\e}{}$ as in Proposition \ref{prop: Hexagons are BLC}.
\end{prop}

\begin{proof}
Set $\rad_{z}:=\rad\left(\lat_{z}\right)$, and similarly for $z^{\prime}$.
To prove the first property, it is sufficient to show that for some
$C>0$,
\[
\text{\ensuremath{\ball{\a\rad_{z}}{}}}+\linmap_{z}^{-1}\left(\ball{\e}{}\right)\subseteq\left(1+C\e\right)\text{\ensuremath{\ball{\a\rad_{z}}{}}}.
\]
By Fact \ref{fact: diagonal entries and covering radius}, there is
a constant $C>0$ such that:
\[
\linmap_{z}^{-1}\left(\ball{\e}{}\right)=z\left(\ball{\e}{}\right)\subseteq\ball{\left\Vert z\right\Vert \e}{}\subseteq\ball{C\left(\a\rad_{z}\right)\e}{}.
\]
As a result, 
\[
\text{\ensuremath{\ball{\a\rad_{z}}{}}}+\linmap_{z}^{-1}\left(\ball{\e}{}\right)\subseteq\text{\ensuremath{\ball{\a\rad_{z}}{}}}+\ball{C\a\rad_{z}\e}{}\subseteq\ball{\a\rad_{z}\left(1+C\e\right)}{}=\left(1+C\e\right)\ball{\rad_{z}}{}.
\]

As for the second property, since it is maintained under intersections,
it is sufficient to prove that
\[
\linmap_{z^{\prime}}\brac{\ball{\a\rad_{z^{\prime}}}{}}\subseteq\left(1+C\e\right)\linmap_{z}\brac{\ball{\a\rad_{z}}{}}.
\]
Or in other words, 
\[
\linmap_{z}^{-1}\linmap_{z^{\prime}}\brac{\ball{\a\rad_{z^{\prime}}}{}}\subseteq\left(1+C\e\right)\ball{\a\rad_{z}}{}.
\]

To this end, we first claim that there exists $C_{2}>0$ such that
\begin{equation}
\rad_{z^{\prime}}\leq\left(1+C_{1}\e\right)\left(1+C_{2}\e\right)\rad_{z};\label{eq: covering radi for z and z'}
\end{equation}
indeed, by property \textbf{BLC} \textbf{(II) }for $\mathcal{\fdomhex}_{\roundo\left(\symfund m\right)}$
(Proposition \ref{prop: Hexagons are BLC}), we have that 
\[
\linmap_{z^{\prime}}\left(\dirdom{\lat_{z^{\prime}}}\right)\subseteq\left(1+C_{2}\e\right)\cdot\linmap_{z}\left(\dirdom{\lat_{z}}\right)
\]
and therefore 
\begin{eqnarray*}
\dirdom{\lat_{z^{\prime}}} & \subseteq & \left(1+C_{2}\e\right)\cdot\linmap_{z^{\prime}}^{-1}\linmap_{z}\left(\dirdom{\lat_{z}}\right)\\
\mbox{(Lem. \ref{rem: L_z})} & \subseteq & \left(1+C_{2}\e\right)\cdot z^{\prime}z^{-1}\cdot\dirdom{\lat_{z}}\\
 & \subseteq & \left(1+C_{2}\e\right)\cdot\lilnorm{z^{\prime}z^{-1}}\,\dirdom{\lat_{z}}\\
\mbox{(Lem. \ref{lem: z' is in OzO})} & \subseteq & \left(1+C_{2}\e\right)\cdot\left(1+C_{1}\right)\dirdom{\lat_{z}}.
\end{eqnarray*}
Now, 
\[
\linmap_{z}^{-1}\linmap_{z^{\prime}}\brac{\ball{\a\rad_{z^{\prime}}}{}}\subseteq\lilnorm{zz^{\prime-1}}\cdot\ball{\a\rad_{z^{\prime}}}{\overset{\mbox{Rmk. \ref{lem: z' is in OzO}}}{\subseteq}}\left(1+C_{1}\e\right)\cdot\ball{\a\rad_{z^{\prime}}}{}
\]
\[
\overset{\mbox{eq. (\ref{eq: covering radi for z and z'})}}{\subseteq}\left(1+C_{1}\e\right)^{2}\left(1+C_{2}\e\right)\cdot\ball{\a\rad_{z}}{}
\]
which establishes that $\linmap_{z}^{-1}\linmap_{z^{\prime}}\brac{\ball{\a\rad_{z^{\prime}}}{}}\subseteq\brac{1+C\e}\ball{\a\rad_{z}}{}$
and completes the proof of the second property. 

Property \textbf{BLC} \textbf{(IV) }is a direct consequence of Lemma
\ref{lem: Boundedness of L_z(Dir(z))}, and so we turn to prove the
third property. First, we claim that for $z=a_{z}n_{z}\in\symfund m$,
the vectors 
\[
\pm\mbox{a}_{j}:=\brac{a_{z}/2}e_{j}=\brac{a_{j}/2}e_{j}
\]
lie in $\dirdom{\lat_{z}}$. Indeed, suppose otherwise that there
exists $\lm\in\lat_{z}$ such that $\norm{\mbox{a}_{j}+\lm}<\norm{\mbox{a}_{j}}$.
Then $\lm$ cannot lie inside $V_{j-1}=\sp\left\{ z_{1},\ldots z_{j-1}\right\} $,
because if it did then it would have been orthogonal to $\mbox{a}_{j}$,
which implies 
\[
\norm{\mbox{a}_{j}}^{2}+\norm{\lm}^{2}=\norm{\mbox{a}_{j}+\lm}^{2}\overset{_{\mbox{assumption}}}{<}\norm{\mbox{a}_{j}}^{2},
\]
a contradiction. Hence $\lm\notin V_{j-1}$, implying $\lm=\lm_{j-1}+\lm_{j-1}^{\perp}$
with $0\neq\lm_{j-1}^{\perp}\in V_{j-1}^{\perp}$. Now, 
\[
a_{j}=\dist{z_{j},V_{j-1}}\overset{_{\left(\lm\notin V_{j-1}\right)}}{\leq}\dist{\lm,V_{j-1}}=\lilnorm{\lm_{j-1}^{\perp}}\leq\norm{\lm}\leq
\]
\[
\leq\norm{\lm+\mbox{a}_{j}}+\norm{\mbox{a}_{j}}\overset{_{\mbox{assumption}}}{<}2\norm{\mbox{a}_{j}}=a_{j}.
\]
This is clearly a contradiction, establishing that the vectors $\pm\mbox{a}_{j}$
indeed lie inside $\dirdom{\lat_{z}}$. 

Let $c>0$ such that $\norm{c\mbox{a}_{j}}=\frac{1}{2}c\mbox{a}_{j}\leq\a\rad_{z}$
for every $j=1,\ldots,m$; such $c$ exists and is independent of
$z$ because $a_{1}\porsmall\cdots a_{m}\porsmall\rad_{z}$ (according
to Fact \ref{fact: diagonal entries and covering radius} and part
(\ref{enu: entries of a increasing}) of Lemma \ref{lem: BLC. facts about z in RS domain}).
We may assume that $c\leq1$ and therefore (since $\dirdom{\lat_{z}}$
is convex and contains the origin and the points $\mbox{a}_{j}$),
that the points $c\,\mbox{a}_{j}$ are also contained in $\dirdom{\lat_{z}}$.
They are obviously contained in $\ball{\a\rad_{z}}{}$ , hence by
convexity 
\[
\left[-c,c\right]\mbox{a}_{1}\times\cdots\times\left[-c,c\right]\mbox{a}_{m}=c^{m}\cdot\prod_{j=1}^{m}\left[-\frac{a_{j}}{2},\frac{a_{j}}{2}\right]\subseteq\dirdom{\lat_{z}}\cap\ball{\a\rad_{z}}{}.
\]
The above shape has volume $c^{m}\cdot\prod_{j=1}^{m}a_{i}=c^{m}\cdot\det\left(z\right)$;
its image under $\linmap_{z}=z^{-1}$ has therefore volume $c^{m}$.
It follows that the volume of $\linmap_{z}\left(\dirdom{\lat_{z}}\cap\ball{\a\rad_{z}}{}\right)$
is bounded from below by $c^{m}$, which does not depend on $z$. 
\end{proof}

\section{\label{sec: Concluding the proofs}Concluding the proofs of the
theorems}

We will prove Proposition \ref{prop: Counting with Hexagons (A counting)}
in a slightly greater generality, when the lattice $\Lat<\sl n\left(\RR\right)$
is general, and when the sets $\funddom_{T}^{\Svec}$ are fibered
over a family $\fam{\fdomN}{\roundo\left(\symfund{n-1}\right)}=\left\{ \domN\left(a^{\dprime}n^{\dprime}\right)\right\} _{\left(a^{\dprime},n^{\dprime}\right)\in\roundo\left(\symfund{n-1}\right)}$
that is not necessarily $\fdomhex^{\a}$. Indeed, we consider: 
\[
\funddom_{T}^{\Svec}\brac{\latset}=\bigcup_{q\in\trunc{\parby{\wc}{\latset}}{\Svec}}q\cdot A_{T}^{\prime}\parby{N^{\prime}}{\domN\brac{\zgt q}}.
\]
Proposition \ref{prop: Counting with Hexagons (A counting)} is a
consequence of Proposition \ref{prop: Hexagons are BLC}, combined
with the following: 
\begin{thm}
\label{prop: counting in fibered sets over F_m}Let $\funddom_{T}^{\Svec}\brac{\latset}$
be as above, where $\latset\subseteq\latspace{n-1,n}$ is a BCS, and
$\fam{\fdomN}{\roundo\left(\symfund{n-1}\right)}$ is a BLC family
of subsets of $\RR^{n}$. Set $\lm_{n}=\frac{n^{2}}{2\left(n^{2}-1\right)}$.
Let $\Lat<\sl n\left(\RR\right)$ be a lattice and $\errexp=\errexp\left(\Lat\right)$. 
\begin{enumerate}
\item For $0<\e<\errexp$, $\Svec=\left(S_{1},\dots,S_{n-2}\right)$, $\sumS=\sum_{i=1}^{n-2}S_{i}$
and every $T\geq\frac{\sumS}{n\lm_{n}\errexp}+O_{\fdomN}\left(1\right)$,
\[
\#\left(\funddom_{T}^{\Svec}\brac{\latset}\cap\Lat\right)=\frac{\mu(\,\funddom_{T}^{\Svec}\brac{\latset})}{\mu\brac{G/\Lat}}\cdot\frac{e^{nT}}{n}+O_{\latset,\e}(e^{\sumS/\lm_{n}}e^{nT\left(1-\errexp+\e\right)}).
\]
\item For $0<\e<\errexp$, $\delta\in\left(0,\errexp-\e\right)$ , $\Svec\left(T\right)=\left(S_{1}\left(T\right),\dots,S_{n-2}\left(T\right)\right)$
such that $\sumS\left(T\right)=\sum S_{i}\left(T\right)<n\dl\lm_{n}T+O_{\latset}(1)$
and every $T\geq O_{\fdomN}\left(1\right)$,
\[
\#\left(\funddom_{T}^{\Svec\brac T}\brac{\latset}\cap\Lat\right)=\frac{\mu(\,\funddom_{T}^{\Svec\brac T}\brac{\latset})}{\mu\brac{G/\Lat})}+O_{\latset,\e}(e^{nT\left(1-\errexp+\dl+\e\right)}).
\]
\end{enumerate}
\end{thm}

\begin{proof}[Proof of Theorem \ref{prop: counting in fibered sets over F_m}]
\textbf{Part 1}. ~ Consider the image of $\funddom_{T}^{\Svec}\brac{\latset}$
under $\roundo$, which is of the form 
\[
\roundo\left(\funddom_{T}^{\Svec}\brac{\latset}\right)=\bigcup_{\substack{\left(k,\left(a^{\dprime},n^{\dprime}\right),a^{\prime}\right)\in\\
\roundo\left(\trunc{\parby{\wc}{\latset}}{\Svec}\right)\times A_{T}^{\prime}
}
}\left(k,a^{\dprime},n^{\dprime},a^{\prime}\right)\times N_{\domN\left(a^{\dprime},n^{\dprime}\right)}^{\prime}
\]
(see Subsection \ref{subsec: Plan of proof}). We claim that it is
a well rounded family with increasing parameter $T$ in the group
$K\times A^{\dprime}\times N^{\dprime}\times A^{\prime}\times N^{\prime}$.
First, since the family $\fam{\fdomN}{r\left(\symfund{n-1}\right)}$
is assumed to be BLC, and the projection of $\roundo\left(\trunc{\parby{\wc}{\latset}}{\Svec}\right)$
to $A^{\dprime}\times N^{\dprime}$ is contained in $\roundo\left(\symfund{n-1}\right)$,
then the restriction of $\fdomN$ to this projection is also BLC.
Since $\fdomN$ is independent of the $k,a^{\prime}$ components,
we may extend the set over which it is parameterized to include these
components (\cite[Cor. 5.3]{HK_WellRoundedness}), hence the family
$\fam{\fdomN}{\roundo\left(\trunc{\parby{\wc}{\latset}}{\Svec}\right)\times A_{T}^{\prime}}$
is BLC.

As for the base set, $\latset$ is a BCS by assumption, and so $\wc_{\latset}$
is also a BCS, by Proposition \ref{prop: spread models that we need}.
Thus, $\roundo\left(\trunc{\parby{\wc}{\latset}}{\Svec}\right)\subset K\times A^{\dprime}\times N^{\dprime}$
is LWR according to Lemma \ref{lem: pushforward of Fm is LWR AxN},
with parameters that do not depend on $\Svec$. Since $A_{T}^{\prime}$
is LWR (Proposition \ref{prop: A cubes are well rounded}), then Remark
\ref{rem: product of LWR is LWR} implies that $\roundo\left(\trunc{\parby{\wc}{\latset}}{\Svec}\right)\times A_{T}^{\prime}$
is LWR inside $K\times A^{\dprime}\times N^{\dprime}\times A^{\prime}$.
By Proposition \ref{prop: Fibered sets are LWR in direct product},
this implies that the family $\roundo\left(\funddom_{T}^{\Svec}\brac{\latset}\right)$
is LWR with Lipschitz constant that is $\asymp1$. 

Since by Corollary \ref{cor: roundo of GI decomposition} combined
with Lemma \ref{lem: err(a) in SLn}, $\roundo$ is an $f$-roundomorphism
with $f\brac{ka^{\prime}a_{\svec}^{\dprime}n^{\dprime}n^{\prime}}\porsmall e^{2\sums}$,
it follows from Proposition \ref{cor: roundo to a produt grp} that
$\funddom_{T}^{\Svec}\brac{\latset}\subset\sl n\left(\RR\right)$
is LWR with $C\prec_{\latset}e^{2\sumS}$ and $T_{0}$ that is independent
of $\Svec$ and of the family $\fdomN$. The first part of the theorem
now follows from Theorem \ref{thm: GN Counting thm}; it is only left
to observe, for the error term, that $\mu\brac{\funddom_{T}^{\Svec}\brac{\latset}}\asymp_{\latset}e^{nT}$
by Assumption \ref{assum: assumption that 2rho(H'')<0}, and to verify
the lower bound  on $T$. The latter is obtained by substituting
the bound on the parameter $C$ into the  condition \ref{eq: def of T_1 in GN thm}
in Theorem \ref{thm: GN Counting thm}. Indeed, using the notation
of Theorem \ref{thm: GN Counting thm}, this condition  is equivalent
to $\errexp\left(\Lat\right)\ln\mu(\Gset_{T}))\geq\frac{\dim G}{1+\dim G}\ln C_{\Gset}+O\left(1\right)$.
Substituting $C_{\Gset}=C\prec_{\latset}e^{2\sumS}$ and $\mu\left(\Gset_{T}\right)=\mu(\funddom_{T}^{\Svec}\brac{\latset})\leq\mu(\funddom_{T}\brac{\latset})\asymp_{\latset}e^{nT}$,
the condition translates into 
\[
\errexp\left(\Gamma\right)\cdot nT\geq\,\frac{\dim G}{1+\dim G}\cdot2\sumS+O_{\latset}\left(1\right)=\sumS/\lm_{n}+O_{\latset}\left(1\right)
\]
i.e. to 
\[
T\geq\sumS/\left(n\errexp\left(\Gamma\right)\lm_{n}\right)+O_{\latset}\left(1\right).
\]

\textbf{Part 2}. ~ Let $\sumS=\sumS\left(T\right)>0$. In order for
the main term in part (1) to be of lower order than the main term,
we require the existence of a parameter $\gamma\in\left(0,1\right)$
for which
\[
\sumS/\lm_{n}+\left(1-\errexp\left(\gam\right)+\e\right)\cdot nT<\gamma\cdot nT.
\]
This is equivalent to 
\[
\sumS<\lm_{n}\cdot\left(\gamma+\errexp\left(\gam\right)-\e-1\right)nT.
\]
Hence, if we denote by $\delta$ the number $\gamma+\errexp\left(\Lat\right)-\e-1$,
we must require that $\delta>0$ and that $\gamma=\delta+\left(1+\e-\errexp\left(\Lat\right)\right)$
lies in $\left(0,1\right)$. If $0<\e<\errexp\left(\Lat\right)$,
then clearly $0<1+\e-\errexp\left(\Lat\right)<1$, so the condition
on $\delta$ becomes $\delta\in\left(0,\errexp\left(\Lat\right)-\e\right).$

The condition on $T$ in part (1) is equivalent to $\sumS\leq n\lm_{n}\errexp\left(\Lat\right)\cdot T+O_{\latset}\left(1\right)$,
i.e.
\[
\sumS\leq\min\left\{ n\lm_{n}\dl\cdot T+O_{\latset}\left(1\right),\:n\lm_{n}\errexp\left(\gam\right)\cdot T\right\} =n\lm_{n}\dl\cdot T+O_{\latset}\left(1\right)
\]
for $T$ large enough and $\delta\in\left(0,\errexp\left(\Lat\right)-\e\right)$. 
\end{proof}
\bibliographystyle{alpha}
\phantomsection\addcontentsline{toc}{section}{\refname}\bibliography{bib_for_gcd_n_variables}

\begin{thebibliography}{EMSS16}

\bibitem[AES16a]{AES_16B}
M.~Aka, M.~Einsiedler, and U.~Shapira.
\newblock Integer points on spheres and their orthogonal grids.
\newblock {\em Journal of the London Mathematical Society}, 93(2):143--158,
  2016.

\bibitem[AES16b]{AES_16A}
M.~Aka, M.~Einsiedler, and U.~Shapira.
\newblock Integer points on spheres and their orthogonal lattices.
\newblock {\em Inventiones mathematicae}, 206(2):379--396, 2016.

\bibitem[BM00]{Bekka_Mayer}
M.B. Bekka and M.~Mayer.
\newblock {\em Ergodic Theory and Topological Dynamics of Group Actions on
  Homogeneous Spaces}, volume 269.
\newblock Cambridge University Press, 2000.

\bibitem[DRS93]{DRS93}
W.~Duke, Z.~Rudnick, and P.~Sarnak.
\newblock Density of integer points in affine homogeneous varieties.
\newblock {\em Duke Mathematical Jurnal}, 71(1):143--179, 1993.

\bibitem[Duk03]{Duke_03}
W.~Duke.
\newblock Rational points on the sphere.
\newblock In {\em Number Theory and Modular Forms}, pages 235--239. Springer,
  2003.

\bibitem[Duk07]{Duke_07}
W.~Duke.
\newblock An introduction to the linnik problems.
\newblock In {\em Equidistribution in number theory, an introduction}, pages
  197--216. Springer, 2007.

\bibitem[EH99]{Erdos_Hall_99}
P.~Erd{\"o}s and R.R. Hall.
\newblock On the angular distribution of gaussian integers with fixed norm.
\newblock {\em Discrete Mathematics}, 200(1-3)(1-3):87--94, 1999.

\bibitem[EMSS16]{EMSS_16}
M.~Einsiedler, S.~Mozes, N.~Sha, and U.~Shapira.
\newblock Equidistribution of primitive rational points on expanding
  horospheres.
\newblock {\em Compositio Mathematica}, 152(4):667--692, 2016.

\bibitem[EMV13]{EMV_13}
J.~Ellenberg, P.~Michel, and A.~Venkatesh.
\newblock Linnik's ergodic method and the distribution of integer points on
  spheres.
\newblock {\em n "Automorphic representations and L-functions}, 22:119--185,
  2013.

\bibitem[ERW17]{ERW17}
M.~Einsiedler, R.~R{\"u}hr, and P.~Wirth.
\newblock Distribution of shapes of orthogonal lattices.
\newblock {\em Ergodic Theory and Dynamical Systems}, pages 1--77, 2017.

\bibitem[Gar14]{Volumes}
P.~Garrett.
\newblock Volume of $sl_n(\mathbb{Z})\setminus sl_n(\mathbb{R})$ and
  $sp_n(\mathbb{Z})\setminus sp_n(\mathbb{R})$.
\newblock Available in http://www.math.umn.edu/\textasciitilde
  garrett/m/v/volumes.pdf, April 20, 2014.

\bibitem[GM02]{GM02}
S.~Goldwasser and D.~Micciancio.
\newblock {\em Complexity of lattice problems: a cryptographic perspective},
  volume 671 of {\em The Springer International Series in Engineering and
  Computer Science}.
\newblock Springer US, 2002.

\bibitem[GN09]{GN_book}
A.~Gorodnik and A.~Nevo.
\newblock {\em The ergodic theory of lattice subgroups}, volume 172 of {\em
  Annals of Mathematics Studies}.
\newblock Princeton University Press, 2009.

\bibitem[GN12]{GN1}
A.~Gorodnik and A.~Nevo.
\newblock Counting lattice points.
\newblock {\em Journal f{\"u}r die reine und angewandte Mathematik},
  2012(663):127--176, 2012.

\bibitem[Goo83]{Good}
A.~Good.
\newblock On various means involving the {F}ourier coefficients of cusp forms.
\newblock {\em Mathematische Zeitschrift}, 183(1):95--129, 1983.

\bibitem[GOS10]{GOS_wavefront}
A.~Gorodnik, H.~Oh, and N.~Shah.
\newblock Strong wavefront lemma and counting lattice points in sectors.
\newblock {\em Israel Journal of Mathematics}, 176(1):419--444, 2010.

\bibitem[Gre93]{Grenier_93}
D.~Grenier.
\newblock On the shape of fundamental domains in {$GL(n,\mathbb{R})/O(n)$}.
\newblock {\em Pacific Journal of Mathematics}, 160(1):53--66, 1993.

\bibitem[HK20]{HK_WellRoundedness}
T.~Horesh and Y.~Karasik.
\newblock A practical guide to well roundedness.
\newblock {\em arXiv:2011.12204}, 2020.
\newblock arXiv preprint.

\bibitem[HN16]{HN16_Counting}
T.~Horesh and A.~Nevo.
\newblock Horospherical coordinates of lattice points in hyperbolic space:
  effective counting and equidistribution.
\newblock {\em arXiv:1612.08215}, 2016.
\newblock arXiv preprint.

\bibitem[J{\"u}s18]{Jus18}
D.~J{\"u}stel.
\newblock The {Z}ak transform on strongly proper {G}-spaces and its
  applications.
\newblock {\em Journal of the London Mathematical Society}, 97(1):47--76, 2018.

\bibitem[Kna02]{Knapp}
A.~W. Knapp.
\newblock {\em Lie Groups: Beyond an Introduction}.
\newblock Birkh{\"a}user Basel, 2002.

\bibitem[Li95]{Li95}
J-S Li.
\newblock The minimal decay of matrix coefficients for classical groups.
\newblock In {\em Harmonic analysis in China}, pages 146--169. Springer, 1995.

\bibitem[Lin68]{Linnik_68}
Y.~V. Linnik.
\newblock {\em Ergodic Properties of Algebraic Fields}, volume~45 of {\em
  Ergebnisse der Mathematik und ihrer Grenzgebiete}.
\newblock Springer-Verlag Berlin Heidelberg, 1968.

\bibitem[LZ96]{Li_Zhu_96}
J-S. Li and C-B. Zhu.
\newblock On the decay of matrix coefficients for exceptional groups.
\newblock {\em Mathematische Annalen}, 305(1):249--270, 1996.

\bibitem[Mar10]{Marklof_10}
J.~Marklof.
\newblock The asymptotic distribution of frobenius numbers.
\newblock {\em Inventiones mathematicae}, 181(1):179--207, 2010.

\bibitem[MMO14]{Margulis_Mohammadi_Oh}
G.~Margulis, A.~Mohammadi, and H.~Oh.
\newblock Closed geodesics and holonomies for kleinian manifolds.
\newblock {\em Geometric and Functional Analysis}, 24(5):1608--1636, 2014.

\bibitem[RR09]{R&R}
M.~Risager and Z.~Rudnick.
\newblock On the statistics of the minimal solution of a linear diophantine
  equation and uniform distribution of the real part of orbits in hyperbolic
  spaces.
\newblock {\em Contemporary Mathematics}, 484:187--194, 2009.

\bibitem[Sca90]{Scaramuzzi90}
R.~Scaramuzzi.
\newblock A notion of rank for unitary representations of general linear
  groups.
\newblock {\em Transactions of the American Mathematical Society},
  319(1):349--379, 1990.

\bibitem[Sch68]{Schmidt_68}
W.~M. Schmidt.
\newblock Asymptotic formulae for point lattices of bounded determinant and
  subspaces of bounded height.
\newblock {\em Duke Mathmatical journal}, 35:327--339, 1968.

\bibitem[Sch98]{Schmidt_98}
W.~M. Schmidt.
\newblock The distribution of sub-lattices of {$Z^m$}.
\newblock {\em Monatshefte f{\"u}r Mathematik}, 125:37--81, 1998.

\bibitem[Sch15]{Schmidt_15}
W.~M. Schmidt.
\newblock Integer matrices, sublattices of $\mathbb{Z}^{m}$, and {F}robenius
  numbers.
\newblock {\em Monatshefte f{\"u}r Mathematik}, 178(3):405--451, 2015.

\bibitem[Tru13]{Truelsen}
J.~L. Truelsen.
\newblock Effective equidistribution of the real part of orbits on hyperbolic
  surfaces.
\newblock In {\em Proceedings of the American Mathematical Society}, volume
  141(2), pages 505--514, 2013.

\end{thebibliography}

\end{document}